\documentclass{article} 

\usepackage[utf8]{inputenc}
\usepackage{mathtools}
\usepackage{amssymb}
\usepackage{mathrsfs}  
\usepackage{stmaryrd}  
\usepackage{yhmath}
\usepackage[left=3cm,right=3cm,top=3cm,bottom=3.5cm]{geometry}
\usepackage{float}  
\usepackage{graphicx}
\usepackage{hyperref}
\usepackage{enumerate}
\usepackage{comment}
\usepackage{cleveref}

\usepackage{tabularx}  
\usepackage{array}
\newcolumntype{C}{>{\centering\arraybackslash}X}
\usepackage{booktabs}

\usepackage[allcommands]{overarrows}

\hypersetup{
    colorlinks=true,
}

\usepackage{amsthm}
\newtheorem{lemma}{Lemma}[section]
\newtheorem{proposition}{Proposition}[section]

\newtheorem{cor}[lemma]{Corollary}
\newtheorem{theorem}{Theorem}
\newtheorem*{theorem*}{Theorem}

\theoremstyle{definition}
\newtheorem{hyp}{Assumption}
 
\newtheorem{conjecture}{Conjecture}[section]

\newtheorem{rem}{Remark}[section]
\newtheorem{definition}{Definition}[section]

\newcommand{\N}{\mathbb{N}}
\renewcommand{\c}{{\bf  c}}
\newcommand{\E}{\mathbb{E}}
\renewcommand{\P}{\mathbb{P}}
\newcommand{\R}{\mathbb{R}}

\newcommand{\Mc}{\mathcal{M}}
\newcommand{\M}{\mathbf{M}}

\newcommand{\la}[1]{\overleftarrow*{#1}}
\newcommand{\diff}{\mathrm{d}}
\newcommand{\vl}{g_{\lambda_\infty+\xi}}

\renewcommand{\epsilon}{\varepsilon}
\newcommand{\eps}{\epsilon}
\newcommand{\vep}{\epsilon}
\newcommand{\hep}{{\hat{\epsilon}}}

\newcommand{\indic}{\mathbf{1}}
\DeclarePairedDelimiter{\abs}{\vert}{\vert}
\DeclarePairedDelimiter{\Angle}{\langle}{\rangle}
\DeclarePairedDelimiter{\norm}{\lVert}{\rVert}
\renewcommand{\angle}{\Angle}

\DeclareMathOperator{\Card}{Card}
\DeclareMathOperator{\Leb}{Leb}

\newcommand{\Dom}{{\Omega}}
\renewcommand{\t}{\hat{t}}

\usepackage{xcolor}

\usepackage{authblk}

\newlength{\affilskip}
\title{Convergence of spatial branching processes to $\alpha$-stable
CSBPs: Genealogy of semipushed fronts}

\author[1]{Félix Foutel-Rodier}
\author[2]{Emmanuel Schertzer}
\author[3]{Julie Tourniaire}

\affil[1]{Université Paris Cité, CNRS, MAP5, F-75006 Paris, France \vspace{\affilskip}}
\affil[2]{Faculty of Mathematics, University of Vienna, Oskar-Morgenstern-Platz 1, 1090 Wien, Austria\vspace{\affilskip}}
\affil[3]{Université Marie et Louis Pasteur, CNRS, LmB (UMR 6623), F-25000 Besançon, France.}

\date{\today}

\begin{document}

\maketitle

\begin{abstract}
    We consider an inhomogeneous branching diffusion on an unbounded
    domain of $\R^d$ and propose a simple condition under which we
    expect the size process (i.e., the number of particles) and the
    genealogy of the system to converge to those of an
    $\alpha$-stable continuous-state branching process, with
    $\alpha\in(1,2)$. This condition can be seen as the spatial analogue
    of the classical assumption that the tail of the offspring
    distribution of a Galton--Watson process is regularly varying.

    We make a first step towards establishing this result by providing a
    set of sufficient conditions under which the branching diffusion,
    seen as a random marked metric measure space that captures both the
    positions and the genealogical structure of the population, converges to
    an $\alpha$-stable genealogy. These conditions are based on the
    convergence of the moments of the process, which can be efficiently
    computed via recursive formulas.

    We apply this framework to a one-dimensional branching Brownian motion with
    inhomogeneous branching rate and negative drift. This model was introduced by
    Tourniaire as a toy model to investigate the internal dynamics of fluctuating
    pushed fronts. By using our general set of conditions we prove convergence of 
    the genealogy of the process in the semipushed regime, which was
    conjectured to hold by Birzu, Hallatschek, and Korolev.
\end{abstract}

\section{Introduction}

\subsection{Context}

\paragraph{Branching processes at large scales.}
One of the achievements of the theory of branching processes is the
classification of possible universality classes of Galton--Watson models.
These universality classes drive the behaviour at large scales of many of
their aspects, including the fluctuation of the population size
\cite{lamperti1967limit, grimvall1974convergence}, the genealogy
at a given time \cite{fleischmann1977structure, yakymiv1981reduced,
    popovic2004}, and the entire branching structure
\cite{aldous1993continuum, le1998branching, LeGall2002}.
These possible limits are characterised in terms of a function $\psi$
called the \emph{branching mechanism}. When the population size has
finite mean, a branching mechanism is of the form
\begin{equation}\label{eq:branching-mechanism}
    \forall \theta \ge 0,\quad \psi(\theta) = a \theta + \frac{b \theta^2}{2} +
    \int_{(0,\infty)} \big( e^{-\theta z} - 1 + z\theta \big) \Lambda(\diff z),
\end{equation}
for some $a \in \R$, $b \ge 0$ and some measure $\Lambda$ on $(0,
    \infty)$ satisfying
\[
    \int_0^\infty (z \wedge z^2) \Lambda(\diff z) < \infty.
\]

In the present article, we will be particularly interested in the class
of critical $\alpha$-stable branching processes with $\alpha \in (1,2)$,
which correspond to the branching mechanism given by $\psi(\theta) = c
    \theta^\alpha$ for some $c > 0$. The importance of this class lies in the
fact that they correspond to the scaling limits of critical Galton--Watson
processes with heavy (power-law) tail, that is, when the offspring
distribution $X$ satisfies
\begin{equation}\label{eq:heavy-tail-dist}
    \P(X > z) \sim \frac{c}{z^{\alpha}},\qquad \text{as $z \to \infty$,}
\end{equation}
for some $c > 0$. The following result, which can be found in
\cite{zolotarev1957more}, makes this claim rigorous. See the forthcoming
Section~\ref{sec:csbps_intro} for the definition of a continuous-state
branching process.

\begin{proposition} \label{prop:GW}
    Let $(Z_t)_{t \ge 0}$ be a Galton--Watson process in continuous time
    with heavy-tailed offspring distribution \eqref{eq:heavy-tail-dist},
    for $\alpha \in (1, 2)$. Then, letting $\gamma = 1 / (\alpha-1)$,
    there exist $C, C' > 0$ such that 
    \begin{equation} \label{eq:survProbGW}
        \P( Z_{tN} > 0 ) \sim 
        \frac{C'}{(Nt)^{\gamma}},\qquad
        \Big( \frac{Z_{tN}}{N^\gamma} \mid Z_{tN} > 0\Big)
        \implies \bar{\Xi}_t,\qquad \text{as $N \to \infty$,}
    \end{equation}
    where $\bar{\Xi}_t$ follows the entrance law at time $t$ of a
    continuous-state branching process with branching mechanism
    $\psi(\theta) = C \theta^\alpha$. 
\end{proposition}

The existence of a heavy tail \eqref{eq:heavy-tail-dist} indicates
that the dynamics of the population are primarily influenced by rare and
exceptional reproductive events. During such an event, a particle
generates a very large offspring, leading in the limit to a jump in the
population size and to a corresponding branching point with infinite degree
in the genealogy. This stands in sharp contrast to situations when $X$
has finite variance ($\psi(\theta) = c \theta^2$). In this case, the
population evolves through the incremental contributions of each
individual, leading to Feller's diffusion in the limit ($a = 0$ and
$\Lambda = 0$ in \eqref{eq:branching-mechanism}) with a binary genealogy.

\paragraph{Emergence of multiple mergers in spatial models.}
Recent findings suggest that the universality class of $\alpha$-stable
branching processes extends beyond Galton--Watson models and that they
can emerge as the scaling limit of branching processes with a spatial
component. In particular, Feller's diffusion and its genealogy have been
shown to describe the scaling limits of branching processes with a wide
variety of spatial structures \cite{hering1971critical, hering1974limit,
    asmussen1983branching, Miermont2008, powell19, horton2020stochastic,
    harris2020stochastic, harris2022yaglom, gonzalez2022asymptotic, foutel22,
    schertzer2023spectral, boenkost2026powerlawscalingeffectivepopulation}.
More surprisingly, Neveu's continuous-state branching process -- which
corresponds to $\alpha = 1$ and $\psi(\theta)= \theta \log \theta$ --
and $\alpha$-stable branching processes with $\alpha \in (1,2)$
also emerge as the scaling limits of certain binary branching Brownian
motions with absorption \cite{berestycki13, tourniaire21}. Conversely to
Galton--Watson processes with a heavy tail, the jumps in the limit are no
longer the consequence of rare reproduction events, but are driven by the
spatial structure of the system. In this type of models, some particles
manage to make rare excursions away from the origin and to survive for an
exceptionally long time. During this time they produce a large offspring
that accumulates to a jump in the population size at the scaling limit.
Although we recover the universality class of a Galton--Watson model in
the limit, the mechanism underlying the jumps is of a very different
nature: jumps emerge from the spatial structure and not from the
stochasticity in the reproduction.

Altogether, these findings raise the following question: what property of
the spatial structure leads to the emergence of an $\alpha$-stable
branching process? The objective of the present article is to start
tackling this question in the context of branching diffusions. A first
contribution of our work is to propose a simple general criterion
under which we conjecture that a branching diffusion belongs to the
universality class of an $\alpha$-stable branching process, with $\alpha
    \in (1, 2)$. A second contribution is that we confirm this prediction for
the class of branching Brownian motions introduced in \cite{tourniaire21},
by studying the scaling limit of the genealogy of the population seen as
a random metric measure space. This is the first rigorous result about
genealogies in this regime. Finally, an additional source of motivation is that
these branching Brownian motion models approximate the front of a noisy
travelling wave. Our conjecture can be seen as an extension of the ideas
of \cite{birzu2018fluctuations, birzu2021genealogical} in this context,
and our result verifies the predictions of \cite{birzu2021genealogical}
regarding the genealogy of the front in the so-called semipushed regime.
We refer to \cite{berestycki13, tourniaire21, schertzer2023spectral} for
background, where the connection with noisy traveling waves is discussed
at length, or to Section~\ref{sec:noisyFront} for a shorter account.

\subsection{Universality classes of branching diffusions}
\label{sec:branchingDiffusion}

\subsubsection{Criticality}
Consider an open domain $\Dom_\infty \subseteq \mathbb{R}^d$, $d \ge 1$.
Let $(X_t)_{0 \le t < \tau_\infty}$ be a diffusion process taking values
in $\bar{\Dom}_\infty$ and killed upon hitting the boundary $\partial \Dom_\infty$.
That is, before the killing time
$\tau_\infty = \inf\{t > 0 : X_t \in \partial \Dom_\infty\}$,
the process $(X_t)$ evolves according to the stochastic differential equation
\begin{equation}
    \label{eq:SDE}
    \diff X_t = d(X_t)\,\diff t + \sigma(X_t)\,\diff W_t,
    \qquad 0 \le t < \tau_\infty,
\end{equation}
where $(W_t)$ is an $m$-dimensional Brownian motion and
$d \in C^{2}(\bar{\Dom}_\infty, \mathbb{R}^d)$ and
$\sigma \in C^{2}(\bar{\Dom}_\infty, \mathbb{R}^{d \times m})$.
We further assume that $b$ and $\sigma$ are locally bounded functions and
that the diffusion matrix $\sigma \sigma^{*}$ is uniformly positive
definite on $\bar{\Dom}_\infty$. Alternatively, this diffusion can be
defined as the unique solution to the generalised martingale problem
associated to the differential operator
\begin{equation}
    \label{eq:op_L}
    \mathcal{L}_{\infty} = \frac{1}{2}
    \sum_{i,j=1}^d (\sigma \sigma^*)_{ij}(x)\partial_{x_i} \partial_{x_j}
    \ + \ \sum_{i=1}^d d_i(x) \partial_{x_i}
\end{equation}
on $\bar{\Dom}_\infty$ (we refer the reader to~\cite[Section 1.13]{pinsky95} 
for a rigorous definition of this martingale problem). We write
$\tilde{\mathcal{L}}_\infty$ for the adjoint operator.

A (dyadic) branching diffusions killed at the boundary of $\Dom_\infty$
is now defined as follows. Suppose that for every $y\in
\bar{\Dom}_\infty$, we are given a non-negative branching rate $r(y)$
The process starts with one particle at $x\in \Dom_\infty$. This particle
moves according to \eqref{eq:SDE} until it reaches the boundary $\partial
\Dom_\infty$ where it is killed (i.e.~sent to the graveyard state). When
at position $y$, the particle branches at rate $r(y)$ in which case it is
replaced by two daughter particles at the same position $y$. Each of
these offspring then evolves independently, following the same stochastic
dynamics as their parent. We denote by $\mathcal{N}_{t,\infty}$ the set
of all particles alive at time $t$ in the branching diffusion and define
the number of particles as $Z_{t,\infty} \coloneqq
\abs{\mathcal{N}_{t,\infty}}$. For a particle $v\in \mathcal{N}_t$ and
$s\leq t$, we write $X_v(s)$ for the position of the unique ancestor of
$v$ alive at time $s$. We write $\P_x$ for the law of the branching
diffusion started from a single particle at $x$ and $\E_x$ for the
associated expectation.

We now introduce the notion of criticality that will be used throughout
the paper. Define the transition measure $p_t(x,\diff y)$ by
\[
    p_t(x,B) =
    \mathbb{E}_x\left[\exp\left(\int_0^t r(X(s))\diff s\right);X(t)\in B\right]\quad
    \text{for measurable $B\subset \Dom_\infty$}.
\]
We say that the operator $\mathcal{L}_\infty+r$ possesses
a Green's measure on $\Dom_\infty$ if
\[
    \int_0^\infty p_t(x,B)\diff t<\infty,
\]
for all $x\in \Dom_\infty$ and for all bounded measurable $B\subset \Dom_\infty$ such that $\bar{B}\subset \Dom_\infty$.
\begin{definition}[Criticality, \cite{pinsky95}, Section~4.3]
    \label{def:criticality}
    The operator $\mathcal{L}_\infty+r$ is critical if:
    \begin{enumerate}
        \item it does not possess a Green's measure on $\Dom_\infty$; and
        \item there exists a positive harmonic function $h_\infty\in
                  C^{2}(\Dom_\infty)$, in the sense that
              \begin{equation} \label{eq:defHarmonic}
                  (\mathcal{L}_\infty+r)h_\infty=0 \quad \text{on $\Dom_\infty$}
                  \quad  \text{and}
                  \quad h_\infty=0\quad \text{on $\partial \Dom_\infty$.}
              \end{equation}
    \end{enumerate}
\end{definition}

The adjoint $\tilde{\mathcal{L}}_\infty+r$ of a critical operator is also
critical \cite[Theorem~4.3.3]{pinsky95} and hence there exists a
positive function $\tilde{h}_\infty \in C^{2}(\Dom_\infty)$ such that
\[
    (\tilde{\mathcal{L}}_\infty+r)\tilde h_\infty=0 \quad \text{on $\Dom_\infty$}
    \quad  \text{and} \quad
    \tilde h_\infty=0\quad \text{on $\partial \Dom_\infty$.}
\]
Moreover, $h_\infty$ and $\tilde h_\infty$ are unique up to constant
multiples~\cite[Theorem 4.3.4]{pinsky95}, and $h_\infty$ (resp.~$\tilde
    h_\infty$) is the unique invariant positive function (resp.~invariante
density) for the transition measure $p_t(x,\diff y)$~\cite[Theorem~4.8.6]{pinsky95}.

\subsubsection{The spine}
\label{sect:subdomains}

We now assume that the operator ${\cal L}_\infty+r$ is critical. We can define a stochastic
process $(\zeta_{t})_{t \ge 0}$, the \emph{spine process}, as the
solution to the SDE
\begin{equation}
    \label{eq:def_spine_g}
    \diff \zeta_{t,\infty} = d(\zeta_{t,\infty}) \,\diff t + 
    \sigma \sigma^*(\zeta_{t,\infty})
    \frac{\nabla h(\zeta_{t,\infty})}{h(\zeta_{t,\infty})} \,\diff t
    + \sigma(\zeta_{t,\infty})\,\diff W_t,
    \quad t\geq 0,
\end{equation}
Equivalently, it is defined as the unique solution to the generalised
martingale problem associated to the differential operator
$({\cal L}_{\infty}+r)(h_\infty \cdot)/h_\infty$. See
e.g.\ \cite{pinsky95}, Chapter~4. We denote by $q_{t,\infty}(x,\diff y)$ its
transition kernel. It is well known that this kernel can be expressed in terms
of $p_{t,\infty}(x, \diff y)$ via
\begin{equation} \label{eq:kernelSpine}
    q_{t,\infty}(x, \diff y) =
    \frac{h_\infty(y)}{h_\infty(x)}\,p_{t,\infty}(x, \diff y), \quad
    x,y \in \Dom_\infty,
\end{equation}

Intuitively, the spine process describes the position along the ancestral
line (the spine) of a ``typical'' particle in the branching process,
conditioned on eternal survival. The spine process admits a stationary
measure given by
\begin{equation} \label{eq:def_Pi_g}
    \Pi_\infty = h_\infty\,\tilde h_\infty.
\end{equation}
In an unbounded domain, the measure $\Pi_\infty$ might have infinite mass.
We impose that this is not the case. The following terminology is
borrowed from \cite{pinsky95}.

\begin{hyp}[Product $L^1$ criticality] \label{hyp:L1criticality}
    The operator $\mathcal{L}_\infty + r$ is critical and $\int_{\Dom_\infty}
        h_\infty \tilde{h}_\infty < \infty$.
\end{hyp}

The spine process is positive recurrent in that case, see \cite[Section
    4.9]{pinsky95}. Further, from Theorem~4.9.5~(ii) in \cite{pinsky95},
Assumption~1 implies that $\tilde h_\infty$ is also integrable and we
choose the renormalization of $h_\infty$ and $\tilde h_\infty$ such that
\[
    \int_{\Dom_\infty} h_\infty \tilde{h}_\infty = 1,
    \ \ \text{and} \ \
    \int_{\Omega_\infty} \tilde h_\infty = 1.
\]
In particular, $\Pi_\infty$ is the invariant probability distribution for the spine.
Moreover, one can show (see e.g.~\cite[Theorem
    4.9.9]{pinsky95}) that for any bounded measurable function $f$ with
compact support in $\Omega_\infty$,
\begin{equation} \label{eq:cv_PF}
    \lim_{t \to \infty} \int_{\Dom_\infty} p_{t,\infty}(x,y) f(y) \diff y
    = h_\infty(x)\int_{\Dom_\infty}
    \tilde h_\infty(y)f(y)\diff y.
\end{equation}
Intuitively, $h_\infty(x)$ can be thought of as the reproductive value at
$x$, that is, the mean long-term contribution of a particle located at
$x$ to the branching diffusion, while $\tilde h_\infty$ represents the
stable configuration of the system. The limit \eqref{eq:cv_PF} can be
thought of as a generalization of the Perron--Frobenius theorem for
branching diffusions.

\subsubsection{Conjecture}
\label{sect:conjecture}

We can now state the last condition under which we expect a branching
diffusion to belong to the universality class of an $\alpha$-stable
branching process. This condition serves as the spatial analogue of the
heavy-tail condition \eqref{eq:heavy-tail-dist} for critical
Galton--Watson processes.

\begin{hyp}[Spatial power-law tail.]
    \label{hyp:alpha_tail}
    There exist $\alpha \in (1, 2)$ and $c > 0$ such that
    \begin{equation}\label{eq:fundamental-condition}
        \int_{\Dom_\infty}
        \indic_{\{ h_\infty(x) \ge z \}}
        \tilde{h}_{\infty}(x) \diff x
        \sim
        \frac{c}{z^\alpha}, \qquad \text{as $z \to \infty$.}
    \end{equation}
\end{hyp}

Note that this requires the function $h_\infty$ to be unbounded and thus,
by standard theory, the domain $\Dom_\infty$ to be unbounded as well.
Under these assumptions, we expect that the branching diffusion belongs
to the universality class of the corresponding $\alpha$-stable
continuous-state branching process (CSBP).

\begin{conjecture} \label{conj}
    Assume that Assumption \ref{hyp:L1criticality} and \ref{hyp:alpha_tail}
    hold and let $\gamma = 1 / (\alpha - 1)$. If time is rescaled by
    $N$ and the number of particles by $N^\gamma$, the size of the
    branching diffusion and its genealogy converge to those of an
    $\alpha$-stable branching process, with $\alpha \in (1,2)$ given by
    \eqref{eq:fundamental-condition}. Furthermore, the particles are
    asymptotically distributed as i.i.d.\ random variables with
    distribution $\tilde{h}_\infty$.
\end{conjecture}

Let us provide a heuristic justification of this conjecture. See the
forthcoming Figure~\ref{fig:intuition} for a graphical illustration for a
branching Brownian motion. By \eqref{eq:cv_PF}, we expect that an initial
particle at $x$ has on average $h_\infty(x) \tilde{h}_\infty(y)$
descendants located at $y$ at large times. For any $z > 0$, let us
introduce the subdomain
\begin{equation} \label{eq:truncatedDomain}
    \Omega_z \coloneqq \{ x \in \Omega_\infty : h_\infty(x) < z \},
\end{equation}
which can be interpreted as the set of locations from which an initial
particle has less that $z$ descendants, on average at large times.

Let us assume that the number of particles at time $tN$ scales as
$N^\gamma$, with $\gamma = 1/(\alpha-1)$, and denote by
$\bar{Z}_{t,\infty} \coloneqq Z_{tN,\infty} / N^\gamma$ the number of
particles at that time, rescaled by $N^\gamma$. Therefore, under condition
\eqref{eq:fundamental-condition}, we see from \eqref{eq:cv_PF} that the
probability of observing a particle in $\Omega_\infty \setminus \Omega_{z
N^\gamma}$ at time $tN$ should be approximately
\[
    N^\gamma \bar{Z}_{t,\infty}
    \int_{\Omega_\infty \setminus \Omega_{zN^\gamma}} \tilde{h}_\infty(y) \diff y
    \sim
    \bar{Z}_{t,\infty} \frac{c}{N z^\alpha},
    \qquad \text{as $N \to \infty$.}
\]
Recall that $h_\infty(x)$ can be interpreted as the number of descendants at
large times of an initial particle at $x$, so that particles in
$\Omega_\infty \setminus \Omega_{zN^\gamma}$ have an offspring larger
than $z N^\gamma$. After accelerating time by $N$, these heuristics are
consistent with the idea that $\bar{Z}_{t,\infty}$ experiences a jump
larger than $z$ at rate $c \bar{Z}_{t,\infty} / z^\alpha$, which is the
infinitesimal description of an $\alpha$-stable CSBP.

\section{Main results and heuristics}

Since our main results are stated in a technical framework, we begin with
an informal overview of this section.

\begin{enumerate}
    \item In Section \ref{sect:genealogical-structure}, we define the
    genealogy of CSBPs and branching diffusions at a fixed time $t > 0$.
    For CSBPs, the limiting genealogy is constructed as a tree-like
    metric on a continuous space \cite{aldous1993continuum,
    evans1998kingman, evans2006probability}. Informally (see
    Figure~\ref{fig:reduced_process}), it may be viewed as an ultrametric
    tree whose leaves are labeled by an interval.

    \item In Section \ref{sect:moments}, we introduce moments for the
    genealogy of both the limiting CSBP and the branching diffusion.
    Roughly speaking, the moment of order $k$ at time $t$
    corresponds to the (rescaled) number of genealogical trees with
    $k$ leaves and a prescribed structure. We derive recursive
    formulas for these moments and state our first main result
    (Theorem~\ref{thm:main-cv}), which asserts that convergence of
    the moments of a branching diffusion together with a polynomial
    growth condition of the limit entail convergence of the associated
    genealogical structures to an $\alpha$-stable branching process. We
    also obtain a Kolmogorov-type estimate for the survival probability
    of the branching diffusion.

    \item In Section~\ref{sec:BBM}, we apply the preceding results to a
    particle system satisfying Assumptions~\ref{hyp:L1criticality} and
    \ref{hyp:alpha_tail}, for which we establish convergence of the
    genealogy in accordance with Conjecture \ref{conj}. The model is a
    branching Brownian motion with absorption, introduced in
    \cite{schertzer2023spectral} as an extension of \cite{tourniaire21},
    and serves as a toy model for the front of a semipushed (or pushed)
    expansion, which provides a central motivation for this work. Using
    the moment method developed above, we prove convergence of its
    genealogical structure to that of an \(\alpha\)-stable branching process
    (Theorem \ref{th:cv_mmm}). A heuristic for the proof of Theorem
    \ref{th:cv_mmm} is presented in Section \ref{sec:heuristics_BBM}.

    While the proof of the general conjecture for branching diffusions
    remains incomplete, this result provides evidence for
    Conjecture~\ref{conj} in a concrete setting.
\end{enumerate}

\subsection{Genealogical structures}\label{sect:genealogical-structure}

\subsubsection{Genealogical structure of CSBPs}
\label{sec:csbps_intro}

We need to recall several facts about continuous-state branching
processes (CSBPs) before stating our main results. Formally, a CSBP is
defined as a strong Markov process $(\Xi_t)_{t \ge 0}$ on $[0, \infty]$
that satisfies the following branching property:
the sum of two independant copies of the process starting from $x$ and $y$ has the same finite-dimensional distributions as the process starting from $x+y$.

The law of
a CSBP with finite first moment is uniquely characterized by a
branching mechanism $\psi$ as given in \eqref{eq:branching-mechanism}.
The two are connected by a Laplace transform as
\[
    \forall \theta \ge 0,\quad
    \E[ e^{-\theta \Xi_t} \mid \Xi_0 = z] = e^{-z u_t(\theta)},
\]
where $(u_t(\theta))_{t \ge 0}$ is the solution to the
differential equation
\begin{equation} \label{eq:ODECSBP}
    \begin{dcases}
        \partial_t u_t(\theta) = -\psi( u_t(\theta) ) \\
        u_0(\theta) = \theta.
    \end{dcases}
\end{equation}

It is a standard result since the work of Grey \cite{grey1974asymptotic}
that a $\psi$-CSBP can hit $0$ in finite time with positive probability
if and only if its branching mechanism satisfies Grey's condition that
\begin{equation} \label{eq:greyCondition}
    \exists \theta_0 > 0,\quad
    \int_{\theta_0}^\infty \frac{\diff \theta}{\psi(\theta)}
    < \infty, \quad \psi(\theta_0) > 0.
\end{equation}
Under this condition, the probability that the process dies out is given by
\begin{equation} \label{eq:laplaceExponent}
    \forall z, t \ge 0,\qquad \P( \Xi_t = 0 \mid \Xi_0 = z) = e^{-z \bar{u}_t},
    \qquad \bar{u}_t \coloneqq \lim_{\theta \to \infty} u_t(\theta)
\end{equation}
and we can make sense of the process ``started from $0$'' as
the random variable $\bar{\Xi}_t$ whose Laplace transform is
\begin{equation}  \label{eq:laplaceEntrance}
    \forall \theta \ge 0,\quad \bar{u}_t \E[ 1-e^{-\theta \bar{\Xi}_t} ] =
    u_t(\theta).
\end{equation}
We call the distribution of $\bar{\Xi}_t$ the entrance law of the $\psi$-CSBP.

An important genealogical consequence of Grey's condition is that the
population at any time $t > 0$ is descended from a finite number of
ancestors at time $0$, see for instance \cite{bertoin2000Bolthausen}.
In that case, one can construct the genealogy at time $t > 0$ from a
time-inhomogeneous branching process $(Z_{s,t})_{s \in [0,t)}$
called the \emph{reduced process}, as illustrated in
Figure~\ref{fig:reduced_process}. The reduced process is defined by
starting from $Z_{0,t} = 1$ and letting each particle at time $s$ be
replaced by $i \ge 2$ offspring at a rate $r_{i,t-s}$ given implicitly by
\begin{equation} \label{eq:rateReducedCSBP}
    \forall \theta \ge 0,\quad
    \sum_{i \ge 2} r_{i,t-s} \theta^i
    = \theta \Big( \psi'(\bar{u}_{t-s}) - \frac{\psi(\bar{u}_{t-s}) -
        \psi((1-\theta)\bar{u}_{t-s})}{\theta \bar{u}_{t-s}}
    \Big).
\end{equation}
This expression, obtained in \cite[Section~2.7]{LeGall2002}, extends
earlier work of \cite{fleischmann1977structure, yakymiv1981reduced} in the
$\alpha$-stable case.

\begin{figure}
    \centering
    \includegraphics[width=.9\textwidth]{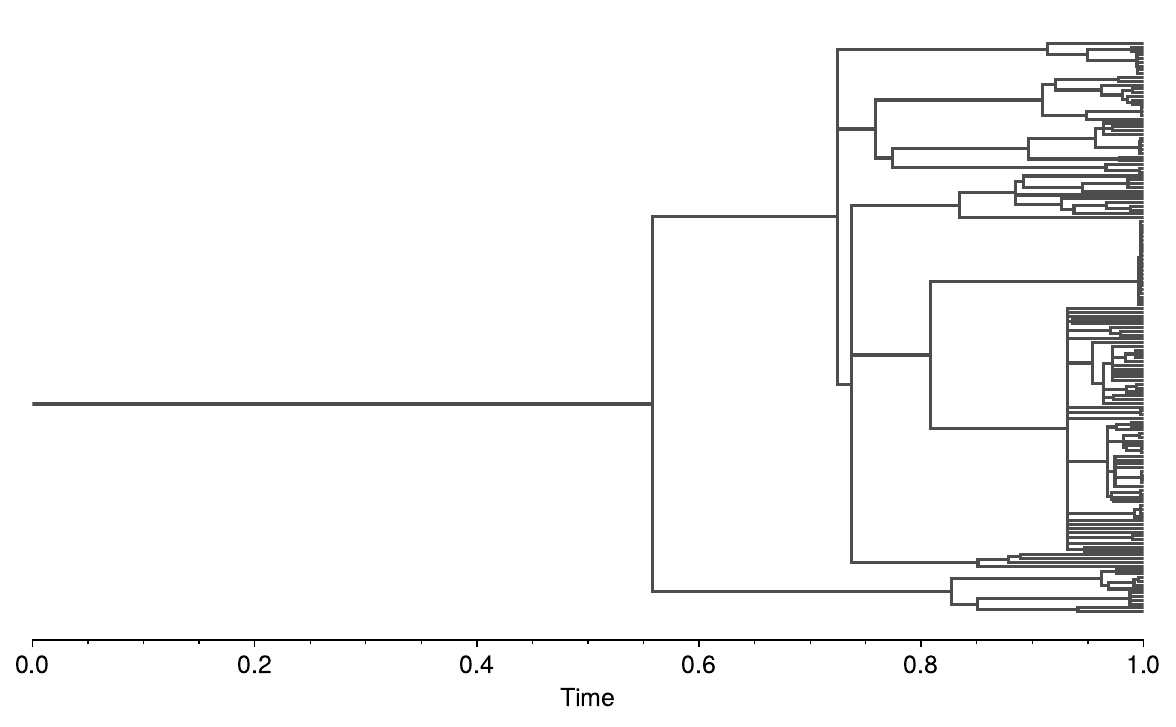}
    \caption{Simulation of the genealogy of an $\alpha$-stable CSBP at
        time $t=1$, with $\alpha = 3/2$. The $\psi$-mm-space corresponds to
        the boundary of this tree, which is the set of leaves of the tree
        (the lineages at time $t=1$) endowed with the tree distance.}
    \label{fig:reduced_process}
\end{figure}

The reduced process is naturally endowed with a tree structure for any
time $s < t$, as illustrated in Figure~\ref{fig:reduced_process}. We want
to define the genealogy at time $t$ of a $\psi$-CSBP as the set of leaves
of that tree. Since the number of lineages in the process explodes as $s
    \uparrow t$, this leaf set is actually continuous and requires a careful
construction. It is a standard idea that such a continuous tree can be
defined as a tree-like metric on a continuous space \cite{aldous1993continuum,
    evans1998kingman, evans2006probability}. In Section~\ref{sec:proofCSBP},
we will define this object rigorously as a random element
$(\mathcal{U}_t, d_t, \vartheta_t)$ of the space of metric measure
spaces, endowed with the Gromov-weak topology of \cite{Greven2009}. We
refer to it as the $\psi$-metric measure space ($\psi$-mm space), see
Definition~\ref{def:psiMmSpace}. The set of labels of the leaves
$\mathcal{U}_t$ can be taken to be an interval with random length, endowed
with the Lebesgue measure $\vartheta_t$. The distance $d_t$ is an
ultrametric distance that gives the coalescence time between pairs of
leaves.

Although the construction of this tree is implicit in
the work of \cite{LeGall2002}, our construction of the measure
$\vartheta_t$ (and thus of the corresponding $\psi$-mm space) appears to
be new. It will enable us to derive recursive moment formulas for
$(\mathcal{U}_t, d_t, \vartheta_t)$ which play a central role in our
approach.

\subsubsection{Genealogical structure of branching diffusions}
\label{sect:gen-structure}

Our convergence results for branching diffusions will be formulated in
terms of the convergence of a random metric space that encodes the
genealogy of the population, augmented with marks that record the spatial
locations of the individuals. Let us define this encoding for branching diffusions.

Recall that $\mathcal{N}_{t,\infty}$ refers to the set  of the particles alive at
time $t$ and that $X_u(t)$ is the location of $u \in \mathcal{N}_{t,\infty}$
at that time. For two individuals $u,v\in\mathcal{N}_{t,\infty}$, we
define their genealogical distance, denoted by $d_{t,\infty}(u,v)$, as
the time to their most recent common ancestor (i.e., their coalescence
time). Furthermore, let $\nu_{t,\infty}$ be the random measure on
$\mathcal{N}_{t,\infty} \times \Omega_\infty$ defined by
\[
    \nu_{t,\infty} = \sum_{u \in \mathcal{N}_{t,\infty}} \delta_{(u, X_u(t))}.
\]
The triple $(\mathcal{N}_{t,\infty}, d_{t,\infty}, \nu_{t,\infty})$ is a
random \emph{marked metric measure space} (mmm-space) in the sense of
\cite{depperschmidt_marked_2011}, see Section~\ref{sec:mmmSpaces} for a
definition. It encodes the genealogy and spatial locations of the
particles at time $t$. We derive our limit by applying the same scaling
as for Galton--Watson processes with heavy tails in
Proposition~\ref{prop:GW}. Namely, we rescale time by $N$ and
population size by $N^\gamma$ and define $(\bar{\mathcal{N}}_{t,\infty},
    \bar{d}_{t,\infty}, \bar{\nu}_{t,\infty})$ as
\begin{equation} \label{eq:rescalingMmmSpace}
    \bar{\mathcal{N}}_{t,\infty} = \mathcal{N}_{tN,\infty},\qquad
    \bar{d}_{t,\infty} = \frac{d_{tN,\infty}}{N},\qquad
    \bar{\nu}_{t,\infty} = \frac{\nu_{tN,\infty}}{N^\gamma}.
\end{equation}
Let us also denote by $\bar{Z}_{t,\infty} = \angle{\bar{\nu}_{t,\infty},
        1}$ the (rescaled) number of particles.

\subsection{Moments and a  convergence criterion}
\label{sect:moments}

\subsubsection{Planar ultrametric matrices}
\label{sec:heuristicMoments}

The moment of order $k$ of a random metric space will be defined as a finite
measure on the space $\mathbb{U}_k$ of $k \times k$ \emph{planar}
ultrametric matrices. Namely, we say that $U = (U_{ij})_{i, j \le k} \in
    \mathbb{U}_k$ is a planar ultrametric matrix if
\begin{equation} \label{eq:planeUltrametric}
    \forall i < \ell < j,\quad U_{i,j} = \max(U_{i,\ell}, U_{\ell,j}).
\end{equation}
Intuitively, one can think of $U \in \mathbb{U}_k$ as the distance
between the $k$ leaves of a planar ultrametric tree. See Figure \ref{fig:notationTree}.

We will enrich this
tree structure with the information about the spatial location of the
leaves, and define the set of \emph{marked} planar ultrametric matrices
as $\mathbb{U}^*_k \coloneqq \mathbb{U}_k \times \Omega^k_\infty$.
For every $i\in[k]$, it will be convenient to introduce the map $X_i
    \colon \mathbb{U}^*_k \to \Omega_\infty$ that gives the mark of the
$i$-th leave of a marked ultrametric matrix.

In order to introduce the planar moments of a $\psi$-mm space in
Section~\ref{sec:recursion-moments}, we will need some preliminary
notation illustrated in Figure~\ref{fig:notationTree}. Informally,
cutting the tree encoded by $U \in \mathbb{U}_k$ at its deepest
branchpoint splits $U$ into several subtrees which leads to the following
definitions. Recall that a \emph{composition} of $k$ is a vector made of
positive integers summing up to $k$.

\begin{figure}[t]
    \centering
    \includegraphics[width=.6\textwidth]{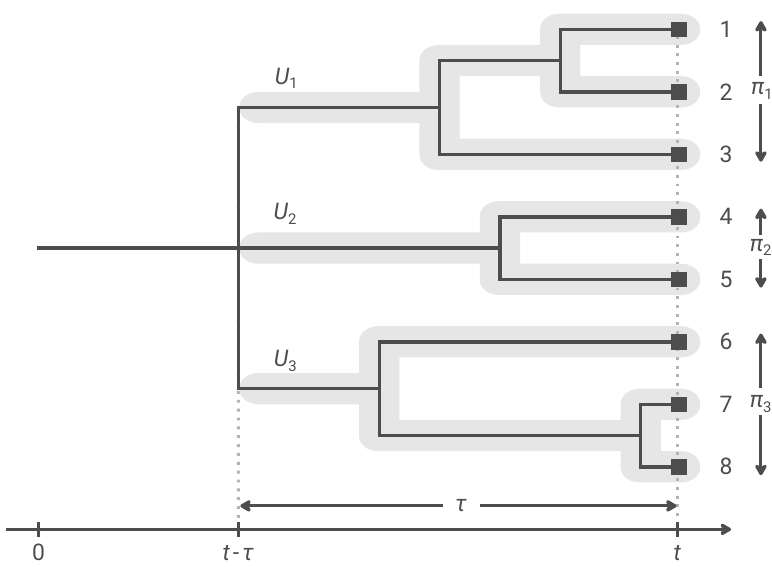}
    \caption{Illustration of Definition~\ref{def:composition}. If the matrix
        $U$ gives the pairwise distances between the leaves of the tree,
        cutting the tree at distance $\tau(U)$ from the leaves splits it into
        three subtrees indicated in grey, corresponding to the partition of
        the leaves $\pi^s(U) = \{ \{1,2,3\}, \{4,5\}, \{6,7,8\}\}$. The
        corresponding composition is $c^s(U) = (3,2,3)$.}
    \label{fig:notationTree}
\end{figure}

\begin{definition}[Decomposition of planar ultrametric matrices]
    \label{def:composition}
    Fix a tree $U \in \mathbb{U}_k$.
    \begin{itemize}
        \item Let $\tau(U) = \max_{i,j \le k} U_{ij}$ be the depth of the first
              branch point of $U$.
        \item Let $\pi(U)$ be the partition of $\{1, \dots, k\}$ such
              that $i \sim j$ iff $U_{ij} < \tau(U)$. The planarity of $U$
              implies that each block of $\pi(U)$ is made of consecutive
              integers and the blocks can be ordered increasingly. We let
              $\pi_p(U)$ be the $p$-th block.
        \item Let $\abs{c(U)}$ be the number of blocks of $\pi(U)$
              and let $c(U) = (c_p(U))_p$ be the composition of $k$
              given by $c_p(U) = \abs{\pi_p(U)}$.
        \item For each $p \le \abs{c(U)}$, let $U_p(U)$ be the
              restriction of $U$ to the $p$-th block $\pi_p(U)$. We will
              often write $U_p \equiv U_p(U)$ to ease the notation.
    \end{itemize}
    All those notations are directly extended to the set of marked
    matrices $\mathbb{U}^*_k$. For instance, for any marked matrix
    $U^* = (U, (x_i)_{i=1}^k) \in \mathbb{U}_k^*$, we will write
    $\tau(U^*) \coloneqq \tau(U)$ for the deepest branching point,
    $U^*_p(U^*)$ for the $p$-th marked subtree at depth $s$. Finally, as
    a shorthand notation, we will write $U^*_p \equiv U_p^*(U^*)$.
\end{definition}

Having introduced the necessary notation, we now define the
planar moments of  branching diffusions and of $\psi$-mm spaces.

\subsubsection{Spatial cutoff and planar moments of a branching diffusion}
\label{sect:cut-off}

The convergence of the genealogy in the forthcoming
Theorem~\ref{th:cv_mmm} will follow from a method of moments for random
marked metric measure spaces as in \cite{foutel22}. A caveat to
using moments in our context is that the target metric space -- the
$\alpha$-stable genealogy -- has infinite variance due to rare
reproductive events. In Section~\ref{sect:conjecture}, we have already
argued that these events correspond to rare excursions of particles in
regions with high reproductive value. In order to recover finite moments,
we will introduce a spatial cutoff by removing all particles that exit an
appropriate subdomain  $\Omega\subseteq \Omega_\infty$. 
This idea will be expanded further in Section
\ref{sec:sufficient_condition}, and see the forthcoming
Figure~\ref{fig:intuition} for a graphical illustration.

Let $\Omega\subseteq \Omega_\infty$ be a \emph{bounded} domain. We define
another branching diffusion by killing particles upon hitting the
boundary of $\Omega$. We will refer to this process as the branching
diffusion on $\Omega$ and we denote by ${\mathcal{L}}_\Omega$ and
$\tilde{\mathcal{L}}_\Omega$ the corresponding operator in
\eqref{eq:op_L} and its dual. Since $\Omega$ is assumed to be bounded, we
have the following classical result on elliptic operators.

\begin{proposition}[Theorem 4.7.1 \cite{pinsky95}]\label{prop:crit-finite}
    If $r$ is bounded, there exists a unique $w_\Omega \in \R$ such
    that ${\cal L}_\Omega + r + w_\Omega$ is critical on $\Omega$.
\end{proposition}

We will denote by $h_\Omega$ and $\tilde h_\Omega$ two  harmonic functions of the operator ${\cal
            L}_{\Omega}+r+w_{\Omega}$ and its dual. We define $\zeta_t \equiv \zeta_{t,\Omega}$ to
be the spine on $\Omega$ defined as in \eqref{eq:def_spine_g} replacing
$h_\infty$ by $h_\Omega$, i.e., the spine $\zeta$
is given by the $h$-transform
\begin{equation}
    \label{eq:spine-cut}
    \diff \zeta_{t} = d(\zeta_{t}) \,\diff t + \sigma \sigma^* (\zeta_{t})
    \frac{\nabla h_\Omega(\zeta_{t})}{h_\Omega(\zeta_{t})} \,\diff t
    + \sigma(\zeta_{t})\,\diff W_t,
    \quad t\geq 0.
\end{equation}
For a fixed $\Omega$, we will often remove the $\Omega$ subscript in the
absence of ambiguity so that quantities without the 
and write for instance ${\cal L} \equiv {\cal L}_\Omega$ and $h \equiv
h_\Omega$ in order to ease the notation. The quantities such as ${\cal
N}_t \equiv {\cal N}_{t,\Omega}$ or $\nu_{t}\equiv \nu_{t,\Omega}$ with
no infinite subscript will refer to the analog quantities of Section
\ref{sect:gen-structure} for the branching diffusion on $\Omega$.

With the notation above, we define the planar moment of order $k$ of the
branching diffusion as the measure $\mathbf{M}^{k,t}_{x} \equiv
    \mathbf{M}^{k,t}_{x,\Omega}$ on $\mathbb{U}^*_k$ such that
\begin{equation} \label{eq:biasedMomentBBM}
    \mathbf{M}^{k,t}_{x}[F]
    =
    \frac{1}{h(x)}
    \E_x\bigg[ \sum_{u_1 < \dots < u_k \in \mathcal{N}_{t}}
        F\big(
        d_t(\mathbf{u}), X_{\mathbf{u}}(t)
        \big) \prod_{i=1}^k h(X_{u_i}(t))
        \bigg],
\end{equation}
for any measurable $F \colon \mathbb{U}^*_k \to \R_+$, and where
\[
    \mathbf{u} = (v_i)_{i \le k},\qquad
    X_{\mathbf{u}}(t) = (X_{u_i}(t))_{i \le k},\qquad
    d_t(\mathbf{u}) = \big( d_t(u_i, u_j) \big)_{i,j \le k}.
\]
Informally, the moment of order $k$ counts the number of trees with $k$
leaves with a given genealogical structure at time $t$ in the population.
Note that in this definition, a particle at $x$ is sampled with a
probability proportional to $h(x)$. Finally, our convergence result will
require that time is rescaled by $N$ and the number of particles by
$N^\gamma$, recall Proposition~\ref{prop:GW}. For that purpose, we define
the rescaled moment of the branching diffusion as 
\begin{equation}\label{eq:scaled_moments}
    \hat{\mathbf{M}}^{k,t}_{x}\big[F(U, (X_i)_i) \big]
    = \frac{1}{N^{\gamma(k-1)}} \mathbf{M}^{k,Nt}_{x}\big[
        F\big( \tfrac{U}{N}, (X_i)_i \big)
        \big].
\end{equation}

\subsubsection{Planar moments of a \texorpdfstring{$\psi$}{psi}-mm space}
\label{sec:recursion-moments}

Assume that $\psi$ is a branching mechanism as in
\eqref{eq:branching-mechanism}, with $a \in \mathbb{R}$, $b \ge 0$, and
$\Lambda$ satisfying a finite exponential moment condition, namely
\begin{equation} \label{eq:finiteExpMoment}
    \exists \eta > 0,\qquad
    \int_{(0, \infty)} \big( e^{\eta z} - 1\big) \Lambda(\mathrm{d} z) < \infty.
\end{equation}
Let $(\mathcal{U}_t, d_t, \vartheta_t)$ denote the associated $\psi$-mm
space as in Section~\ref{sect:genealogical-structure} (or see
the forthcoming Definition~\ref{def:psiMmSpace}). In
contrast to branching diffusions, which are equipped with the
Ulam--Harris labeling and thus carry a natural planar structure,
$\psi$-mm spaces are not planar objects. As a result, planar moments
cannot be defined directly and require additional care. We proceed in two
steps. First, define the $k$-th \emph{unplanar moment} at time $t$ as the
measure on $\mathbb{R}_+^{k \times k}$ whose integral against a
measurable test function $F \colon \mathbb{R}_+^{k \times k} \to
    \mathbb{R}_+$ is given by
\[
    \mathbb{E}\bigg[
        \int_{\mathcal{U}_t^k}
        F\big( d_t(\mathbf{u}) \big)\, \vartheta_t^{\otimes k}(\mathrm{d} \mathbf{u})
        \bigg],
    \qquad \mathbf{u} = (u_i)_{1 \le i \le k},\quad
    d_t(\mathbf{u}) = (d_t(u_i, u_j))_{1 \le i,j \le k}.
\]
Secondly, this unplanar moment can be obtained by summing over all
permutations of the leaves of another measure $\hat{\mathcal{M}}^{k,t}$
on the space of planar matrices $\mathbb{U}_k$. This family of measures,
that we name the $k$-th \emph{planar moment} of a $\psi$-mm space, is
defined as follows.

\begin{definition}[Planar moments of $\psi$-mm spaces]
    \label{def:planarMoments}
    Let $\psi$ be a branching mechanism defined in
    \eqref{eq:branching-mechanism}, with $a \in \R$, $b \ge 0$ and
    $\Lambda$ having a finite exponential moment in the sense of
    \eqref{eq:finiteExpMoment}. The $k$-th planar moment of the $\psi$-mm
    space at time $t > 0$ is the measure $\hat{\mathcal{M}}^{k,t}$ on
    $\mathbb{U}_k$ defined recursively as follows.
    \begin{itemize}
        \item For $k = 1$ and $t \ge 0$, $\hat{\Mc}^{1,t}[1] = e^{-a t}$.
        \item For every $k \geq 2$ and any function $G$ of the product form
          \begin{equation}
              \label{eq:product-form1}
              \forall U \in \mathbb{U}_k,\quad
              G(U) = \indic_{\{ c(U) = c \}} f(\tau(U)) \prod_{p=1}^{\abs{c}} F_p( U_p ),
          \end{equation}
          where $c$ is a composition of $k$, and $f \coloneq \R_+ \to \R$ and the
          functions $F_i \colon \mathbb{U}_{c_i}\to\R$ are continuous and
          bounded,
              \begin{equation}\label{moments:CSNP}
                  \hat{\Mc}^{k, t}[G] = \frac{m_{\abs{c}}}{\abs{c}!}
                  \int_0^t e^{-a(t-s)} f(s)  \prod_{i=1}^{\abs{c}}
                  \hat{\Mc}^{c_i,s}[F_i] \diff s,
              \end{equation}
              where
              \begin{equation} \label{eq:momentJumpMeasure1}
                  \forall p \ge 2,\quad m_p = (-1)^p \psi^{(p)}(0) =
                  \indic_{\{ p = 2\}} b +
                  \int_0^\infty
                  z^p \Lambda(\diff z).
              \end{equation}
    \end{itemize}
\end{definition}

The unplanar moment is now obtained by summing over all permutations of the
leaves as follows. The following result will be proved in
Section~\ref{sec:proofCSBP}.

\begin{proposition} \label{prop:momentMmSpace}
    Suppose that $\psi$ is a branching mechanism as in
    \eqref{eq:branching-mechanism} that satisfies Grey's condition
    \eqref{eq:greyCondition} and let $(\mathcal{U}_t, d_t, \vartheta_t)$
    be the $\psi$-mm from Definition~\ref{def:psiMmSpace} and $\bar{u}_t$
    be as in \eqref{eq:laplaceExponent}. Then for any measurable $F
    \colon \R_+^{k \times k} \to \R_+$,
    \begin{equation} \label{eq:momentPsiMM}
        \bar{u}_t \E\bigg[
            \int_{\mathcal{U}^k_t} F\big( d_t(\mathbf{v}) \big)
            \vartheta_t^{\otimes k}(\diff \mathbf{v})
            \bigg] = \sum_P \hat{\mathcal{M}}^{k,t}[F \circ P],
    \end{equation}
    where the sum runs over all permutations $P$ of $\{1, \dots, k\}$,
    and $P$ acts on $\mathbb{U}_k$ by permuting rows and columns.
\end{proposition}

Definition~\ref{def:planarMoments} provides a recursive characterization
of the (planar) moments of a $\psi$-mm space. In
Section~\ref{sect:moments-recursion}, we use product functionals of the
form \eqref{eq:product-form1} to derive a similar recursion formula for
the moments of the branching diffusion, in analogy with 
Definition~\ref{def:planarMoments}. As these formulas involve heavy
notation, we do not state them here.

Let us nevertheless briefly note that the recursive characterization of
the planar moments of the branching diffusion involves the spine process
\eqref{eq:spine-cut} and exhibits a structure similar to that of the
recursive construction of the planar moments of the $\psi$-mm space
above. In fact, we will show that for branching Brownian motions in the
semipushed regime of Section~\ref{sec:BBM} below, the fast mixing of the
spine implies that the two recursive characterizations coincide at the
limit.

\subsubsection{Sufficient convergence criterion to \texorpdfstring{$\alpha$}{alpha}-stable genealogies}
\label{sec:sufficient_condition}

In the previous sections, we alluded to recursive formulas for the
(planar) moments of the $\psi$-mm space and branching diffusions. In this
section, we address the following question: under which conditions does
convergence of the moments of a branching diffusion imply convergence to
an $\alpha$-stable branching process? This is the focus of the present section.

To illustrate the forthcoming approach, we briefly outline how one might
prove the convergence of a Galton--Watson process with heavy tailed
offspring distribution to an $\alpha$-stable CSBP using the method of
moments. The idea is to truncate large reproduction events by suppressing
all births with more than $A N^\gamma$ offspring for some arbitrary large
$A$. Under this truncation, one expects the moments to remain bounded and
to converge to those of a $\psi_A$-CSBP with jump measure
\[
    \Lambda_A(\diff z) = c \mathbf{1}_{(0,A)}(z) \frac{\diff z}{z^{\alpha+1}}.
\]
By letting $A \to \infty$ one should recover the full jump structure.

Let us now try to apply a similar methodology for branching diffusions.
For any $z > 0$, we defined $\Omega_z$ in \eqref{eq:truncatedDomain} as
the set of locations $x \in \Omega_\infty$ such that $h_\infty(x) \le z$.
By our heuristics in Section~\ref{sect:conjecture}, we expect a particle
outside of $\Omega_{AN^\gamma}$ to generate a jump in the limit which is
larger than $A$, after appropriate rescaling. Therefore, it is natural to
introduce a cutoff at $z = AN^\gamma$ for some $A \ge 1$ and $N > 0$ and
to kill the diffusion at the boundary of the bounded domain
$\Omega_{AN^\gamma}$. The purpose of this specific choice is that we
expect that the branching diffusion converges as $N \to \infty$ with $A
\ge 1$ fixed to a CSBP with truncated jump measure. The parameter $A$
controls the accuracy of the cutoff, and we should recover all the jumps
of the process by further letting $A \to \infty$. In the following, we
will denote  by
\[
    \mathcal{N}_{t,(N,A)} \coloneqq \mathcal{N}_{t,\Omega_{AN^\gamma}},
    \ \ h_{(N,A)} \coloneqq h_{\Omega_{AN^\gamma}}, \ \
    \mathbf{M}^{k,t}_{x,(N,A)} \coloneqq \mathbf{M}^{k,t}_{x,\Omega_{AN^\gamma}}
\]
and $\hat{\mathbf{M}}^{k,t}_{x,(N,A)}$ the rescaled moment as in
\eqref{eq:biasedMomentBBM}. In the absence of ambiguity, we will remove
the subscript $(N,A)$ to ease the notation. We will also write
$\Lambda_A$ for the pushforward of a measure $\Lambda$ by the dilatation
map $z \mapsto Az$ with mass rescaled by $A^{-\alpha}$, formally defined
as
\begin{equation}\label{eq:rescaled_lambda}
    \int_{(0,\infty)} f(z) \Lambda_A(\diff z)
    = A^{-\alpha} \int_{(0,\infty)} f(Az) \Lambda(\diff z).
\end{equation}

\begin{theorem} \label{thm:main-cv}
    Consider a branching diffusion satisfying
    Assumption~\ref{hyp:L1criticality} and such that the following holds
    for some $\alpha \in (1, 2)$.
    \begin{enumerate}
        \item There exists a measure $\Lambda$ with a finite exponential
              moment \eqref{eq:finiteExpMoment}  such that for each $A \ge 1$,
              there exist $a_A \in \R$, $b_A \ge 0$ with
              \begin{equation} \label{eq:convergenceMoment}
                  \forall k \geq 1, \: t \ge 0, \quad
                  \lim_{N \to \infty} \hat{\mathbf{M}}^{k,t}_{x} =
                  \hat{\mathcal{M}}^{k,t}_{A} \otimes
                  (h_\infty \tilde h_\infty)^{\otimes k},
              \end{equation}
              where $\hat{\mathcal{M}}^{k,t}_{A}$ is the $k$-th planar moment
              of the $\psi_{A}$-mm space from Definition~\ref{def:planarMoments}
              with
              \begin{equation} \label{eq:psiA}
                  \psi_{A}(\theta) = a_{A}\theta + \frac{b_{A}}{2}\theta^2
                  + \int_{(0,\infty)} \big( e^{-\theta z} - 1 + \theta z \big)
                  \Lambda_{A}(\diff z).
              \end{equation}
              Further, $\lim_{A\to\infty} a_{A}, b_{A}= 0$.

        \item For each $A \ge 1$, $\lim_{N\to\infty} h(x) = h_\infty(x)$
              uniformly on compact sets,
              \begin{equation} \label{eq:mainCV2}
                  \sup_{N \ge 1} \sup_{x \in \Omega_{AN^\gamma}}
                  h(x) N^{-\gamma} < \infty,
                  \qquad
                  \lim_{N \to \infty} \frac{1}{h(x)} \E_x[ Z_{tN} ] 
                  = e^{-a_A t}, \quad t \ge 0.
              \end{equation}
    \end{enumerate}
    Then there exists $C > 0$ such that, for any $t > 0$ and
    $x \in \Omega_\infty$, the following holds for some sequence
    $(\eta_N)_{N \ge 1}$ with $\eta_N \to 0$ as $N \to \infty$:
    \begin{enumerate}
        \item[(1)] Let $\bar{u}_t$ be as in \eqref{eq:laplaceExponent}
              for the branching mechanism $\psi(\theta) = C \theta^\alpha$,
              then there exists $C'$ such that 
              \[
                  \P_x( \bar Z_{t,\infty} > \eta_N ) \sim
                  \frac{C' h_\infty(x)}{(Nt)^\gamma},
                  \qquad \text{as $N \to \infty$}.
              \]
        \item[(2)] Let $(\mathcal{U}_t, d_t, \vartheta_t)$ be the
              $\psi$-mm space corresponding to $\psi(\theta) = C
                  \theta^\alpha$. Conditional on $\{ \bar{Z}_{t,\infty} >
                  \eta_N\}$ under $\P_x$,
              \[
                  \big( \bar{\mathcal{N}}_{t,\infty}, \bar{d}_{t,\infty},
                  \bar{\nu}_{t,\infty} \big)
                  \longrightarrow
                  \left(\mathcal{U}_t, d_t,\vartheta_t \otimes \tilde h_\infty \right),
                  \qquad \text{as $N \to \infty$},
              \]
              in distribution for the marked Gromov-weak topology.
    \end{enumerate}
\end{theorem}

This result should be compared with that for heavy-tailed critical
Galton--Watson processes that we recalled in Proposition~\ref{prop:GW}.
The first item of our theorem may be interpreted as a weak analogue
for branching diffusions of the first part of \eqref{eq:survProbGW}:
the probability of observing a macroscopic population at time $tN$
exhibits the same polynomial decay $C'/(tN)^{\gamma}$. We conjecture that
the same result also holds for $\eta_N = 0$; however, to keep the paper
at a reasonable length, we defer the investigation of this case to future
work. The marked Gromov-weak convergence in the second item formalises
Conjecture~\ref{conj}. Under our condition, it implies that 1) the
genealogy of the branching diffusion converges to the tree structure
encoded by $(\mathcal{U}_t, d_t, \vartheta_t)$, 2) the population size
converges to the same limit as in \eqref{eq:survProbGW}, and that 3) the
location of the particles are asymptotically i.i.d.\ with distribution
$\tilde{h}_\infty(y) \diff y$.

Next, let us discuss the conditions of the theorem  briefly. First, the
branching process with a cutoff is defined on a bounded domain so that the
moment formulas of Section \ref{sect:moments-recursion} allow for an
explicit computation. In practice, this allows us to compute the moments of
the branching diffusion recursively. In Section~\ref{sec:heuristics_BBM},
we will show how such a computation can be carried out on a specific
branching diffusion.

Secondly, it will follow directly from the convergence of the moments in
\eqref{eq:convergenceMoment} that, for $A \ge 1$ fixed, the genealogy of
the branching diffusion on $\Omega$ converges as $N \to \infty$ to some
$\psi_A$-mm space. The convergence of the branching diffusion on the full
domain $\Omega_\infty$ then follows in two steps. First, the fact that
the jump measure $\Lambda_A$ of $\psi_A$ is obtained by the dilatation map
\eqref{eq:rescaled_lambda} will be sufficient to deduce that the
$\psi_A$-mm space converges as $A \to \infty$ to an $\alpha$-stable
genealogy. Second, a perturbation argument for random metric spaces will
allow us to interchange the limits as $N \to \infty$ and as $A \to
    \infty$, showing that the limit of the branching diffusion without cutoff
coincides with the $\alpha$-stable genealogy. Finally, Point~2 is a set
of technical conditions that ensure that we can go from sampling
particles proportional to $h(x)$ to sampling particles uniformly from the
population.

\subsection{Branching Brownian motion in the semipushed regime}
\label{sec:BBM}

Theorem~\ref{thm:main-cv} reduces the convergence of the genealogical
structure of a branching diffusion to an $\alpha$-stable CSBP to a moment
computation. Moments are computed through the recursive formula displayed
in Section~\ref{sec:recursion-moments}. To illustrate the method, we will
now apply those results to a particle system satisfying
Assumption~\ref{hyp:L1criticality} and \ref{hyp:alpha_tail} and for which
we can rigorously compute the limit of the moments. This system is a
branching Brownian motion with absorption that was introduced in
\cite{schertzer2023spectral} (as an extension of the model in
\cite{tourniaire21}) as a toy approximation of the front of a pushed
expansion. This was the original source of motivation for our work, see
Section~\ref{sec:noisyFront} below for a discussion of the connexion
between this model and the front of noisy travelling waves.

\subsubsection{Model and convergence of the genealogy}
\label{sec:model_BMM}

Consider a dyadic branching Brownian motion (BBM) on $\Omega_\infty = (0,\infty)$
with killing at $0$, negative drift $-\mu$ and space-dependent branching
rate
\begin{equation}\label{def:r}
    r(x) = \frac{1}{2}W(x)+\frac{1}{2}.
\end{equation}
We assume that $W \colon [0, \infty) \to \R$ is non-negative,
continuously differentiable, and with support in $[0, 1]$. For this BBM,
the operator $\mathcal{L}_\infty$ from \eqref{eq:op_L} is
\[
    \mathcal{L}_\infty = \frac{1}{2}\partial_{xx}-\mu\partial_x,
    \quad \text{on $\Dom_\infty=(0,\infty)$.}
\]
We first argue as in \cite[Section~1.6]{tourniaire21} that, for some
functions $W$, the drift $\mu$ can be chosen so that $\mathcal{L}_\infty+r$ is
critical. Note that if we write $g(x) = e^{\mu x}f(x)$ then
\[
    e^{\mu x}(\mathcal{L}_\infty+r)f(x)=\left(\mathcal{A}+\frac{1-\mu^2}{2}\right)g(x), \quad x\in \Dom_\infty,
\]
with
\begin{equation} \label{eq:operatorA}
    \mathcal{A} \coloneqq \frac{1}{2}\partial_{xx}+\frac{1}{2} W,
    \quad \text{on} \quad \Dom_\infty=(0,\infty).
\end{equation}
We can distinguish three cases:
\begin{enumerate}[(i)]
    \item If $\mathcal{A}$ possesses a Green's measure -- in which case
        we say that $\mathcal{A}$ is \emph{subcritical} -- then the
        operator $\mathcal{L}_\infty+r$ is also subcritical for any
        choice of $\mu \geq 1$.
    \item If $\mathcal{A}$ is critical, then so is $\mathcal{L}_\infty+r$ for $\mu=1$.
    \item If $\mathcal{A}$ is neither subcritical nor critical, it is
          said to be \emph{supercritical} and it follows from
          \cite[Theorem~4.6.7]{pinsky95} that
          \begin{equation}
              \label{eq:A_critical}
              \exists ! \;
              \lambda_\infty>0\quad \text{such that} \quad
              \mathcal{A}-\lambda_\infty \quad  \text{is critical}
          \end{equation}
          and
          \begin{equation}
              \label{def:mu}
              \mathcal{L}_\infty+r \quad \text{critical} \iff
              \mu=\sqrt{1+2\lambda_\infty}.
          \end{equation}
\end{enumerate}

In case~(iii), let $v_\infty$ denote the positive harmonic function for
$\mathcal{A}-\lambda_\infty$ with $v_\infty(1) = 1$. On $[1,\infty)$,
\begin{equation}
    \label{eq:decay_v}
    v_\infty(x)=e^{-\beta(x-1)}, \quad \text{with}
    \quad \beta \coloneqq \sqrt{2\lambda_{\infty}}.
\end{equation}
The set of positive  harmonic functions for $\mathcal{L}_\infty+r$
and for $\tilde{\mathcal{L}}_\infty+r$ on $\Dom_\infty$ then consist,
respectively, of positive multiples of
\begin{equation} \label{eq:h_inf}
    h_\infty(x) = c_{h} e^{\mu x}v_\infty(x), \qquad\text{and}
    \qquad \tilde{h}_\infty(x) = c_{\tilde h} e^{-\mu x}
    v_\infty(x),
\end{equation}
where the constants $c_{h}$ and $c_{\tilde h}$ are chosen so that
$\angle{\tilde{h}_\infty, 1} = 1$ and $\angle{\tilde{h}_\infty, h_\infty} = 1$.
With this definition, we see from \eqref{eq:decay_v} that
there exists $c > 0$ such that
\begin{equation} \label{def:alpha}
    \int_0^\infty \indic_{\{ h_\infty(x) < z\}} \tilde{h}_\infty(x) \diff x
    \sim
    \frac{c}{z^\alpha},\qquad
    \text{with} \quad  \alpha \coloneqq \frac{\mu+\beta}{\mu-\beta}.
\end{equation}
Therefore, the BBM has a spatial power-law tail in the sense of
Assumption~\ref{hyp:alpha_tail} precisely when
\begin{equation}\label{Hwp}
    \alpha\in(1,2) \Longleftrightarrow \mu> 3\beta  \;
    {\Longleftrightarrow  \mu\in \left(1,\frac{3}{2\sqrt{2}}\right)}, \tag{$H_{sp}$}
\end{equation}
and in this case, we say that the BBM is in the semipushed
regime. (Note that under \eqref{Hwp}, the operator $\mathcal{L}_\infty+r$
also satisfies Assumption~\ref{hyp:L1criticality}.)

\begin{rem}
    In case (iii), $\lambda_\infty$ is the \emph{generalised principal
        eigenvalue} of the operator $\mathcal{A}$ on $\Omega_\infty$. In
    cases (i) and (ii), this principal eigenvalue is actually
    $\lambda_\infty = 0$, see \cite[Theorem 4.3.2]{pinsky95}.
\end{rem}

\begin{theorem} \label{th:cv_mmm}
    Under \eqref{Hwp}, the conditions of Theorem~\ref{thm:main-cv} hold
    for the BBM for $\alpha \in (1,2)$ given by~\eqref{def:alpha}.
\end{theorem}

\begin{rem}
    Our result derives the limit of the genealogy when the process starts
    from a single particle. It should not be hard from our estimates to
    extend this result to a front-like initial configuration, for which
    the process starts from a large number of particles. In turn, this
    can be used to derive the finite-dimensional convergence of the
    population size to an $\alpha$-stable CSBP as in
    \cite[Theorem~1.2]{tourniaire21}.
\end{rem}

\subsubsection{Related literature and stochastic fronts}
\label{sec:noisyFront}

Branching Brownian motions with absorption have been extensively studied
in the probabilistic literature, in particular when $W \equiv 0$
\cite{kesten1978branching, harris2006further, Harris2007SurvivalBBM,
    berestycki2011survival, berestycki13, berestycki2014critical,
    berestycki2015critical, maillard20}. The version of the model with $W
    \not\equiv 0$ that we consider here was introduced more recently in
\cite{tourniaire21} as a particle system that approximates the front of
the following stochastic partial differential equation
\begin{equation} \label{eq:spde}
    \partial_t u = \partial_{xx} u + u(Bu + 1)(1 - u) +
    \sqrt{\tfrac{u(1-u)}{N}} \xi,
\end{equation}
for some space-time white noise $\xi$. We refer to the discussion in
\cite{tourniaire21, schertzer2023spectral} for the connexion between the
two models. Noisy reaction--diffusion equations as \eqref{eq:spde} have
received considerable attention, we refer to~\cite{panja2004effects} for background.
Let us simply mention that Birzu, Hallatschek, and Korolev
\cite{birzu2018fluctuations, birzu2021genealogical} recently identified a
new phase transition for \eqref{eq:spde} at $B = 4$. Using simulations
and analytical arguments, they further showed that this transition
impacts the shape and timescale of the genealogy of the population
\cite{birzu2021genealogical}: the genealogy is expected to be binary when
$B > 4$ (the fully pushed regime) and to have multiple mergers when $B
\in (2, 4)$ (the semipushed regime). In the fully pushed regime,
convergence to a binary genealogy has already been rigorously established
for a Moran-type model \cite{etheridge2022genealogies} and for our BBM
when $\alpha > 2$ in \eqref{def:alpha} \cite{schertzer2023spectral}.
The original motivation for our work was to show rigorously the results
of \cite{birzu2021genealogical} in the semipushed regime for the BBM
model.

In the special case when $W$ is a step function, \cite{tourniaire21}
already showed that the number of particles in the BBM converges to an
$\alpha$-stable CSBP when $\alpha \in (1, 2)$. Our work improves upon
this result in several directions. Most obviously,
Theorem~\ref{th:cv_mmm} yields the convergence of both the number of
particles and the genealogy of the system to those of an $\alpha$-stable
branching process. Hence, this verifies rigorously the results of
\cite{birzu2021genealogical} in the semipushed case. Moreover, the proof
of Theorem~\ref{th:cv_mmm} relies on a completely different set of
arguments from those in \cite{tourniaire21}. This new approach is more
robust and allowed us to take a first step towards describing the
universality class of this model by considering a whole class of
functions $W$. Lastly, a large share of our arguments are fully general.
This led us to conjecture that the phenomena uncovered by
\cite{birzu2018fluctuations, birzu2021genealogical} are generic features
of spatial particle systems which are not restricted to stochastic
fronts. We gave in \eqref{eq:fundamental-condition} a characterisation of
the exponent that should drive the dynamics of a general branching
diffusion, and provided in Theorem~\ref{thm:main-cv} a set of moment
conditions under which the genealogy does converge.

Finally, let us mention that \cite{berestycki13} proved convergence of
the genealogy of the BBM to the Bolthausen--Sznitman
coalescent when $W \equiv 0$ (which corresponds to the pulled regime of
\eqref{eq:spde}). We believe that this result should extend to any $W$
such that the operator $\mathcal{A}$ in \eqref{eq:operatorA} is
subcritical (Case~(i) above). Note that the operator $\mathcal{A}$
becomes critical (Case~(ii) above) exactly at the transition between the
pulled and the pushed regime ($B = 2$ in \eqref{eq:spde}), which is
referred to as the pushmi-pullyu case in PDE theory
\cite{an2023quantitative}. See also the work of \cite{derrida2023cross}
for the critical window around this transition.

\subsection{Heuristics for the BBM moments}
\label{sec:heuristics_BBM}

\begin{figure}[t]
    \centering
    \includegraphics[width=.5\textwidth]{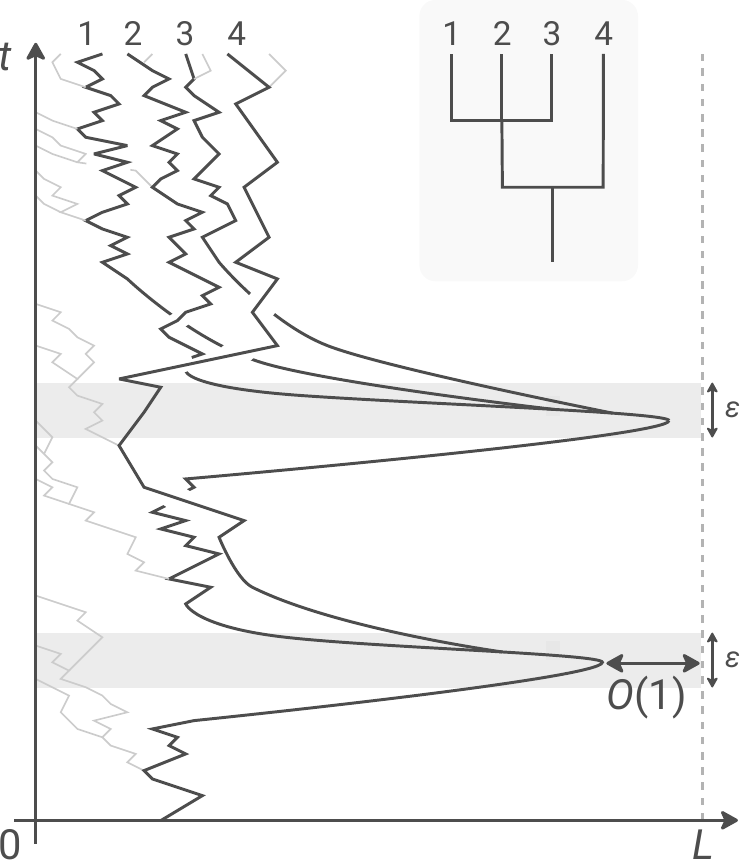}
    \caption{Illustration of the behaviour of the BBM. The ancestral
    lineages of particles sampled at time $t$ typically remain close to
    the absorbing boundary except at coalescence times. Forward in time,
    these events correspond to a particle reaching a distance of
    order $L$ from the origin and producing a large offspring in a short
    amount of time (of order $\epsilon$). Recall that $L$ is chosen in
    such a way that $h(L)$ is of order $N^\gamma$. To capture these, we
    contract the branch points occuring on a time window of lenght
    $\epsilon$ after this event, leading to the genealogy on the top. 
    }
    \label{fig:intuition}
\end{figure}

In this section we give the main ideas behind the proof of
Theorem~\ref{th:cv_mmm}.  In view of
Theorem~\ref{thm:main-cv}, the bulk of the technical work in showing
convergence to an $\alpha$-stable genealogy is to compute the limit of
the moments of the BBM with a cutoff, see Section~\ref{sec:sufficient_condition}.
Let $A \ge 1$ and $N > 0$. It is a
direct computation using \eqref{eq:h_inf} that the truncated domain of
the BBM is an interval $\Omega_{A N^\gamma} = (0, L_{(N,A)})$ and that, up to
a constant shift,
\begin{equation} \label{def:LA0}
    L \equiv L_{(N,A)} \coloneqq \frac{1}{2\beta} \log(N)
    + \frac{1}{\mu-\beta} \log(A).
\end{equation}
The eigenelements of the BBM on $(0, L)$ are defined in terms of the
solutions to the following Sturm--Liouville problem
\begin{equation}\tag{SLP}\label{SLP}
    \frac{1}{2} u''(x) + \frac{1}{2} W(x)u(x) = \lambda u(x),
\end{equation}
with  boundary conditions
\begin{equation}\tag{BC}\label{BC}
    u(0) = u(L) = 0.
\end{equation}
Let $\lambda \equiv \lambda_{(N,A)}$ and $v \equiv v_{(N,A)} \ge 0$
be respectively the maximal eigenvalue and eigenfunction (see
\cite[Chapter 4]{zettl10}) of \eqref{SLP} with boundary conditions
\eqref{BC}. 
Recall the definition of $\lambda_\infty$ from \eqref{eq:A_critical} and
let
\begin{equation} \label{eq:w}
    w \equiv w_{(N,A)} \coloneqq \lambda_{\infty} - \lambda.
\end{equation}
We further define
\begin{equation} \label{eq:def_h}
    h(x) \equiv h_{(N,A)}(x) = c_{h}e^{\mu x} v(x),
    \qquad
    \tilde{h}(x) \equiv \tilde{h}_{(N,A)}(x) = c_{\tilde{h}} e^{-\mu x} v(x),
    \quad x \in [0, L],
\end{equation}
where $v$ is renormalized in such a way that $v(1) = 1$ and $c_{h}$,
$c_{\tilde h}$ are as in \eqref{eq:h_inf}. It is not hard to check that
$\tilde{h}$ and $h$ are the left and right eigenfunctions, respectively,
of the generator of the BBM with cutoff at $L$, corresponding to the
eigenvalue $-w$.
The spectral
analysis of \eqref{SLP} carried out in \cite{schertzer2023spectral} shows
(Lemma~\ref{lem:spectral}) that there exists $c_w > 0$ such that
\begin{equation} \label{eq:heuristicW}
    \lim_{N \to \infty} Nw = \frac{c_w}{A^{\alpha-1}}.
\end{equation}
A large part of the argument is dedicated to proving that for every
$A \ge 1$,
\begin{equation} \label{eq:convergenceMoment2}
    \forall k \geq 1,\: t \ge 0, \quad
    \lim_{N \to \infty} \hat{\mathbf{M}}^{k,t}_{x} =
    \hat{\mathcal{M}}^{k,t}_{A} \otimes
    (h_\infty \tilde{h}_\infty)^{\otimes k},
\end{equation}
where $h_\infty$, $\tilde h_\infty$ are as in \eqref{eq:h_inf},
$\hat{\mathcal{M}}^{k,t}_A$ is the $k$-th planar moment of the
$\psi_A$-mm space with
\[
    \psi_{A}(\theta)
    = \frac{c_{w}}{A^{\alpha-1}}\theta
    +\int_{(0,\infty)} \big( e^{-\theta z} - 1 + \theta z \big)
    \Lambda_{A}(\diff z), \quad \theta \geq 0,
\]
and $\Lambda_A$ is the pushforward of some measure $\Lambda$ by a dilatation map
as defined in \eqref{eq:rescaled_lambda}.
The aim of this section is to outline the main steps of the proof of
\eqref{eq:convergenceMoment2}.

\subsubsection{Capturing the multiple mergers}

One of the difficulties of our model is that the higher order branchpoints
in the limit emerge from the accumulation of binary branchpoints and are
not directly visible in the BBM. In order to circumvent this issue,
we consider a modification of the moments of the BBM for which all
branchpoints that occur less than $\epsilon$ unit of time after the first
branchpoint $\tau$ are merged, see the picture in Figure~\ref{fig:mergingMap}.
Formally, this is done by defining a map $S^\epsilon = (S^\epsilon_{ij})_{i,j} 
\colon \mathbb{U}_k \to \mathbb{U}_k$ on the space of ultrametric
matrices as follows
\begin{equation} \label{eq:mergedMap}
    \forall U \in \mathbb{U}_k,\qquad
    S_{ij}^\epsilon(U) =
    \begin{dcases}
        U_{ij}  & \text{if $U_{ij} < \tau(U) - \epsilon$,} \\
        \tau(U) & \text{else.}
    \end{dcases}
\end{equation}
As long as $\epsilon \ll N$, the branchpoints which are merged by the map
$S^\epsilon$ eventually accumulate to a single multiple merger on the
timescale $N$. We will establish \eqref{eq:convergenceMoment2} by
computing the limit of the moments of the BBM evaluated against
functionals of the form $G \circ S^{\epsilon / N}$, where $G$ is of the
product form
\begin{equation} \label{eq:productFuncMarks}
    \forall U^* \in \mathbb{U}^*_k,\quad
    G(U^*) = \indic_{\{ c(U) = c \}} f(\tau(U)) \prod_{p=1}^{\abs{c}} F_p(U^*_p)
\end{equation}
similarly to equation \eqref{eq:product-form1}. Using the branching
property of the BBM, we will show (Proposition~\ref{cor:many-to-few})
that
\begin{multline} \label{eq:branching:rn0}
    \hat{\mathbf{M}}_x^{k,t}[G \circ S^{\frac{\epsilon}{N}}]
    = \frac{1}{N^{\gamma(\abs{c}-\alpha)}}
    \int_{\epsilon/N}^{t} f(u) e^{-Nw(t-u)}
    \E_x\bigg(
    r(\bar{\zeta}_{t-u}) h(\bar{\zeta}_{t-u})
    \\ \times
    \sum_{n=1}^{\abs{c}-1} \mathbf{M}_{\bar{\zeta}_{t-u}}^{n,\eps}\bigg[
        \prod_{i=1}^n \hat{\mathbf{M}}^{c_i,u-\frac{\eps}{N}}_{X_i}[F_i]
        \bigg]
    \mathbf{M}_{\bar{\zeta}_{t-u}}^{\abs{c}-i,\eps}\bigg[
    \prod_{i=1}^{\abs{c}-n} \hat{\mathbf{M}}^{c_{i+n},u-\frac{\eps}{N}}_{X_i}[F_{i+n}]
    \bigg] \bigg) \diff u,
\end{multline}
where $\bar{\zeta}_s = \zeta_{Ns}$ is the spine process of the BBM with
cutoff with time accelerated by a factor $N$. (We recall that this spine
process is obtained via the Doob $h$-transform~\eqref{eq:spine-cut}, with
$h$ defined in~\eqref{eq:def_h}.)

This recursion should be compared with that for the planar moments of the
limiting $\psi_A$-mm space in \eqref{moments:CSNP}. We compute the limit as
$N \to \infty$ of \eqref{eq:branching:rn0} in two steps.

\begin{figure}
    \centering
    \includegraphics[width=.7\textwidth]{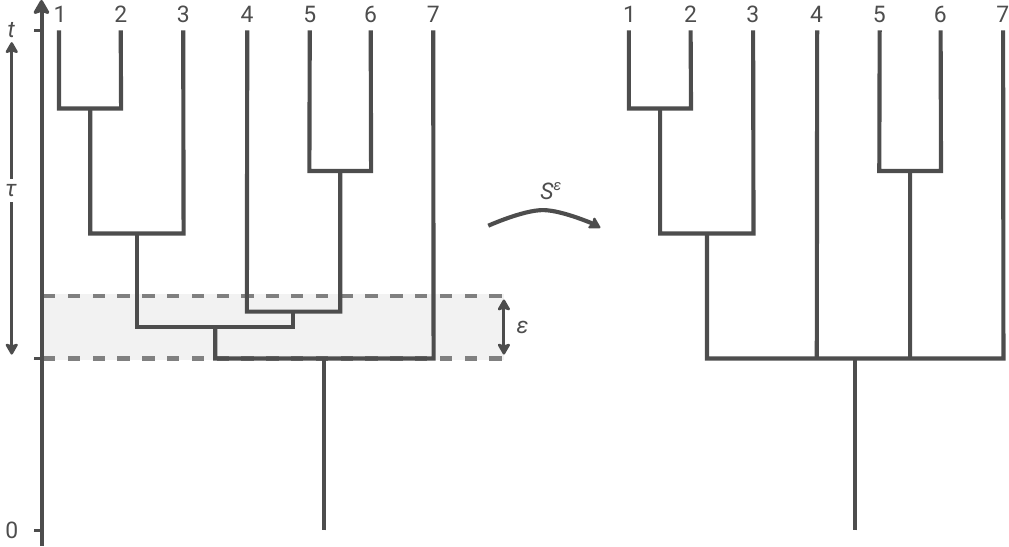}
    \caption{Illustration of the map $S^{\epsilon}$. All branchpoints
        in the original tree (left) whose distance to the leaves lies within
        $[\tau, \tau+\epsilon]$ are merged to a single branchpoint at $\tau$
        (right).}
    \label{fig:mergingMap}
\end{figure}

\subsubsection{Step 1: Fast mixing of the spine process}
\label{sec:step1}

According to \eqref{eq:branching:rn0}, the $k$-th moment of the BBM
depends on the initial position $x$ only through the value of the
accelerated spine process $\bar{\zeta}_s$ for $s>0$. We will show in
Section~\ref{sec:spectral_theory} that the spine process relaxes to its
equilibrium distribution on the timescale of order $\log(N)\ll N$. More
precisely, if we denote by $q_t$ the transition kernel of the spine
process, then
\[
    q_t(x,y) \approx \Pi(y) \quad \text{for } t \gg \log(N),
\]
with
\begin{equation} \label{eq:Pi}
    \Pi(y)\, \diff y = h(y)\tilde h(y)\, \diff y
    = c_h c_{\tilde h} v(y)^2 \diff y. 
\end{equation}
Thus, we replace $\bar{\zeta}_{t-u}$ in \eqref{eq:branching:rn0} by a
random variable distributed according to $\Pi$ to obtain the
approximation that, for any $\mathbf{c} \in (0,1)$,
\begin{equation} \label{eq:heuristicInitialCondition}
    \hat{\M}^{k,t}_x[F] \approx
    \hat{\M}^{k,t}_\Pi[F] \coloneqq
    \int_0^L \hat{\M}^{k,t}_y[F] \Pi(y) \diff y,\qquad
    x \in (0, \mathbf{c}L).
\end{equation}
Note that the starting point $x$ needs to be far from the killing
boundary at $L$ because $\mathbf{c} < 1$.

Indeed, the approximation \eqref{eq:heuristicInitialCondition}
breaks down if $x$ is too close from the right boundary, since
$L$ was chosen so that a particle started close to $L$ releases a number
of particles of the order $N^\gamma$ before it returns close to the
left boundary at $0$. In turn, we will use
\eqref{eq:heuristicInitialCondition} to deduce that
\[
    \forall y\in[0,L], \quad \mathbf{M}_{y}^{j,\eps}\left[
        \prod_{i=1}^j \hat{\mathbf{M}}^{c_i,u-\frac{\eps}{N}}_{X_i}[F_i]
        \right] \ \approx  \ \mathbf{M}_{y}^{j,\eps}[1]
    \prod_{i=1}^j \hat{\mathbf{M}}^{c_i,u}_{\Pi}[F_i].
\]
For this approximation to hold, we need that $\epsilon$ is large enough
so that the typical particles of a BBM started close to the right
boundary are far from $L$ (in $(0, \mathbf{c}L)$ for some $\mathbf{c}
< 1$) at time $\epsilon$. Since spine particles have a drift to the left, we
will see (Lemma~\ref{lem:boundSpineProcess2}) that it is sufficient to let
$\epsilon = a L$ for some $a > 0$.

If for $p \ge 2$ we let
\begin{equation}\label{eq:m_c}
    \hat{m}_{p} \equiv \hat m_{p,(N,A)}
    \coloneqq \frac{p!}{N^{\gamma(p-\alpha)}}
    \int_0^L r(y) h(y)\Pi(y) \sum_{j=1}^{p-1}
    \mathbf{M}_y^{j,\eps}[1] {\mathbf{M}}_y^{p-j,\eps}[1]
    \diff y
\end{equation}
the two approximations above and the estimate for $wN$ in
\eqref{eq:heuristicW} show (Proposition~\ref{lem:approx:1}) that
\begin{equation} \label{eq:approxRecursion}
    \hat{\mathbf{M}}^{k,t}_x[G \circ S^\epsilon]
    \approx \hat{\mathbf{M}}^{k,t}_\Pi[G \circ S^\epsilon]
    \approx
    \frac{\hat{m}_{\abs{c}}}{\abs{c}!}
    \int_0^t f(u) e^{-\frac{c_w}{A^{\alpha-1}}(t-u)}
    \prod_{i=1}^{\abs{c}} \hat{\mathbf{M}}_\Pi^{c_i,u}[F_i] \diff u.
\end{equation}
The latter relation mimics the recursive relation for the moments
of $\psi$-mm spaces as written in Definition~\ref{def:planarMoments}.
\subsubsection{Step 2. The reversed process}
\label{sec:step2}

According to the recursion formula for the moments of a $\psi$-metric
measure space, see \eqref{moments:CSNP}, we now need to identify
$\hat{m}_p$ in \eqref{eq:m_c} with the $p$-th moment of a certain jump
measure. The key observation is that the invariant distribution of the
spine has density proportional to $\tilde h(x)h(x)$, see~\eqref{eq:Pi}.
As a
consequence, the integral in \eqref{eq:m_c} involves integrating $h^2(x)
    \tilde{h}(x)$. From \eqref{eq:h_inf}, for every $x > 1$
\[
    h^2(x) \tilde{h}(x) \to h_\infty^2(x) \tilde h_{\infty}(x)
    = c_h^2 c_{\tilde{h}} e^{(\mu -3\beta)x},
    \qquad \text{as $N \to \infty$.}
\]
In the semipushed regime \eqref{Hwp} we have $\mu - 3\beta > 0$ so that
the mass of the integral is concentrated at the killing boundary $L$. This
reflects that branching points in the BBM are typically close to the
boundary $L$. See Figure~\ref{fig:intuition}.

When estimating $\hat{m}_p$ in \eqref{eq:m_c}, it then becomes natural to
look at the process started from an initial particle close to $L$.
Reversing the spatial axis, this is a BBM on $[0, L]$ with
\emph{positive} drift $\mu$ and branching rate $r(L-x)$. Since $r(L-x) =
\frac{1}{2}$ for $x<L-1$, we will approximate this process by a BBM with
a single killing boundary at $0$ and constant branching rate $1/2$. We
will call this approximation the \emph{reversed branching Brownian
motion} (reversed BBM). The unique boundary of the reversed BBM is at $0$
and corresponds to the right boundary at $L$ in the original BBM. As will
become clear later, the reversed BBM will be a good approximation of the
original BBM started close to $L$ until a time of order $a L$ for $a <
1/\beta$, which will ensure that no particles in the original process
enter the interval $[0,1]$ with high probability 
(see Lemma~\ref{lem:boundSpineProcess}). Combining this constraint with that
from Step~1 leads us to choose $\epsilon = aL$ with $0 < a < 1/\beta$,
and we let
\begin{equation} \label{eq:defEps}
    \epsilon \equiv \epsilon_{(N,A)} \coloneqq \frac{L}{2\beta}.
\end{equation}

We now denote by $\overleftarrow{\mathcal{N}_t}$ the particles alive at
time $t$ in the reversed process, and by $\overleftarrow{X_v}(t)$ the
position of a particle $v \in \overleftarrow{\mathcal{N}_t}$ at time $t$.
Define for $z \ge 0$
\begin{equation}\label{eq:notation_reversed}
    \la{v}(z) = 2 e^{\beta} \sinh(\beta z),\quad
    \la{h}(z) = c_h \la{v}(z) e^{-\mu z},\quad
    \la{\widetilde{h}}(z) = c_{\tilde{h}} \la{v}(z) e^{\mu z},\quad
    \la{\Pi}(z) = \la{\widetilde{h}}(z) \la{h}(z)
\end{equation}
with $c_h$ and $c_{\tilde{h}}$ as in \eqref{eq:h_inf}.
All these quantities play an analogous role to their corresponding
objects for the forward BBM; for instance $\overleftarrow{h}$ can
be interpreted as a harmonic function for the reversed BBM.
A direct computation (Lemma~\ref{lem:relationReverseEigenelements})
will reveal that
\[
    h(L-z) \approx AN^\gamma \la{h}(z),\qquad \Pi(L-z) \approx
    (AN^\gamma)^{1-\alpha} \la{\Pi}(z).
\]
Using this estimate, by coupling the process started from $L-z$ with the
reversed BBM started from $z$ we will show
(Lemma~\ref{lem:comparisonReversed}) that
\[
    \M^{p,t}_{L-z}[1]
    =
    \frac{1}{h(L-z)}
    \E_{L-z}\Bigg[
        \sum_{v_1 < \dots < v_p \in \mathcal{N}_t}
        \prod_{i=1}^p h(X_{v_i}(t))
        \Bigg]
    \approx
    (AN^\gamma)^{p-1} \la{\M}^{p,t}_z[1],
\]
where $\la{\M}^{p,t}_z[1]$ is the $p$-th planar moment of the reversed
BBM defined analogously to the original BBM as
\begin{equation} \label{eq:biasedReversedMoment}
    \la{\M}^{p,t}_z[1] \coloneqq
    \frac{1}{\la{h}(z)}
    \la{\E}_z\Bigg[
        \sum_{v_1 < \dots < v_p \in \la{\mathcal{N}}_t}
        \prod_{i=1}^p \la{h}(\la{X_v}(t))
        \Bigg].
\end{equation}
Making a change of variable $y = L - z$ in \eqref{eq:m_c}, this
approximation will entail (Proposition~\ref{lem:estimateBranchingMoment})
that
\begin{equation} \label{eq:heuristic-1}
    \hat{m}_p \approx A^{p-\alpha} \frac{p!}{2}
    \int_0^\infty \la{h}(z) \la{\Pi}(z) \sum_{j=1}^{p-1}
    \la{\M}^{j,\infty}_z[1] \la{\M}^{p-j,\infty}_z[1] \diff z.
\end{equation}

\subsubsection{Identifying the jump measure}
\label{sec:step3}

The previous two steps have established that
\[
    \hat{\M}_x^{k,t}[G \circ S^{\frac{\epsilon}{N}}]
    \approx
    \frac{\hat{m}_{\abs{c}}}{\abs{c}!}
    \int_0^t f(u) e^{-\frac{c_w}{A^{\alpha-1}}(t-u)}
    \prod_{i=1}^{\abs{c}} \hat{\mathbf{M}}_\Pi^{c_i,u}[F_i] \diff u
\]
with $\hat{m}_{\abs{c}}$ given in \eqref{eq:heuristic-1}.
Comparing this expression with that for the moments of a $\psi$-mm space
in \eqref{moments:CSNP}, the result will follow by induction provided
that we can find a measure $\Lambda$ such that $\hat{m}_p$ is the $p$-th
moment of the dilated measure $\Lambda_A$ defined in
\eqref{eq:rescaled_lambda}. Since, by definition of $\Lambda_A$,
\[
    \int_{(0,\infty)} z^p \Lambda_A(\diff z)
    = A^{p-\alpha} \int_{(0,\infty)} z^p \Lambda(\diff z),
\]
we are left with finding $\Lambda$ such that
\[
    \int_{(0,\infty)} z^p \Lambda(\diff z) = \frac{p!}{2}
    \int_0^\infty \la{h}(z) \la{\Pi}(z) \sum_{j=1}^{p-1}
    \la{\M}^{j,\infty}_z[1] \la{\M}^{p-j,\infty}_z[1] \diff z.
\]
By computing the moments of the reversed BBM, we shall see in
Proposition~\ref{prop:limitReversed} that the appropriate measure
$\Lambda$ can be defined as the following limit in the vague topology
\[
    \Lambda \coloneqq \lim_{z \to \infty}
    \frac{\beta}{2} \la{\widetilde{h}}(z) \mathscr{L}_z(\la{W}_\infty),
\]
where $\la{W}_\infty$ is the limit of the additive martingale of the
reversed BBM defined as
\begin{equation} \label{eq:reversedMartingale}
    \forall t \ge 0,\quad \la{W_t} =
    \sum_{v \in \la{\mathcal{N}}_t} \la{h}\big( \la{X_v}(t) \big).
\end{equation}
Intuitively, the measure $\Lambda$ is obtained by picking a particle from
the ``bulk'' of the population, that is according to $\la{\tilde{h}}$,
and looking at the number of particles that it releases in the long term,
which is given by the value of $\la{W}_\infty$.

\begin{table}[t]
    \centering\renewcommand{\arraystretch}{1.2}
    \begin{tabularx}{\textwidth}{ccC}
        \toprule
        Notation
        & Definition 
        & Description\\
        \midrule

        $\lambda_\infty$, $v_\infty$
        & \eqref{eq:operatorA}
        & Principal eigenvalue, eigenfunction of
        $\mathcal{A}$ \\

        $\mu = \sqrt{1 + 2\lambda_\infty}$                         
        & \eqref{def:mu}
        & Drift of the particles in the BBM \\

        $\beta = \sqrt{2 \lambda_\infty}$ 
        & \eqref{eq:decay_v}
        & Drift of the spine process $(\zeta_{t,\infty})_{t \ge 0}$ \\

        $\alpha = \frac{\mu+\beta}{\mu-\beta}$ 
        & \eqref{def:alpha}
        & Exponent of the spatial power law \\

        $\gamma = \frac{1}{\alpha-1}$ 
        & \eqref{eq:survProbGW}
        & The number of particles at time $N$ scales as $N^\gamma$ \\

        $\epsilon = \frac{1}{2\beta} L$
        & \eqref{eq:defEps}
        & Timescale over which branch points accumulate \\

        $h_\infty(x) = c_h e^{\mu x} v_\infty(x)$ 
        & \eqref{eq:h_inf}
        & Principal right eigenfunction of $\mathcal{L}_\infty + r$ \\

        $\tilde{h}_\infty(x) =c_{\tilde h} e^{-\mu x} v_\infty(x)$ 
        & \eqref{eq:h_inf} 
        & Principal left eigenfunction of $\mathcal{L}_\infty + r$ \\

        $\Pi_\infty(x) = h_\infty(x) \tilde{h}_\infty(x)$
        & \eqref{eq:def_Pi_g}
        & Stationary distribution of the spine process $(\zeta_{t,\infty})_{t \ge 0}$ \\

        $(\mathcal{N}_{t,\infty}, d_{t,\infty}, \nu_{t,\infty})$
        & Section~\ref{sect:gen-structure}
        & Mmm-space encoding the genealogy of the BBM \\

        $(\bar{\mathcal{N}}_{t,\infty}, \bar{d}_{t,\infty}, \bar{\nu}_{t,\infty})$
        & \eqref{eq:rescalingMmmSpace}
        & Rescaled mmm-space of the BBM \\

        $L \equiv L_{(N,A)}$
        & \eqref{def:LA0}
        & Length of the cutoff interval \\

        $\lambda, v$
        & \eqref{SLP}
        & Principal eigenvalue, eigenfunction of \eqref{SLP} \\

        $w = \lambda_\infty - \lambda$
        & \eqref{eq:w}
        & Subcriticality parameter of the BBM on $(0, L)$ \\

        $h(x) = c_h e^{\mu x} v(x)$
        & \eqref{eq:def_h}
        & Principal right eigenfunction of $\mathcal{L} + r$ on $(0,L)$ \\

        $\tilde{h}(x) = c_{\tilde{h}} e^{-\mu x} v(x)$
        & \eqref{eq:def_h}
        & Principal left eigenfunction of $\mathcal{L} + r$ on $(0,L)$ \\

        $\Pi(x) = h(x)\tilde{h}(x) $
        & \eqref{eq:defPi}
        & Stationary distribution of $(\zeta_t)_{t \ge 0}$ \\

        $(\bar{\mathcal{N}}_{t}, \bar{d}_{t}, \bar{\nu}_{t})$
        & Section~\ref{sect:cut-off}
        & Rescaled mmm-space of the BBM on $[0, L]$ \\

        $\mathbf{M}^{k,t}_{x}$ 
        & \eqref{eq:biasedMomentBBM} 
        & Planar moments of the mmm-space of the BBM with cutoff \\

        $\hat{\mathbf{M}}^{k,t}_{x}$ 
        & \eqref{eq:scaled_moments} 
        & Rescaled version of the moments $\mathbf{M}^{k,t}_{x}$ \\

        $\hat{m}_p$ 
        & \eqref{eq:m_c}
        & Approximate moment of the jump measure \\

        $(\mathcal{U}_t, d_t, \vartheta_t)$
        & Definition~\ref{def:psiMmSpace}
        & $\psi$-mm space \\

        $\hat{\Mc}^{k, t}$ 
        & \eqref{eq:momentJumpMeasure1}
        & Planar moments of the $\psi$-mm space \\

        $\la{v}$, $\la{h}$, $\la{\tilde{h}}$, $\la{\Pi}$
        & \eqref{eq:notation_reversed}
        & Eigenelements of the reversed BBM \\

        $\la{\M}^{k,t}_z[1]$ 
        & \eqref{eq:biasedReversedMoment}
        & Moments of the reversed BBM \\
        \bottomrule
    \end{tabularx}
    \caption{Main notation for the BBM. 
    }
\end{table}

\section{The metric and sampling structure of CSBPs}
\label{sect:geneaalogy_CSBP}

The objective of this section is to construct the genealogy of a CSBP,
the $\psi$-metric measure space $(\mathcal{U}_t, d_t, \vartheta_t)$, and
to give an expression for its moments. We construct the genealogy as the
boundary of a time-inhomogeneous branching process called a reduced process.
We carry the construction out and compute the moments for general
reduced processes, and specialise these results to CSBPs in
Section~\ref{sec:proofCSBP}.

\subsection{Marked metric measure spaces}
\label{sec:mmmSpaces}

In this section, we recall the minimal amount of theory on metric
measure spaces needed to construct the limit and to prove convergence to
it. We refer to \cite{depperschmidt_marked_2011, Greven2009} for
extensive accounts on the general theory.

A marked metric measure space is a triple $(U, d, \nu)$ with $(U,
d)$ a complete separable metric space and $\nu$ a finite measure on
$U \times \Omega_\infty$, that encodes the spatial locations of the
elements of $U$. All the metric spaces that we will consider here
are ultrametric, and so we use the letter $U$ to denote them. We define
the marked Gromov-weak topology as the topology induced by the
functionals
\begin{equation} \label{eq:defPolynomial}
    \Phi \colon (U, d, \nu) \mapsto \int_{(U \times \Omega_\infty)^k}
    F\big( d(\mathbf{u}), \mathbf{x} \big)
    \nu^k(\diff \mathbf{u}, \diff \mathbf{x}),
\end{equation}
where $F \colon \R^{k \times k}_+ \times \Omega_\infty^k \to \R$ is continuous
bounded and with the notation
\[
    \mathbf{u} = (u_i)_{1 \le i \le k},\qquad
    d(\mathbf{u}) = (d(u_i, u_j))_{1 \le i,j \le k},\qquad
    \mathbf{x} = (x_i)_{1 \le i \le k}.
\]
Such a functional is referred to as a polynomial in \cite{Greven2009}.
Two mmm-spaces are indistinguishable in the marked Gromov-weak topology
if and only if their supports are in measure and mark preserving isometry
\cite[Theorem~1]{depperschmidt_marked_2011}, in which case they are
called isomorphic. In order for the marked Gromov-weak topology to be
separated (and actually Polish), we formally work with isomorphic classes
of mmm-spaces without mentioning it in practice.

Although marks are required to represent the location of the particles in
a branching diffusion, the genealogy of a CSBP is more naturally
constructed as plain ultrametric space (without marks). Metric spaces
with no marks can be fit into the framework above by assuming that the
mark space $E$ be made of a single element. In this case, any measure on
$U \times E$ can be identified with its projection on $U$, and we say
that $(U, d, \nu)$ is an ultrametric measure space. We simply refer to the
corresponding topology as the Gromov-weak topology.

Finally, the space of all marked ultrametric measure spaces is endowed
with its Borel $\sigma$-algebra and with a corresponding notion of weak
convergence of finite measures defined on that $\sigma$-algebra. Here, we
will need to work with a weaker notion of \emph{vague} convergence of
measures, formally introduced in \cite{foutel24vague} for instance. A
sequence of (finite) measures $(M_n)_{n \ge 1}$ on the space of
mmm-spaces converges vaguely to some limit $M$ if its restriction to the
set $\{ (U,d,\nu) : \nu(U \times \Omega_\infty) > a \}$ converges weakly
for a.e.\ $a > 0$. Finally, we denote by $\mathscr{L}(U,d,\nu)$ the law
of a random mmm-space $(U,d,\nu)$.

\subsection{Reduced processes and their boundary}
\label{sec:reducedProcess}

\paragraph{Definition.}
Consider a collection $(r_{i,\tau})_{i \ge 2, \tau > 0}$ of rates, and a
fixed time horizon $t > 0$. Throughout this section, we will use $s$ for
a time variable going in the natural direction of time, and $\tau = t-s$
for the time remaining until the horizon $t$, which is interpreted as the
distance to the leaves. We consider a branching particle system in which
each particle at time $t-\tau$ branches independently into $i \ge 2$
daughter particles at rate $r_{i,\tau}$. It will be convenient to
introduce the notation
\[
    \forall \tau > 0,\qquad r_\tau = \sum_{i \ge 2} r_{i,\tau},\qquad
    K_\tau \sim \Big(\frac{r_{i,\tau}}{r_\tau}\Big)_{i \ge 2}, \qquad
    m_\tau = \sum_{i \ge 2} (i-1) r_{i,\tau}
\]
for the total branching rate $r_\tau$, the number of offspring $K_\tau$,
and the expected jump size $m_\tau = r_\tau(\E[K_\tau]-1)$. We assume that
\[
    \forall \tau > 0,\quad \int_\tau^t m_x \diff x < \infty,
\]
which prevents the population from exploding before time $t$.

As is usually done, the branching process before time $t$ can be
constructed as a random tree $\mathcal{T}$, which is a random subset of
the Ulam--Harris labels
\[
    \mathcal{V} \coloneqq \{ \varnothing \} \cup \bigcup_{n = 1}^\infty
    \N^n.
\]
For $v, w \in \mathcal{V}$, we let $vw$ be the concatenation of $v$ and
$w$; and $v \wedge w$ be their most-recent common ancestor, that is the
longest $u$ such that $v = uv'$ and $w = uw'$ for some $v', w' \in
    \mathcal{V}$. This tree is endowed with (continuous) edge lengths, which
we represent as a collection of birth times $(\sigma_u)_{u \in
            \mathcal{T}}$ and life lengths $(\omega_u)_{u \in \mathcal{T}}$. We will
further denote by
\[
    \forall s < t,\qquad \mathcal{N}_{s,t} \coloneqq
    \{ v \in \mathcal{T} : s \in [\sigma_v, \sigma_v + \omega_v) \},
    \qquad Z_{s,t} \coloneqq \Card \mathcal{N}_{s,t}
\]
the set of individuals alive at time $s$ and the population size at that
time. It will be important to note that the process $(e^{- \int_0^s
            m_{t-y} \diff y} Z_{s,t})_{s \in [0,t)}$ is a non-negative martingale.

\paragraph{Constructing the boundary.}
If $Z_{s,t} \to \infty$ as $s \uparrow t$, the boundary of
$\mathcal{T}$ is a continuous set. In this case, we define the genealogy
of the population as a random mm-space obtained by letting $s \uparrow t$
as follows. Define the genealogical distance of the population at time $s
    < t$ as
\[
    \forall v, w \in \mathcal{N}_{s,t},\quad d_{s,t}(v, w)
    = s - (\sigma_{v \wedge w} + \omega_{v \wedge w}).
\]
(The distance between $v$ and $w$ is their coalescence time.) We endow
each $v \in \mathcal{N}_{s,t}$ with a weight that captures its asymptotic
contribution to the growth of the population. Namely, fix some $v \in
    \mathcal{V}$ and define
\[
    \forall s \in [0, t),\qquad
    Z^v_{s,t} \coloneqq \Card \{ w \in \mathcal{N}_{s,t} : w \succeq v
    \}.
\]
By the branching property, $(Z^v_{s',t})_{s' \in [s,t)}$ is distributed
as $(Z_{s',t-s})_{s' \in [0,t-s)}$ for any $v \in \mathcal{N}_{s,t}$.
Therefore, the following is well-defined as the limit of a non-negative
martingale:
\[
    \forall v \in \mathcal{N}_{s,t},\quad
    W^v_{s,t} \coloneqq \lim_{s' \to t} e^{-\int_0^{s'} m_{t-y} \diff y}
    Z^v_{s',t},
    \qquad
    W^\varnothing_{0,t} \coloneqq \lim_{s' \to t} e^{-\int_0^s m_{t-y} \diff y}
    Z_{s',t}.
\]
Note that these limits satisfy
\begin{align}
    \sum_{v \in \mathcal{N}_{s,t}} W^v_{s,t}
     & = \lim_{s' \to t} e^{-\int_0^{s'} m_{t-a} \diff a}
    \sum_{v \in \mathcal{N}_{s,t}} Z^v_{s',t} \nonumber             \\
     & = \lim_{s' \to t} e^{-\int_0^{s'} m_{t-a} \diff a} Z_{s', t}
    = W^\varnothing_{0,t}.
    \label{eq:additivityMeasure}
\end{align}
This finally allows us to define a measure $\vartheta_{s,t}$ on
$\mathcal{N}_{s,t}$ as
\[
    \vartheta_{s,t} = \sum_{v \in \mathcal{N}_{s,t}} W^v_{s,t} \delta_v.
\]

The triple $(\mathcal{N}_{s,t}, d_{s,t}, \vartheta_{s,t})$ is a random
element of the space of ultrametric measure spaces endowed with the
Gromov-weak topology. The following result shows that we can define
another random ultrametric space $(\mathcal{U}_t, d_t, \vartheta'_t)$ by
letting $s \uparrow t$. The formulation uses Gromov's box metric
$\underline{\Box}_1$, defined for instance in \cite{janson2020gromov}. It
is a distance between metric measure spaces that induces the Gromov-weak
topology, and is equivalent to the more standard Gromov--Prohorov
distance \cite{loehr2013equivalence}.

\begin{proposition} \label{prop:defUMS}
    There exists a random ultrametric measure space $(\mathcal{U}_t, d_t,
        \vartheta_t')$ such that
    \[
        \lim_{s \to t} (\mathcal{N}_{s, t}, d_{s,t}, \vartheta_{s,t})
        = (\mathcal{U}_t, d_t, \vartheta_t')
    \]
    almost surely in the Gromov-weak topology. Moreover the
    following bound holds almost surely
    \begin{equation} \label{eq:boundBoxMetric}
        \underline{\Box}_1\big( (\mathcal{N}_{s,t}, d_{s,t},
        \vartheta_{s,t}), (\mathcal{U}_t, d_t, \vartheta_t') \big)
        \le
        t-s.
    \end{equation}
\end{proposition}

\begin{proof}
    We will show that the sequence is almost surely Cauchy for Gromov's
    $\underline{\Box}_1$ metric, which is a complete metric generating
    the Gromov-weak topology. We make use of the formulation of this
    metric using relations given in \cite[Definition~3.2]{janson2020gromov}.

    Fix $s < s' < t$. We construct a relation
    $R \subseteq \mathcal{N}_{s, t} \times \mathcal{N}_{s', t}$ as
    \[
        (v, v') \in R \iff v \preceq v'
    \]
    as well as a measure $\pi$ on $\mathcal{N}_{s, t} \times \mathcal{N}_{s', t}$
    defined as
    \[
        \pi( \{ (v, v') \} ) = \indic_{\{ v \preceq v'\}} W^{v'}_{s',t}.
    \]
    By definition of $W^v_{s,t}$, the projection of $\pi$ on
    $\mathcal{N}_{s,t}$ (resp.\ $\mathcal{N}_{s',t}$) coincides with
    $\vartheta_{s,t}$ (resp.\ $\vartheta_{s',t}$) and $\pi(R) = W_t$. This checks
    condition (3.3) of \cite[Definition~3.2]{janson2020gromov} for any
    $\epsilon > 0$.

    Let $(v,v'), (w,w') \in R$. We show that
    \[
        \abs{d_{s,t}(v,w) - d_{s',t}(v', w')} \le s' - s.
    \]
    If $v = w$, then $d_{s,t}(v,w) = 0$ and both $v'$ and $w'$ descend
    from the same ancestor $v$ at time $s$ since $(v,v'), (v,w') \in R$.
    Thus $d_{s',t}(v', w') \le s'-s$. Conversely if $v \ne w$, then $v
        \wedge w = v' \wedge w'$ and thus
    \[
        d_{s',t}(v',w') - d_{s,t}(v,w) = s'-s,
    \]
    hence the claim. This shows that condition (3.4) of
    \cite[Definition~3.2]{janson2020gromov} is fulfilled for $\epsilon =
        s'-s$. Therefore $\underline{\Box}_1( \mathcal{N}_{s,t},
        \mathcal{N}_{s',t} ) \le s'-s$ and the sequence is Cauchy almost
    surely. The bound \eqref{eq:boundBoxMetric} is obtained by letting
    $s' \to t$ in the previous inequality.
\end{proof}

\begin{rem}
    One could directly construct $(\mathcal{U}_t, d_t, \vartheta'_t)$
    as the boundary of $\mathcal{T}$ (that is, the set of infinite rays
    $(u_1, u_2, \dots)$) endowed with a measure obtained from
    Caratheodory's extension theorem, as is for instance done in
    \cite[Section~2.1]{Duchamps2018}. This construction would, moreover,
    define a planar ordering on the limiting tree. However, it is not
    straightforward to compute the moments directly on the continuous
    object, whereas our construction will allow us to carry out the
    computation on the discrete trees and then take a limit.
\end{rem}

\subsection{Moments of reduced processes}

Recall from Section~\ref{sec:heuristicMoments} that we defined the $k$-th
moment of the random metric measure space $(\mathcal{U}_t, d_t,
    \vartheta'_t)$ as the measure on $\R_+^{k\times k}$ whose integral
against some measurable test function $F \colon \R_+^{k\times k} \to \R_+$
is
\[
    \E\bigg[ \int_{\mathcal{U}_t^k}
    F\big( d_t(\mathbf{u}) \big) \vartheta'_t{}^{\otimes k}(\diff \mathbf{u})
    \bigg],
    \qquad \mathbf{u} = (u_i)_{i \le k},\quad
    d_t(\mathbf{u}) = (d_t(u_i, u_j))_{i,j \le k}.
\]
The following result gives a recursive formula to compute the $k$-th
moment in terms of moments of lower order, which proves
\eqref{eq:momentPsiMM}.

\begin{proposition}[Moments of $(\mathcal{U}_t, d_t, \vartheta'_t)$]
    \label{cor:momentsReduced}
    Suppose that, for each $t > 0$,
    \begin{equation} \label{eq:conditionMomentReduced}
        \forall k \ge 2,\quad \sup_{0 < s \le t} r_{t-s}\E[K_{t-s}^{(k)}]
        e^{-(k-1)\int_0^s m_{t-x} \diff x} < \infty,
        \qquad
        \int_0^t m_s \diff s = \infty.
    \end{equation}
    Define a family of measure $(\mathcal{M}^{k,t})_{t \ge 0}$
    on $\mathbb{U}_k$ inductively on $k$ as follows. For $k = 1$,
    $\Mc^{1,t}[1] = 1$. For $k \ge 2$ and any functional $G$ of the
    product form \eqref{eq:productFunctional},
    \begin{equation} \label{eq:recursionReduced}
        \Mc^{k,t}[G] = \frac{1}{\abs{c}!}\int_0^t r_{s}\E[K^{(\abs{c})}_s]
        f(s) e^{-(k-1) \int_s^t m_x \diff x}
        \prod_{i=1}^{\abs{c}} \Mc^{c_i,s}[F_i] \diff s.
    \end{equation}
    Then, for any $F \colon \mathbb{U}_k \to \R_+$
    \begin{equation} \label{eq:momentsReduced-1}
        \E\Big[
            \int_{\mathcal{U}_t^k} F\big( d_t(\mathbf{u}) \big)
            {\vartheta'_t}^{\otimes k}(\diff \mathbf{u}) \big)
            \Big]
        =
        \sum_{P} \mathcal{M}^{k,t}[F \circ P],
    \end{equation}
    where $d_t(\mathbf{u}) = (d_t(u_i,u_j))_{i,j\le k}$ and the sum is
    taken over all permutations $P$ of $\{1,\dots,k\}$.
\end{proposition}

This result will follow by deriving a similar formula for the (planar)
moments of the reduced process at time $s < t$, and then letting $s
    \uparrow t$. By analogy with \eqref{eq:biasedMomentBBM}, let us define a
measure on $\mathbb{U}_k$ such that
\begin{equation} \label{eq:defPlanarMoment}
    \Mc^{k,s,t}[F] =
    e^{- k \int_0^s m_{t-y} \diff y}
    \E\bigg[
        \sum_{\substack{v_1, \dots, v_k \in \mathcal{N}_{s,t}\\ v_1 < \dots < v_k}}
        F\big(d_{s,t}(\mathbf{v})\big)
        \bigg]
\end{equation}
for all non-negative $F \colon \mathbb{U}_k \to \R$, and where
\[
    \mathbf{v} = (v_i)_{i \le k},\qquad
    d_{s,t}(\mathbf{v}) = \big( d_{s,t}(v_i, v_j) \big)_{i,j \le k}.
\]

\begin{rem} \label{rem:planarMoments}
    For combinatorial reasons, it is easier to work with planar objects
    and to define moments as measures on $\mathbb{U}_k$. However, our
    construction of $(\mathcal{U}_t, d_t, \vartheta'_t)$ does not endow
    the metric measure space with a notion planar ordering. As an
    intermediate step, we will show that the measure $\mathcal{M}^{k,t}$
    arising in \eqref{eq:momentsReduced-1} is the limit as $s \uparrow t$
    of the planar moment $\mathcal{M}^{k,s,t}$ of the reduced process.
    Thus, $\mathcal{M}^{k,t}$ can be interpreted as a form of planar
    moment of $(\mathcal{U}_t, d_t, \vartheta'_t)$ and
    \eqref{eq:momentsReduced-1} simply indicates that the (unplanar)
    moments are obtained from $\mathcal{M}^{k,t}$ by relabeling the
    leaves uniformly.
\end{rem}

\begin{lemma}[Moments of a reduced process] \label{prop:many2fewReduced}
    Fix $k \ge 1$ and let $G$ be a functional of the product form
    \eqref{eq:productFunctional}. Then, for all $s < t$,
    \begin{equation} \label{eq:momentFiniteReduced}
        \Mc^{k,s,t}[G] = \frac{1}{\abs{c}!} \int_0^s r_{t-x}\E[K_{t-x}^{(\abs{c})}]
        f(s-x) e^{-(k-1) \int_0^x m_{t-y} \diff y}
        \prod_{i=1}^{\abs{c}} \Mc^{c_i,s-x,t-x}[F_i]
        \diff x.
    \end{equation}
\end{lemma}

\begin{proof}
    Fix a composition $c$ of $k$ with $\abs{c} = d$ blocks, and let us
    write $\pi$ for the partition of $\{1,\dots,k\}$ made of blocks of
    consecutive integers with respective lengths $(c_1,\dots,c_d)$.
    If $w = v_1 \wedge \dots \wedge v_k$, then $c(d_{s,t}(\mathbf{v})) = c$
    if and only if there exist some labels $p_1 < \dots < p_d \le K_w$
    such that, for each $i \le d$, all the leaves in $\pi_i$ descend from
    $wp_i$ (that is, $wp_i \preceq v_j$, $j \in \pi_i$), where we recall
    that $wp_i$ is the $p_i$-th child of $w$. This allows us to re-write
    the sum in the right-hand side of \eqref{eq:defPlanarMoment} as
    \begin{multline*}
        \sum_{\substack{(v_j) \in \mathcal{N}_{s,t} \\ v_1 < \dots < v_k}}
        f(\tau(d_{s,t}(\mathbf{v}))) \indic_{\{ c(d_{s,t}(\mathbf{v})) = c\}}
        \prod_{i=1}^{\abs{c}} \Big(
        F_i(U_i(d_{s,t}(\mathbf{v})))
        \prod_{j \in \pi_i} W^{v_j}_{s,t}
        \Big)
        \\
        =
        \sum_{w} \indic_{\{ \sigma_w + \omega_w < s\}}
        f(s - \sigma_w - \omega_w)
        \sum_{\substack{p_1, \dots, p_{\abs{c}} = 1 \\
                p_1 < \dots < p_{\abs{c}}}}^{K_w}
        \prod_{i=1}^{\abs{c}}
        \Bigg(
        \sum_{\substack{v^i_1 < \dots < v^i_{c_i} \in \mathcal{N}_{s,t} \\
        wp_i \preceq v^i_j}}
        F_i(d_{s,t}(\mathbf{v}^i))
        \prod_{j \in \mathbf{v}^i} W^{v_j}_{s,t}
        \Bigg),
    \end{multline*}
    where we used the notation $\mathbf{v}^i = (v^i_j)_j$, and that with
    this notation $U_i(d_{s,t}(\mathbf{v})) = d_{s,t}(\mathbf{v}^i)$.
    We will use that by definition of $W^{v}_{s,t}$ and the branching
    property, for $v \in \mathcal{N}_{s,t}$,
    \[
        W^{v}_{s,t}
        \overset{\mathrm{(d)}}{=}
        W^{\varnothing}_{0,t-s} e^{- \int_0^s m_{t-y} \diff y}.
    \]
    Let us denote by $K_w$ the number of offspring of individual $w$ (its
    out-degree). Applying the strong Markov property at $\sigma_w +
    \omega_w$ and the branching property further yields that
    \begin{align*}
        \E\bigg[
         & \sum_{\substack{(v_j) \in \mathcal{N}_{s,t} \\v_1 < \dots < v_k}}
        f(\tau(U_s)) \indic_{\{ c(U_s) = c\}}
        \prod_{i=1}^{\abs{c}} \Big(
        F_i(U_i(d_{s,t}(\mathbf{v})))
        \prod_{j \in \pi_i} W^{v_j}_{s,t}
        \Big)
        \bigg]                                         \\
         & = \E\bigg[
        \sum_{w} \indic_{\{ \sigma_w + \omega_w < s\}} f(s - \sigma_w - \omega_w)
        \frac{K_w^{(\abs{c})}}{\abs{c}!}
        \prod_{i=1}^{\abs{c}}
        e^{-c_i \int_0^{\sigma_w+\omega_w} m_{t-y} \diff y}
        \Mc^{c_i, s-\sigma_w-\omega_w, t-\sigma_w-\omega_w}[F_i]
        \bigg]                                         \\
         & =
        \int_0^s \frac{1}{\abs{c}!} e^{\int_0^{s'} m_{t-y} \diff y} r_{t-s'}
        \E[K_{t-s'}^{(\abs{c})}] f(s-s')
        \prod_{i=1}^{\abs{c}}
        e^{-c_i \int_0^{s'} m_{t-y} \diff y}
        \Mc^{c_i,s-s',t-s'}[F_i] \diff s'              \\
         & =
        \int_0^s \frac{1}{\abs{c}!} e^{-(k-1)\int_0^{s'} m_{t-y} \diff y} r_{t-s'}
        \E[K_{t-s'}^{(\abs{c})}] f(s-s')
        \prod_{i=1}^{\abs{c}}
        \Mc^{c_i,s-s',t-s'}[F_i] \diff s',
    \end{align*}
    where we have used that $K_w \sim K_{\sigma_w + \omega_w}$, and that
    the death times in the population $(\sigma_w + \omega_w)_{w \in
                \mathcal{V}}$ is a point process with intensity $r_{t-s'}
        e^{\int_0^{s'} m_{t-y} \diff y} \diff s'$.
\end{proof}

At several locations, we will need to go from convergence of moments
evaluated at product-form functionals to convergence of moments for any
continuous bounded functional. A slight inconvenience in doing
so is that the map $U \mapsto (\tau(U), (U_i)_i)$
(see Definition~\ref{def:composition})
is not continuous
because of the possible accumulation of branch points in the limit. (It
is discontinuous at any $U$ whose first branch point is non-binary.) 
This
issue is resolved by the following argument that we isolate as a lemma.

\begin{lemma} \label{lem:product2Tree}
    Let $(\mathcal{M}_n)_{n \ge 1}$ be a sequence of measures on
    $\mathbb{U}_k$ such that, for any functional $G$ of the product
    form \eqref{eq:product-form1}, there exists $m_G \in \R$ with
    \[
        \mathrm{(i)}\quad \lim_{n \to \infty} \mathcal{M}_n[G] = m_G
        ,\qquad \mathrm{(ii)}\quad
        \adjustlimits\lim_{\eta \to 0} \limsup_{n \to \infty}
        \mathcal{M}_n\big(
        \{ U : \tau(U) - \max_i \tau(U_i(U)) \le \eta \}
        \big)
        = 0.
    \]
    Then there exists a measure $\mathcal{M}$ on $\mathbb{U}_k$ such that
    $(\mathcal{M}_n)_{n \ge 1}$ converges weakly to $\mathcal{M}$ and
    $m_G = \mathcal{M}[G]$.
\end{lemma}

\begin{proof}
    Let $\Phi \colon U \mapsto (\tau(U), (U_i(U))_i)$. Its range is the
    set
    \begin{equation} \label{eq:imageDecomposition}
        D = \Big\{
        (\tau, (U_i)_i) : U_i \in \mathbb{U}_{c_i},\, \sum_i c_i = k,\,
        \tau > \max_i \tau(U_i)
        \Big\}.
    \end{equation}
    A reformulation of the first assumption is that the pushforward of
    $(\mathcal{M}_n)_{n \ge 1}$ by $\Phi$ converges weakly to some limit
    $\mathcal{M}'$, which is a measure on the closure of $D$ (for the
    product topology) and is such that $\mathcal{M}'[G] = m_G$. The
    boundary of $D$ is $\partial D = \{ (\tau, (U_i)_i) : \tau = \max_i
        \tau(U_i) \}$, and thus second assumption ensures that
    $\mathcal{M}'(\partial D) = 0$.

    Now note that $\Phi$ has an explicit inverse $\Phi^{-1} \colon D \to
        \mathbb{U}_k$ given by
    \begin{equation} \label{eq:inverseDistanceMatrix}
        \Phi^{-1} \colon \big( \tau, (U_i)_i \big) \mapsto
        (U_{m,n}) =
        \begin{dcases}
            ( U_i )_{m,n} & \text{if $m, n \in \pi_i$,} \\
            \tau          & \text{else,}
        \end{dcases}
    \end{equation}
    where $\pi = (\pi_i)_i$ is the partition with blocks made of
    consecutive integer and whose $i$-th block has length $c_i$. Clearly,
    $\Phi^{-1}$ is continuous and thus $(\mathcal{M}_n)_{n \ge 1}$
    converges weakly to the pushforward $\mathcal{M}$ of $\mathcal{M}'$
    by $\Phi^{-1}$ (which is well-defined since $\mathcal{M}'(\partial D)
        = 0$). Since $\Phi$ is bijective, $\mathcal{M}[G] = \mathcal{M}'[G]$.
\end{proof}

We now go back to computing the moments of $(\mathcal{U}_t, d_t,
    \vartheta'_t)$ and prove an intermediate result.

\begin{lemma} \label{lem:convReducedPlanarMoment}
    Under the assumptions of Proposition~\ref{cor:momentsReduced}, for any
    $k \ge 1$,
    \[
        \lim_{s \uparrow t} \Mc^{k,s,t} = \Mc^{k,t}
    \]
    weakly as measures on $\mathbb{U}_k$. Moreover, the limit $\Mc^{k,t}$
    is uniquely characterised by \eqref{eq:recursionReduced}.
\end{lemma}

\begin{proof}
    Let us show by induction that, for all $p \ge 1$,
    \begin{equation} \label{eq:inductionConvergence}
        \Mc^{p,t} \coloneqq \lim_{s \to t} \Mc^{p,s,t}
    \end{equation}
    exists in the weak topology, with $\Mc^{p,t}$ satisfying
    \eqref{eq:recursionReduced}, and that
    \begin{equation} \label{eq:boundedMassReduced}
        \forall T > 0,\quad \sup_{0 \le s < t \le T} \Mc^{p,s,t}[1] < \infty.
    \end{equation}

    For $p = 1$, by the martingale property
    \[
        \mathcal{M}^{1,s,t}[1] = e^{-\int_0^s m_{t-y} \diff y}
        \E[Z_{s,t}] = 1,
    \]
    hence both properties.

    Now, fix $k \ge 2$ and assume that these properties hold for all $p <
        k$. The recursive formula for the moments of the reduced process,
    Lemma~\ref{prop:many2fewReduced}, entails that
    \begin{equation} \label{eq:integrandMomentsReduced}
        \Mc^{k,s,t}[G]
        = \frac{1}{\abs{c}!}
        \int_0^s r_{t-x} \E[K_{t-x}^{(\abs{c})}]
        f(s-x) e^{-(k-1)\int_0^x m_{t-y} \diff y}
        \prod_{i=1}^{\abs{c}} \Mc^{c_i, s-x, t-x}[F_i] \diff x.
    \end{equation}
    By assumption~\eqref{eq:conditionMomentReduced} and
    \eqref{eq:boundedMassReduced}, the integrand on the right-hand side
    is bounded. Summing over all compositions with $f, F_i \equiv 1$
    shows that \eqref{eq:boundedMassReduced} holds for $p=k$. Moreover,
    by induction \eqref{eq:inductionConvergence}, the integrand of
    \eqref{eq:integrandMomentsReduced} converges pointwise as $s \uparrow
        t$. The dominated convergence theorem thus shows that the
    limit as $s \uparrow t$ of $(\mathcal{M}^{k,s,t}[G])_{s < t}$ exists
    for all $G$ of the product form, and is given by the right-hand side
    of \eqref{eq:recursionReduced}. By Lemma~\ref{lem:product2Tree}, this
    entails that \eqref{eq:inductionConvergence} also holds for $p = k$.
    (The second assumption of Lemma~\ref{lem:product2Tree} is a direct
    consequence of the fact that $\tau(U)$ has a density under
    $\mathcal{M}^{p,t}$.)
\end{proof}

\begin{proof}[Proof of Proposition~\ref{cor:momentsReduced}]
    Re-ordering the sum in the definition of $\mathcal{M}^{k,s,t}$ from
    \eqref{eq:defPlanarMoment},
    \begin{equation*}
        e^{-k \int_0^s m_{t-y} \diff y}
        \E\bigg[\sum_{\substack{v_1, \dots, v_k \in \mathcal{N}_{s,t}\\
                    \text{$(v_i)$ distinct}}}
            F(d_{s,t}(\mathbf{v}))
            \bigg]
        = \sum_{P} \Mc^{k,s,t}[F \circ P],
    \end{equation*}
    where the sum is taken over all partitions of $\{1,\dots,k\}$.
    Therefore, using Lemma~\ref{lem:convReducedPlanarMoment}, it is
    sufficient to prove that for any continuous bounded $F \colon
        \mathbb{U}_k \to \R$
    \begin{equation} \label{eq:comparisonFactReduced}
        \lim_{s \to t}
        e^{-k \int_0^s m_{t-y} \diff y}
        \E\bigg[\sum_{\substack{v_1, \dots, v_k \in \mathcal{N}_{s,t}\\
                    \text{$(v_i)$ distinct}}}
            F(d_{s,t}(\mathbf{v}))
            \bigg]
        =
        \E\Big[
            \int_{\mathcal{U}_t^k} F\big( d_t(\mathbf{v}) \big)
            {\vartheta'_t}^{\otimes k}(\diff \mathbf{v})
            \Big].
    \end{equation}
    Without loss of generality, by Portmanteau's theorem, we can assume
    that $F$ is Lipschitz.

    First, by \eqref{eq:boundedMassReduced},
    \begin{equation} \label{eq:uniformMomentReduced}
        \sup_{s < t} \E\Big[
            e^{-k \int_0^s m_{t-y} \diff y}
            Z_{s,t}^k
            \Big] < \infty,
    \end{equation}
    and therefore, since $e^{-\int_0^s m_{t-y} \diff y} Z_{s,t} \to
    W^\varnothing_{0,t}$
    as $s \to t$,
    \[
        \forall t > 0,\quad \E[W^\varnothing_{0,t}] = 1, 
        \qquad \E[(W^\varnothing_{0,t})^k] < \infty.
    \]
    Recall that $(W^v_{s,t})_{v \in \mathcal{N}_{s,t}}$ are i.i.d.\ and
    distributed as $e^{-\int_0^s m_{t-y} \diff y} W^\varnothing_{0,t-s}$ and thus
    \[
        e^{-k \int_0^s m_{t-y} \diff y}
        \E\bigg[\sum_{\substack{v_1, \dots, v_k \in \mathcal{N}_{s,t}\\
                    \text{$(v_i)$ distinct}}}
            F(d_{s,t}(\mathbf{v}))
            \bigg]
        =
        \E\bigg[\sum_{\substack{v_1, \dots, v_k \in \mathcal{N}_{s,t}\\
                    \text{$(v_i)$ distinct}}}
            F(d_{s,t}(\mathbf{v}))
            \prod_{i=1}^k W^{v_i}_{s,t}
            \bigg].
    \]
    There are $Z^k_{s,t} - Z^{(k)}_{s,t} = O(Z^{k-1}_{s,t})$ ways of
    choosing $k$ elements in $\mathcal{N}_{s,t}$ that are not all
    distinct. Therefore,
    \begin{align*}
        \E\bigg[
        \sum_{\substack{v_1, \dots, v_k \in \mathcal{N}_{s,t} \\
                    (v_i) \text{ distinct}}}
         & F\big( d_{s,t}(\mathbf{v}) \big)
            \prod_{i=1}^k W^{v_i}_{s,t}
        \bigg]                                                \\
         & =
        \E\bigg[
            \sum_{v_1, \dots, v_k \in \mathcal{N}_{s,t}}
            F\big( d_{s,t}(\mathbf{v}) \big)
            \prod_{i=1}^k W^{v_i}_{s,t}
            \bigg]
        + O\Big(
        \E\Big[ e^{-k \int_0^s m_{t-y} \diff y} (Z_{t,s}^k - Z^{(k)}_{t,s}) \Big]
        \Big)                                                 \\
         & =
        \E\bigg[
            \sum_{v_1, \dots, v_k \in \mathcal{N}_{s,t}}
            F\big( d_{s,t}(\mathbf{v}) \big)
            \prod_{i=1}^k W^{v_i}_{s,t}
            \bigg]
        + e^{- \int_0^s m_{t-y} \diff y}
        O\Big(
        e^{-(k-1) \int_0^s m_{t-y} \diff y} \E[Z_{t,s}^{k-1}]
        \Big).
    \end{align*}
    The second term vanishes as $s \to t$ by \eqref{eq:uniformMomentReduced}
    and the second part of assumption \eqref{eq:conditionMomentReduced}.
    For the first term, we know from applying
    Lemma~\ref{lem:boundLipschitz} to $(\mathcal{N}_{s,t},
        d_{s,t}, \vartheta_{s,t})$ and $(\mathcal{U}_t, d_t, \vartheta'_t)$
    that
    \begin{multline*}
        \abs[\bigg]{
            \E\bigg[
                \sum_{v_1, \dots, v_k \in \mathcal{N}_{s,t}}
                F\big( d_{s,t}(\mathbf{v}) \big) \prod_{i=1}^k
                W^{v_i}_{s,t}
                -
                \int_{\mathcal{U}_t^k} F\big( d_t(\mathbf{v}) \big)
                {\vartheta'_t}^{\otimes k}(\diff \mathbf{v})
                \bigg]
        } \\
        \le C_k \E\Big[
            \underline{\Box}_1\Big( (\mathcal{N}_{s,t}, d_{s,t},
            \vartheta_{s,t}), (\mathcal{U}_t, d_t, \vartheta_t') \Big)
            \Big],
    \end{multline*}
    and that the right-hand side vanishes as $s \uparrow t$ by
    Proposition~\ref{prop:defUMS}. Altogether, these steps show
    \eqref{eq:comparisonFactReduced} and prove the result.
\end{proof}

\subsection{The genealogy of CSBPs}
\label{sec:proofCSBP}

We can now provide the definition of the genealogy of a CSBP. Recall that
a branching mechanism is a function of the form
\[
    \forall \theta \ge 0,\quad \psi(\theta) = a \theta + \tfrac{b}{2} \theta^2 +
    \int_0^\infty \big( e^{-\theta z} - 1 + z\theta \big)
    \Lambda(\diff z).
\]
Under Grey's condition \eqref{eq:greyCondition}, we defined a limiting
Laplace exponent $\bar{u}_t$ in \eqref{eq:laplaceExponent}. Under that
condition, recall that we can also define for each $t > 0$ a random
variable $\bar{\Xi}_t$ corresponding to the law of the CSBP ``started
from $0$'' and conditioned on survival. It is defined through its Laplace
transform as
\begin{equation}  
    \forall \theta \ge 0,\quad \bar{u}_t \E[ 1-e^{-\theta \bar{\Xi}_t} ] =
    u_t(\theta).
\end{equation}
We will refer to this random variable as being distributed under the
\emph{entrance law} of the CSBP. The reduced process in the following
definition was originally constructed in \cite{LeGall2002}, Section~2.7.

\begin{definition}[$\psi$-metric measure space]
    \label{def:psiMmSpace}
    Let $\psi$ be a branching mechanism \eqref{eq:branching-mechanism}
    satisfying Grey's condition \eqref{eq:greyCondition} and let
    $(r_{i,\tau})_{i \ge 2}$ be such that
    \begin{equation}  \label{eq:CSBPReducedBis}
        \forall \theta \ge 0,\quad
        \sum_{i \ge 2} r_{i,\tau} \theta^i
        = \theta \Big( \psi'(\bar{u}_{\tau}) - \frac{\psi(\bar{u}_{\tau}) -
            \psi((1-\theta)\bar{u}_{\tau})}{\theta \bar{u}_{\tau}}
        \Big).
    \end{equation}
    Let $(\mathcal{U}_t, d_t, \vartheta'_t)$ be the metric measure space
    constructed from the rates \eqref{eq:CSBPReducedBis} in
    Proposition~\ref{prop:defUMS}. The $\psi$-metric measure space
    ($\psi$-mm space) at time $t$ is the random metric measure space
    $(\mathcal{U}_t, d_t, \vartheta_t)$, where  $\vartheta_t = e^{-a t}
        \vartheta'_t / \bar{u}_t$.
\end{definition}

Proposition~\ref{prop:momentMmSpace}, which provides a recursive expression for
the moments of a $\psi$-mm space, now follows by specialising
Proposition~\ref{cor:momentsReduced} to the rates defined in
\eqref{eq:CSBPReducedBis}.

\begin{proof}[Proof of Proposition~\ref{prop:momentMmSpace}]
    Recall that $K_\tau$ denotes a random variable with distribution
    $(r_{i,\tau} / r_{\tau})_{i \ge 2}$ and denote by
    $G_\tau(\theta) = r_\tau \E[ \theta^{K_\tau} ]$ the
    generating function in \eqref{eq:CSBPReducedBis}. We obtain the
    factorial moments of $K_\tau$ by differentiating $G_\tau$. First,
    \[
        r_\tau = G_\tau(1)
        = \psi'(\bar{u}_\tau) - \frac{\psi(\bar{u}_\tau) - \psi(0)}{\bar{u}_\tau}
    \]
    and
    \[
        \partial_\theta G_\tau(\theta) = \psi'(\bar{u}_\tau) -
        \psi'((1-\theta)\bar{u}_\tau),
    \]
    so that
    \[
        m_\tau = r_\tau\E[K_\tau-1] = \partial_\theta G_\tau(1) - G_\tau(1) =
        \frac{\psi(\bar{u}_\tau)}{\bar{u}_\tau} - \psi'(0) =
        \frac{\psi(\bar{u}_\tau)}{\bar{u}_\tau} - a.
    \]
    By taking further derivatives we obtain
    \begin{equation} \label{eq:momentCSBP-1}
        r_\tau\E[K^{(k)}_\tau] = \partial_\theta^k G_\tau(1)
        = (-1)^k \bar{u}_\tau^{k-1} \psi^{(k)}(0)
        = \bar{u}_\tau^{k-1}\Big(\indic_{\{ k = 2\}}b
        + \int_0^\infty x^k \Lambda(\diff x)\Big).
    \end{equation}
    Moreover,
    \begin{equation*}
        \forall 0 < s \le t,\quad \int_s^t \frac{\psi(\bar{u}_x)}{\bar{u}_x} \diff x =
        - \int_s^t \frac{\bar{u}'_x}{\bar{u}_x} \diff x = \log \bar{u}_s
        - \log \bar{u}_t,
    \end{equation*}
    from which we deduce that
    \begin{equation}\label{eq:compensationReduced}
        \forall s < t,\quad e^{-\int_0^s m_{t-x} \diff x} =
        \frac{\bar{u}_t}{\bar{u}_{t-s}} e^{sa}.
    \end{equation}
    It can be directly checked from this expression using that $\bar{u}_t
        \to \infty$ as $t \to 0$ that the two conditions in
    \eqref{eq:conditionMomentReduced} are fulfilled.

    Let $\mathcal{M}^{k,t}$ be the planar moment of $(\mathcal{U}_t, d_t,
        \vartheta'_t)$ as constructed in Proposition~\ref{cor:momentsReduced}
    and define
    \[
        \hat{\mathcal{M}}^{k,t} = \frac{e^{-ka t}}{\bar{u}_t^{k-1}}
        \mathcal{M}^{k,t}.
    \]
    By Proposition~\ref{cor:momentsReduced}, since $\vartheta_t$ is
    obtained by dividing the mass of $\vartheta'_t$ by $e^{at} \bar{u}_t$,
    \begin{multline*}
        \E\Big[
            \int_{\mathcal{U}_t^k} F\big( d_t(\mathbf{v}) \big)
            \vartheta_t^{\otimes k}(\diff \mathbf{v}) \big)
            \Big]
        =
        \frac{e^{-ka t}}{\bar{u}_t^k}
        \E\Big[
            \int_{\mathcal{U}_t^k} F\big( d_t(\mathbf{v}) \big)
            {\vartheta'_t}^{\otimes k}(\diff \mathbf{v}) \big)
            \Big] \\
        = \sum_{P} \frac{e^{-ka t}}{\bar{u}_t^k}
        \mathcal{M}^{k,t}[F \circ P]
        = \frac{1}{\bar{u}_t} \sum_{P} \hat{\mathcal{M}}^{k,t}[F \circ P].
    \end{multline*}
    Finally, by the recursive expression for $\mathcal{M}^{k,t}$ in
    \eqref{eq:recursionReduced}, \eqref{eq:momentCSBP-1} and
    \eqref{eq:compensationReduced} above,
    \begin{align*}
        \hat{\Mc}^{k,t}[G]
         & = \frac{e^{-k a t}}{\bar{u}_t^{k-1}} \frac{1}{\abs{c}!}
        \int_0^t f(t-s) r_{t-s}\E[K^{(\abs{c})}_{t-s}]
        e^{-(k-1) \int_0^s m_{t-x} \diff x}
        \prod_{i=1}^{\abs{c}} \Mc^{c_i, t-s}[F_i] \diff s          \\
         & =
        \frac{e^{-k a t}}{\bar{u}_t^{k-1}}
        \frac{m_{\abs{c}}}{\abs{c}!} \int_0^t f(t-s) \bar{u}_{t-s}^{\abs{c}-1}
        \Big(\frac{\bar{u}_t e^{sa}}{\bar{u}_{t-s}}\Big)^{k-1}
        \prod_{i=1}^{\abs{c}} \bar{u}_{t-s}^{c_i-1} e^{c_i a (t-s)}
        \hat{\Mc}^{c_i,t-s}[F_i] \diff s
        \\
         & = \frac{m_{\abs{c}}}{\abs{c}!} \int_0^t f(t-s) e^{-s a}
        \prod_{i=1}^{\abs{c}} \hat{\Mc}^{c_i,t-s}[F_i] \diff s.
        \qedhere
    \end{align*}
\end{proof}

\subsection{Perturbation results}

When removing the spatial cutoff of the branching diffusion by letting $A
    \to \infty$, we will need a perturbation result for the genealogies of
CSBPs. So far, we implicitly assume that $\int z \wedge z^2
    \Lambda(\diff z) < \infty$ to express the branching mechanism as
\eqref{eq:branching-mechanism}. For the perturbation result we will need
to relax this assumption and consider general branching mechanisms of
the form
\begin{equation} \label{eq:generalBranchingMechanism}
    \psi(\theta) = \widetilde{a} \theta + \frac{b}{2} \theta^2 +
    \int_{(0, \infty)} \big( e^{-z \theta} - 1 - \indic_{\{ z \le 1 \}}
    \theta z \big) \Lambda(\diff z)
\end{equation}
with $\widetilde{a} \in \R$, $b \ge 0$, and $\Lambda$ such that
\[
    \int_{(0,\infty)} \big( 1 \wedge z^2 \big) \Lambda(\diff z) < \infty.
\]

\begin{proposition} \label{prop:equivalenceBranching}
    Let $(\psi_j)_{j \ge 1}$ be a sequence of branching mechanisms of the
    form \eqref{eq:generalBranchingMechanism} and $(\Xi_{t,j}(z))_{j \ge
                1}$ be the corresponding $\psi_j$-CSBPs, started from $z > 0$. The
    following statements are equivalent.
    \begin{enumerate}
        \item[(i)] There exists $\psi$ such that, for all $\theta \ge 0$,
              $\psi_j(\theta) \to \psi(\theta)$ as $j \to \infty$.
        \item[(ii)] There is some $z > 0$ such that,
              for all $t > 0$, $\Xi_{t,j}(z) \to \Xi_t(z)$ as $j \to
                  \infty$ in distribution for some limit $\Xi_t(z)$.
    \end{enumerate}
    If these statements are fulfilled, $\psi$ is a branching mechanism of
    the form \eqref{eq:generalBranchingMechanism} and $(\Xi_t)_{t \ge 0}$
    is the corresponding $\psi$-CSBP. Suppose in addition that
    $(\psi_j)_j$ and $\psi$ fulfill Grey's condition
    \eqref{eq:greyCondition} and let $(\bar{u}_{t,j})_j$ and
    $(\bar{\Xi}_{t,j})_j$ (resp.\ $\bar{u}_t$ and $\bar{\Xi}_t$) be defined
    as in \eqref{eq:laplaceExponent} from $(\psi_j)_j$ (resp.\ $\psi$).
    If
    \[
        \forall t > 0,\qquad
        \lim_{j \to \infty} \bar{u}_{t,j} = \bar{u}_t,
        \qquad
        \lim_{j \to \infty} \E\big[ \bar{\Xi}_{t,j} \big]
        = \E\big[ \bar{\Xi}_t \big] < \infty,
    \]
    then
    \[
        \lim_{j \to \infty} (\mathcal{U}_{t,j}, d_{t,j}, \vartheta_{t,j})
        = (\mathcal{U}_t, d_t, \vartheta_t),
    \]
    in distribution for the Gromov-weak topology, where
    $(\mathcal{U}_{t,j}, d_{t,j}, \vartheta_{t,j})_j$ and
    $(\mathcal{U}_t, d_t, \vartheta_t)$ are the corresponding mm-spaces.
\end{proposition}

Before proving Proposition~\ref{prop:equivalenceBranching}, we will need
the following preliminary result.

\begin{lemma} \label{prop:convergenceCSBP}
    Let $(Z^j_{s,t};\, s < t)_{j \ge 1}$ be a sequence of reduced
    processes, and let $(\mathcal{U}_t^j, d_t^j, {\vartheta'_t}^j)_{j \ge 1}$
    be the corresponding mm-spaces from Proposition~\ref{prop:defUMS}.
    Suppose that there exists a reduced process $(Z_{s,t};\, s < t)$ such
    that:
    \begin{enumerate}
        \item For each $\tau > 0$, $((r^j_{i,\tau})_{i \ge 2})_{j \ge 1}$
              converges in $\ell^1$ as $j \to \infty$ to the sequence
              $(r_{i,\tau})_{i \ge 2}$. Moreover,
              \[
                  \forall \epsilon > 0,\quad
                  \sup_{j \ge 1} \sup_{\tau \in (\epsilon, t)} r^j_{\tau} < \infty.
              \]
        \item For each $t > 0$, $W^j_t \to W_t$ in distribution as $j \to \infty$.
    \end{enumerate}
    Then $(\mathcal{U}_t^j, d_t^j, {\vartheta'_t}^j)_{j \ge 1}$ converges in
    distribution to $(\mathcal{U}_t, d_t, \vartheta'_t)$ in the
    Gromov-weak topology.
\end{lemma}

\begin{proof}
    By the uniform bound in \eqref{eq:boundBoxMetric}, it is sufficient
    to show that, for all fixed $s < t$, $(\mathcal{N}^j_{s,t},
        d_{s,t}^j, \vartheta^j_{s,t}) \to (\mathcal{N}_{s,t}, d_{s,t},
        \vartheta_{s,t})$ as $j \to \infty$. Let $T_j$ be the first branch
    time in $(Z^j_{s,t}; s < t)$ and $K_j$ be the number of offspring
    produced at time $T_j$. By construction, their joint law is given by
    \[
        \P(T_j \in \diff s, K_j = k) =
        \indic_{s \in (0, t)} e^{- \int_0^s r^j_{t-x} \diff x}
        r^j_{k,t-s} \diff s.
    \]
    Our assumptions and the dominated convergence theorem entail that
    \[
        \forall s \in (0, t),\quad
        \lim_{j \to \infty}
        e^{- \int_0^s r^j_{t-x} \diff x} r^j_{k,t-s}
        =
        e^{- \int_0^s r_{t-x} \diff x} r_{k,t-s}
    \]
    and therefore $(T_j, K_j) \to (T, K)$ in distribution as $j \to
        \infty$, where $T$ and $K$ are the first branch time and number of
    offspring of the limiting reduced process $(Z_{s,t}; s < t)$.

    Iterating over finitely many branch points, we see that
    \begin{equation} \label{eq:convergenceCSBP-1}
        \forall s < t,\quad
        \lim_{j \to \infty}
        \big(
        \mathcal{N}^j_{s,t}, (\sigma^j_v, \omega^j_v)_{v : \sigma^j_v \le s}
        \big)
        =
        \big(
        \mathcal{N}_{s,t}, (\sigma_v, \omega_v)_{v : \sigma_v \le s}
        \big)
    \end{equation}
    in distribution, where we view $\mathcal{N}_{s,t}^j$ as a random
    element of the (countable) set of all finite subsets of Ulam--Harris
    labels $\mathcal{V}$. Recall that we construct the measure
    $\vartheta^j_{t-s}$ by attaching i.i.d.\ weights
    $(W^{j,v}_{t-s})_{v \in \mathcal{N}^j_{s,t}}$ to the individuals at
    time $s$, each distributed as $W^j_{t-s}$. By our second assumption,
    these weights converge in distribution to $(W^v_{t-s})_{v \in
                \mathcal{N}_{s,t}}$ jointly with the other variables in
    \eqref{eq:convergenceCSBP-1}.

    By Skorohod's representation theorem, let us assume that these random
    variables actually converge almost surely. For $j$ large enough,
    $\mathcal{N}^j_{s,t} = \mathcal{N}_{s,t}$ and by definition of
    $d^j_{s,t}$ and $\vartheta^j_{t-s}$,
    \[
        \forall v, w \in \mathcal{N}_{s,t},\quad
        \lim_{j \to \infty} d^j_{s,t}(v, w) = d_{s,t}(v, w),\quad
        \lim_{j \to \infty} \vartheta^j_{s,t}( \{ v \} ) =
        \vartheta_{s,t}(\{ v \}).
    \]
    It is not hard to see that these two limits entail that
    $(\mathcal{N}^j_{s,t}, d^j_{s,t}, \vartheta^j_{s,t})_{j \ge 1}$
    converges to $(\mathcal{N}_{s,t}, d_{s,t}, \vartheta_{s,t})$ in the
    Gromov-weak topology. One can for instance rely again Gromov's
    $\underline{\Box}_1$ metric, see \cite[Definition~3.2]{janson2020gromov},
    with the trivial relation $R = \{ (v, v) \in \mathcal{N}^j_{s,t}
        \times \mathcal{N}_{s,t} \}$ and a coupling of $\vartheta^j_{s,t}$
    and $\vartheta_{s,t}$ that puts maximal weight on the diagonal.
\end{proof}

\begin{proof}[Proof of Proposition~\ref{prop:equivalenceBranching}]
    The direction (i) $\implies$ (ii) was already proved for $z=1$ in
    \cite{labbe2014genealogy}, Proposition~5.6. The argument relies on
    the fact that CSBPs can be represented as time-changed Lévy
    processes, and that this so-called Lamperti transformation is a
    continuous operation by a result of \cite{Caballero2009}. The exact
    same argument shows the result for any $z > 0$.

    We prove the converse implication and suppose that (ii) holds. Then
    by assumption
    \[
        u_{t,j}(\theta) = - \frac{1}{z} \log \E[e^{-\theta \Xi_{t,j}(z)}]
        \to -\frac{1}{z} \log \E[e^{-\theta \Xi_t(z)}] \eqqcolon
        u_t(\theta), \qquad \text{as $j \to \infty$}
    \]
    and
    \[
        \partial_\theta u_{t,j}(\theta)
        = \frac{\E[\Xi_{t,j}(z) e^{-\theta \Xi_{t,j}(z)}]}%
        {z\E[e^{-\theta \Xi_{t,j}(z)}]}
        \to \frac{\E[\Xi_t(z) e^{-\theta \Xi_t(z)}]}%
        {z\E[{e^{-\theta \Xi_t(z)}}]}
        = \partial_\theta u_t(\theta), \qquad \text{as $j \to \infty$.}
    \]
    Moreover, \eqref{eq:ODECSBP} has a corresponding forward equation
    that reads
    \[
        \forall \theta, t \ge 0,\quad
        \partial_t u_{t,j}(\theta) = -\psi_j(\theta) \partial_\theta
        u_{t,j}(\theta).
    \]
    see \cite{li2011measure}, (3.6) for instance. We deduce that
    \[
        \forall \theta \ge 0,\quad \psi_j(\theta)
        = \frac{\theta - u_{t,j}(\theta)}%
        {\int_0^t \partial_\theta u_{s,j}(\theta) \diff s}.
    \]
    The dominated convergence theorem entails that $\psi_j \to \psi$
    pointwise as $j \to \infty$.

    Let us now prove the second part of the statement. 
    By Dini's theorem, $(\psi_j)_j$ and $(\psi'_j)_j$ converge
    to $\psi$ and $\psi'$ uniformly on compact subsets of $(0, \infty)$.
    Therefore, if $L_{j,\tau}(\theta)$ (resp.\ $L_{\tau}(\theta)$)
    denotes the Laplace transform at time $\tau$ obtained from $\psi_j$
    (resp.\ $\psi$) in \eqref{eq:CSBPReducedBis}, it follows that
    \[
        \forall \theta \in [0, 1],\: \forall \tau > 0,\quad
        \lim_{j \to \infty} L_{j,\tau}(\theta)
        = L_\tau(\theta).
    \]
    This checks the first point of Lemma~\ref{prop:convergenceCSBP}.

    For the second point of that lemma, recall the notation $Z_{s,t}$
    for the number of particles at time $s < t$ and recall that
    \[
        \widetilde{W}_{s,t} \coloneqq e^{-\int_0^s m_{t-x} \diff x} Z_{s,t},\quad
        s \in [0, t)
    \]
    is a martingale, the limit of which we denote by $\widetilde{W}_t$.
    (In the notation of Section~\ref{sec:reducedProcess}, $\widetilde{W}_t =
    W^\varnothing_{0,t}$).
    We start by showing that $\widetilde{W}_t \sim \bar{u}_t e^{at} \bar{\Xi}_t$,
    where $\bar{\Xi}_t$ is defined in \eqref{eq:laplaceEntrance}.
    Theorem~2.7.1 of \cite{LeGall2002} shows that
    \[
        \E[e^{-\theta Z_{s,t}}] = 1 -
        \frac{u_t((1-e^{-\theta})\bar{u}_{t-s})}{\bar{u}_t}.
    \]
    Recall from \eqref{eq:compensationReduced} that we have
    \[
        \bar{u}_{t-s} e^{-\int_0^s m_{t-x} \diff x} = \bar{u}_t e^{as}.
    \]
    Therefore,
    \begin{align*}
        \E[e^{-\theta \widetilde{W}_t}]
         & = \lim_{s \to t} \E[e^{-\theta \widetilde{W}_{s,t}}]
        = 1 - \lim_{s \to t} \frac{u_t((1-e^{-\theta e^{-\int_0^s
                m_{t-x} \diff x}})\bar{u}_{t-s})}{\bar{u}_t}
        \\
         & = 1 - \frac{u_t(\theta \bar{u}_t e^{at})}{\bar{u}_t}
        = \E[ e^{- \theta \bar{u}_t e^{at} \bar{\Xi}_t}],
    \end{align*}
    hence the claim. Now, looking at the explicit Laplace transform
    in \eqref{eq:laplaceEntrance}, we see that the convergence of
    $(\bar{u}_{t,j})_j$ and $(u_{t,j}(\theta))_j$ to $\bar{u}_t$ and
    $u_t(\theta)$ as $j \to \infty$ directly entails that
    $(\bar{\Xi}_{t,j})_j$ converges in distribution to $\bar{\Xi}_t$ as
    $j \to \infty$. Moreover, since $\widetilde{W}_{t,j} \sim \bar{u}_{t,j} e^{a_j t}
        \bar{\Xi}_{t,j}$,
    \[
        \E\big[ \bar{\Xi}_{t,j} \big] = \frac{e^{-a_j t}}{\bar{u}_{t,j}}
        \longrightarrow
        \frac{e^{-at}}{\bar{u}_{t}} = \E\big[ \bar{\Xi}_{t} \big],
        \qquad \text{as $j \to \infty$}
    \]
    by assumption. Therefore, $(\widetilde{W}_{t,j})_j$ converges to
    $\widetilde{W}_t$ in
    distribution, and so does the corresponding mm-spaces by
    Lemma~\ref{prop:convergenceCSBP}.
\end{proof}

\section{Moment recursion for branching diffusions}
\label{sect:moments-recursion}

Throughout this section, we consider the (dyadic) branching diffusion on an
arbitrary domain $\Omega$ -- either bounded or unbounded -- with
local branching rate $r$ and killed at the boundary, as defined in
Section~\ref{sec:branchingDiffusion}. Our sole assumption is that there
exist $w \in \R$ and a harmonic function $h$ for $\mathcal{L} + r + w$ in
the sense of \eqref{eq:defHarmonic}, where $\mathcal{L}$ is the generator
of the spatial motion. Note that by Proposition~\ref{prop:crit-finite},
this holds for any branching diffusion on a bounded domain. It also
holds for any critical branching diffusion in the sense of
Definition~\ref{def:criticality}. This extension will be needed to deal
with the BBM seen from the tip in Section~\ref{sect:BB M-from-tip}.

Given this harmonic function, we can define a spine process $(\zeta)_{t \ge 0}$ 
as in \eqref{eq:def_spine_g} and some corresponding planar moments
$\mathbf{M}^{k,t}_x$ according to \eqref{eq:biasedMomentBBM}. We let
$\hat{\M}^{k,t}_x$ be the rescaled moments as in
\eqref{eq:scaled_moments}. In analogy with \eqref{eq:product-form1}, we
say that a function $G\colon \mathbb{U}^*_{k}\to \R$ is of the
\emph{product form} if
\begin{equation} \label{eq:productFunctional}
    \forall U^* \in \mathbb{U}^*_k,\quad
    G(U^*) = \indic_{\{ c(U) = c \}} f(\tau(U)) \prod_{p=1}^{\abs{c}} F_p( U^*_p )
\end{equation}
for a fixed composition $c$ of $k$, and some continuous bounded functions
$f \colon \R_+ \to \R$ and $F_p \colon \mathbb{U}^*_{c_p} \to \R$, $p \le
\abs{c}$. The main result of this section is a recursive relation for
moments which mirrors the recursive relations for planar moments of a
$\psi$-mm space in Definition~\ref{def:planarMoments}.

\begin{proposition}[Moments of branching diffusions]
    \label{cor:many-to-few}
    Denote by $\bar{\zeta}_{t} \coloneqq \zeta_{tN}$ the accelerated spine
    process. Under the assumptions of this section,
    \begin{enumerate}
        \item For $k=1$, for every continuous bounded function $F$
              \[
                  \hat \M_{x}^{1,t}[F] = e^{-Nw t} \E_x[  F(\bar \zeta_{t}) ].
              \]

        \item Let $G$ be of the product form \eqref{eq:productFunctional}
              with the support of $f$ is included in $(\epsilon/N,
                  \infty)$, for some $\epsilon > 0$. Then,
              \begin{multline} \label{eq:branching:rn0-1}
                  \hat{\mathbf{M}}_x^{k,t}[G\circ S^{\frac{\vep}{N}}]
                  = \frac{1}{N^{\gamma(\abs{c}-\alpha)}}
                  \int_{\frac{\eps}{N}}^{t} f(u) e^{-Nw(t-u)} \\
                  \times \E_x\left(
                  r(\bar{\zeta}_{t-u}) h(\bar{\zeta}_{t-u})
                  \sum_{n=1}^{\abs{c}-1} \mathbf{M}_{\bar{\zeta}_{t-u}}^{n,\eps}\left[
                      \prod_{i=1}^n \hat{\mathbf{M}}^{c_i,u-\frac{\eps}{N}}_{X_i}[F_i]
                      \right]
                  \mathbf{M}_{\bar{\zeta}_{t-u}}^{\abs{c}-n,\eps}\left[
                  \prod_{i=1}^{\abs{c}-n} \hat{\mathbf{M}}^{c_{i+n},u-\frac{\eps}{N}}_{X_i}[F_{i+n}]
                  \right]
                  \right) \diff u.
              \end{multline}
    \end{enumerate}
\end{proposition}

The proof is based on a first recursive relation that was established in
\cite{schertzer2023spectral}.

\begin{proposition}[\cite{schertzer2023spectral}, Proposition~9]
    \label{def:planar-moments} Under the assumptions of this section,
    \begin{itemize}
        \item[(i)] Assume $k=1$. Then for every continuous and bounded function $F\colon \Omega \to \R$
              \[
                  \M_{x}^{1,t}[F] = e^{-wt} \E_x[  F(\zeta_{t}) ].
              \]
        \item[(ii)] Let $k\geq 2$.  Assume that $n\in[k-1]$ and  consider a product functional of the form
            \[
                  \forall U^*\in \mathbb{U}_k^*, \quad 
                  \bar G(U^*) = \indic_{\{c(U) = (n, k-n)\}}  
                  f(\tau(U)) \bar F_1(U^*_1) \bar F_2(U^*_2).
            \]
              Then
              \begin{equation}\label{eq:recursion-f}
                  \M_{x}^{k,t}[\bar G] = \int_0^t e^{-w (t-u)} f(u)
                  \E_x\left[ r(\zeta_{t-u}) h(\zeta_{t-u})
                      \M_{\zeta_{t-u}}^{n,u}[\bar F_1]
                      \M_{\zeta_{t-u}}^{k-n,u}[\bar F_2]
                      \right]
                  \diff u.
              \end{equation}
    \end{itemize}
\end{proposition}

The proof of Proposition~\ref{cor:many-to-few} requires to generalize
Definition~\ref{def:composition}. Let $s>0$. Informally, cutting the tree
encoded by $U \in \mathbb{U}_k$ at distance $s>0$ from the leaves splits
$U$ into several subtrees which leads to the following definitions. See
Figure~\ref{fig:notationTree} for an illustration at $s=\tau$. Formally,
fix a marked tree $U^* \in \mathbb{U}^*_k$ and $s > 0$.
\begin{itemize}
    \item Let $\pi^s(U)$ be the partition of $\{1, \dots, k\}$ such
          that $i \sim j$ iff $U_{ij} < s$.  Each block of $\pi^s(U)$ is
          made of consecutive integers and the blocks can be ordered
          increasingly. We let $\pi^s_p(U)$ be the $p$-th block.
    \item Let $\abs{c^s(U)}$ be the number of blocks of $\pi^s(U)$
          and let $c^s(U) = (c_p^s(U))_p$ be the composition of $k$
          given by $c^s_p(U) = \abs{\pi^s_p(U)}$, $p \le \abs{c^s(U)}$.
    \item For each $p \le \abs{c^s(U)}$, let $U^{*,s}_p \equiv
              U^{*,s}_p(U^*)$ be the restriction of $U^*$ to the $p$-th block
          $\pi^s_p(U)$.
\end{itemize}
Note that according to Definition~\ref{def:composition}, we have $c(U)
    = c^{\tau(U)}(U)$ and $U^*_p = U_p^{*,\tau(U)}$. We now
show that Proposition~\ref{def:planar-moments} entails the following
lemma.

\begin{lemma}
    \label{lem:recursion-moments}
    Let $k \geq 1$ and $0 < s < t$. Let $c$ be a composition of $k$, and
    $H^s$ a function of the following form
    \[
        \forall U^*\in \mathbb{U}_k^*, \ \  H^s(U^*) \ \coloneqq \
        \indic_{\{c^s(U)=c\}}
        \prod_{i=1}^{\abs{c}} F_{i}(U^{*,s}_i),
    \]
    where $F_i$ are continuous and bounded functions and $U_i^s \equiv
        U_i^s(U^*)$ is as in Definition \ref{def:composition}. Under the
    assumptions of this section,
    \[
        \M_{x}^{k,t}[H^s] \ = \ \M_{x}^{\abs{c},t-s}\bigg(
        \prod_{i=1}^{|c|} \M_{X_i}^{c_i,s}[F_i]
        \bigg).
    \]
\end{lemma}

\begin{proof}
    We start by showing the case $|c|=1$ for any $k$. In this
    case, the indicator function $\indic_{\{c^s(U)=c\}}$ in  the
    functional $H^{s}$ is non zero iff $\tau(U) < s$. Therefore,
    \begin{align*}
    \M_{x}^{k,t}[H^s] 
    &= \frac{1}{h(x)} \E_x\bigg(
        \sum_{u_1<\cdots<u_k \in{\cal N}_t} 
        \indic_{\{ \abs{c^s(d_t(\mathbf{u}))} = 1 \}}
        F\left(d_{t}({\bf u}), X_{\bf u}(t)\right) 
        \prod_{i=1}^k  h(X_{u_i}(t))
    \bigg) \\
    &=  \frac{1}{h(x)} \E_{x}\bigg( \sum_{v\in {\cal N}_{t-s}} 
        \sum_{\substack{u_1<\cdots<u_k \in{\cal N}_t \\ \forall i: v\preceq u_i }} 
        F\left(d_{t}({\bf u}), X_{\bf u}(t)\right) 
        \prod_{i=1}^k  h(X_{u_i}(t))   
    \bigg) \\
    &= \frac{1}{h(x)} \E_{x}\bigg( \sum_{v\in{\cal N}_{t-s}} 
        \E_{X_v(t-s)}\bigg(\sum_{u_1<\cdots<u_k \in {\cal N}_{s}} 
        F\left(d_{s}({\bf u}, X_{\bf u}(s))\right) 
        \prod_{i=1}^k h(X_{u_i}(s))
    \bigg) \\
    &= \frac{1}{h(x)} \E_{x}\bigg(
        \sum_{u\in{\cal N}_{t-s}} h(X_v(t-s)) 
        \M^{k,s}_{X_v(t-s)}[F]
    \bigg) \\
    &= \M_{x}^{1,t-s}(\M^{k,s}_{\cdot}[F]),
    \end{align*}
    where we used the  the Markov property in the third equality and the
    fact that the functional $H^{s}$ is non zero iff $\tau(U)<s$ in the
    second inequality . Let us now show our result by induction. The
    previous result shows the case $k=1$. 
    Recall that cutting the tree encoded by $U^* \in \mathbb{U}^*_k$ at
    distance $s$ from the leaves splits $U^*$ into the subtrees
    $U_i^{*,s}\equiv U^{*,s}_i(U^*)$ and that we can think of
    the $i$-th element of the composition $c^{s}(U)$, as the number of
    the leaves in $U_i^{*,s}$. In particular for a given $U^*$,
    $(c^{s}(U); s\geq0)$ is nested in the sense that
    \[
        \forall s'\geq s, \ \ c^{s'}(U) \subseteq c^{s}(U)
    \]
    where the inclusion means that $c^{s'}(U)$ is coarser than $c^s(U)$.
    (For instance, the trivial composition $(k)$ is coarser than $(n, k-n)$
    for any $n \in [k-1]$.)

    Fix some composition $c$ of $k$ with at least two elements (the case $|c|=1$ has already been covered), some
    continuous bounded functionals $(F_i)_{i \le \abs{c}}$ as in the
    statement and let
    \begin{equation} \label{def:calc}
        \forall n \in [|c|-1], \ \ \   \ {\cal C}_{n} \ = \ \sum_{i=1}^{n} c_i.
    \end{equation}
    If $U^*$ is a binary marked tree with $c^s(U) = c$, the nestedness
    property implies the existence of $n\in[|c|-1]$ such that
    \[
        c(U) = (\mathcal{C}_{n}, k-\mathcal{C}_{n})
    \]
    and $c^{s}(U_1) = (c_1, \dots, c_n)$ and
     $c^{s}(U_2) = (c_{n+1},\dots, c_{|c|})$.
    Now, define for every $s \leq t$, 
    \begin{gather*}
        \forall U^*\in \mathbb{U}^*_{{\cal C}_n}, \ \ \
        \bar{H}^s_{l,n}(U^*) \ \coloneqq \ \indic_{\{c^s(U) = (c_1,\dots,c_{n})\}}
        \prod_{i=1}^n F_i(U^{*,s}_i), \\
        \forall U^*\in \mathbb{U}^*_{k-{\cal C}_n}, \ \
        \bar{H}^{s}_{r,n}(U^*) \ \coloneqq \
        \indic_{\{c^s(U) = (c_{n+1},\dots,c_{|c|})\}}
        \prod_{i=1}^{|c|-n} F_{i+n}(U^{*,s}_{i}).
    \end{gather*}
    The previous observation yields for $s<\tau(U)$,
    \begin{equation} \label{eq:H-decomposition}
        H^{s}(U^*) \ = \ \sum_{n=1}^{|c|-1} \bar H^s_{l,n}(U^*_1) \bar H^s_{r,n}(U^*_2)
    \end{equation}
    where we used that fact that if $c(U) = ({\cal C}_n, k-{\cal C}_n)$
    \[
        \forall i\leq n, \ \ U_i^{*,s}(U^*) = (U^*_1)^{*,s}_i,
        \ \ \forall i> n, \ \ U_i^{*,s}(U^*) = (U^*_2)^{*,s}_{i-n}.
    \]
    From \eqref{eq:recursion-f}, and since $H^{s}(U^*)=0$ for $\tau \leq
        s$, we obtain
    \begin{equation*}
        \M_{x}^{k,t}[H^s] \ = \ \int_{u\in [s,t]}
        e^{-w(t-u)} \E_{x}\left[ r(\zeta_{t-u}) h(\zeta_{t-u})
        \sum_{n=1}^{|c|-1} \M^{{\cal C}_{n},u}_{\zeta_{t-u}}[\bar H^s_{l,n}]
        \M^{k-{\cal C}_{n},u}_{\zeta_{t-u}}[\bar H^s_{r,n}]
        \right] \diff u.
    \end{equation*}
    By induction,
    \begin{gather*}
        \M^{{\cal C}_{n},u}_{y}[\bar H^s_{l,n}] \ = \
        \M_{y}^{n,u-s}\left( \prod_{i=1}^{n} \M_{X_i}^{c_i,s}[  F_i ] \right) \\
        \M^{k-{\cal C}_{n},u}_{y}[\bar H^s_{r,n}] \ = \
        \M_{y}^{|c|-n,u-s}\left( \prod_{i=1}^{|c|-n} \M_{X_i}^{c_{i+n},s}[F_{i+n}]
        \right)
    \end{gather*}
    so that
    \begin{multline*}
        \M_{x}^{k,t}[H^s] \ = \ \int_{[s,t]} e^{-w(t-u)} \E_{x}\bigg[
            r(\zeta_{t-u}) h(\zeta_{t-u})
            \sum_{n=1}^{|c|-1} \M_{\zeta_{t-u}}^{|c|,u-s}\left(
            \prod_{i=1}^{n} \M_{X_i}^{c_i,s}[  F_i ] \right)
            \\
            \times \M_{\zeta_{t-u}}^{|c|-n,u-s}\left( \prod_{i=1}^{|c|-n}
            \M_{X_i}^{c_{i+n},s}[  F_{i+n}]\right)
            \bigg]\diff u.
    \end{multline*}
    The result follows from the change of variable $\bar{u} = u-s$ and
    applying \eqref{eq:recursion-f} again on
    \[
        \bar{F}_1 = \prod_{i=1}^{n} \M_{X_i}^{c_i,s}[F_i], \ \ \
        \bar{F}_2 = \prod_{i=1}^{\abs{c}-n} \M_{X_i}^{c_{i+n},s}[F_{i+n}].
        \qedhere
    \]
\end{proof}

\begin{proof}[Proof of Proposition \ref{cor:many-to-few}]
    Let $\bar{\eps}>0$ and write $G^{\bar \eps} = G\circ S^{\bar \eps}$.
    For any $U^* \in \mathbb{U}^*_k$ such that $\tau \equiv \tau(U) >
    \bar{\eps}$, by definition of $S^{\bar{\epsilon}}$ in
    \eqref{eq:mergedMap}
    \begin{equation} \label{eq:Feps2}
        G^{\bar \epsilon}(U^*) = f(\tau(U))
        \indic_{\{ c^{\tau-\bar \epsilon}(U) = c\}}
        \prod_{i=1}^{|c|} F_i(U_{i}^{*,\tau-\bar \epsilon}).
    \end{equation}
    By reasoning as in \eqref{eq:H-decomposition}, we can decompose
    $G^{\bar \epsilon}$ as the sum of product functions on the two
    subtrees attached at the first branching point, if $\tau>\bar \eps$
    \begin{equation} \label{eq:cond-tau}
        G^{\bar \epsilon}(U^*) \ = \ f(\tau)
        \sum_{n=1}^{|c|-1} H^{\tau-\bar{\epsilon}}_{l,n}(U^*_1)
        H_{r,n}^{\tau-\bar{\epsilon}}(U^*_2),
    \end{equation}
    where for any $n\in[|c|-1]$,
    \begin{gather*}
        \forall U^*\in \mathbb{U}^*_{{\cal C}_n}, \ \
        H^u_{l,n}(U^*) \ = \ \indic_{\{c^u(U) = (c_1,\cdots,c_n)\}}
        \prod_{i=1}^n F_i(U_{i}^{*,u}), \\
        \forall U^*\in \mathbb{U}^*_{k-{\cal C}_{n}}, \ \ \
        H^u_{r,n}(U^*) \ = \
        \indic_{\{c^u(U) = (c_{n+1},\cdots, c_{|c|})\}}
        \prod_{i=1}^{|c|-n} F_{i+n}(U_i^{*,u}),
    \end{gather*}
    and $\mathcal{C}_n$ is defined as in \eqref{def:calc}. If the
    support of $f$ is included in $(\bar{\epsilon}, \infty)$,
    \eqref{eq:cond-tau} also holds even when $\tau < \bar \eps$, with
    both sides being equal to $0$.

    The next step is to extend the recursion equation
    \eqref{eq:recursion-f} to a larger class of product functionals
    \eqref{eq:productFunctional}. Assume that $k\geq 2$. For $i=1,2$,
    consider two continuous and bounded functions
    \begin{eqnarray*}
        \bar F_i: & \mathbb{U}^* \times \R_+ & \rightarrow  \R_+ \\ &
        (U^*,s) & \rightarrow  F_i^s(U^*).
    \end{eqnarray*}
    Finally, for $n\in[k]$, let $\bar G$ be of the form
    \begin{equation}
        \label{eq:prodctFunctionalE}
        \forall U^*\in \mathbb{U}_k^*, \ \  \ \ 		\bar G(U^*) = \indic_{\{c(U^*) = (n, k-n)\}}  f(\tau(U^*)) \bar F^{\tau(U^*)}_1(U^*_1) \bar F^{\tau(U^*)}_2(U^*_2).
    \end{equation}
    where $U^*_1$ and $U^*_2$ are as in Definition \ref{def:composition}. In particular, we note that if $\bar F_i$ has no dependence in the variable $s$, then $\bar G$ coincides with the product form \eqref{eq:productFunctional}. Then, we have the following generalization of the recursive equation
    (\refeq{eq:recursion-f})
    \begin{equation}\label{eq:recursion-f0}
        \M_{x}^{k,t}[\bar G] = \int_0^t e^{-w (t-u)} f(u)
        \E_x\left[ r(\zeta_{t-u}) h(\zeta_{t-u})
            \M_{\zeta_{t-u}}^{n,u}[\bar F_1^u]
            \M_{\zeta_{t-u}}^{k-n,u}[\bar F_2^u]
            \right]
        \diff u.
    \end{equation}
    In order to prove \eqref{eq:recursion-f0}, one can directly check that
    the proof of Lemma~\ref{lem:recursion-moments} (see
    \cite{schertzer2023spectral}, Proposition~9) can be directly adapted to
    functionals of the form \eqref{eq:prodctFunctionalE}. Alternatively, one
    can prove \eqref{eq:recursion-f0} as a consequence
    Lemma~\ref{lem:recursion-moments} by a monotone-class argument.

    We are now ready to complete the proof of Proposition~\ref{cor:many-to-few}.
    From \eqref{eq:cond-tau} and \eqref{eq:recursion-f0}, we get
    \[
        \M^{k,t}[G^{\bar \epsilon}] \ =\
        \int_{\bar{\epsilon}}^t f(u) e^{-w(t-u)} \sum_{n=1}^{|c|-1}
        \E_{x}\left( r(\zeta_{t-u}) h(\zeta_{t-u})
        \M^{{\cal C}_n,u}_{\zeta_{t-u}}[H^{u-\bar{\epsilon}}_{l,n}]
        \M^{k-{\cal C}_n,u}_{\zeta_{t-u}}[ H^{u-\bar{\epsilon}}_{r,n}]
        \right)\diff u.
    \]
    From Lemma~\ref{lem:recursion-moments}, this yields
    \begin{multline*}
        \mathbf{M}_x^{k,t}(G^{{\bar \epsilon}})
        =  \int_{\bar{\epsilon}}^t f(u) e^{-w(t-u)} \\
        \times \E_x\left(
        r({\zeta}_{t-u}) h({\zeta}_{t-u})
        \sum_{n=1}^{\abs{c}-1} \mathbf{M}_{{\zeta}_{t-u}}^{n,{\bar \epsilon}}\left[
        \prod_{i=1}^n {\mathbf{M}}^{c_i,u-{\bar \epsilon}}_{ X_i}[F_i]
        \right]
        \mathbf{M}_{\zeta_{t-u}}^{\abs{c}-n,{\bar \epsilon}}\left[
        \prod_{i=1}^{\abs{c}-n} {\mathbf{M}}^{c_{i+n},u-{\bar \epsilon}}_{X_i}[F_{i+n}]
        \right]
        \right) \diff u.
    \end{multline*}
    Proposition~\ref{cor:many-to-few} follows after scaling (see the
    rescaling of the moments in \eqref{eq:scaled_moments}, and
    $\bar{\zeta}_{t} = \zeta_{tN}$) and taking $\bar{\epsilon} = \epsilon/N$.
\end{proof}

\section{Proof of Theorem~\ref{thm:main-cv}}

We prove Theorem~\ref{thm:main-cv} for general branching diffusions on a
domain $\Omega_\infty$. Let us recall some notation first. This result
relies on a spatial cutoff of the process, for which particles are killed
upon exiting a spatial domain $\Omega_{AN^\gamma}$ defined in
\eqref{eq:truncatedDomain}, for some $A \ge 1$ and $N \ge 1$. Let us
remind the reader of our subscript convention. Quantities related to the
original diffusion (with no cutoff) are denoted with an $\infty$
subscript (e.g.\ $\mathcal{N}_{t,\infty}$, $d_{t, \infty}$) whereas
similar quantities constructed from the branching diffusion on
$\Omega_{AN^\gamma}$ are denoted with a $(N, A)$ subscript
(e.g.\ $\mathcal{N}_{t,(N,A)}$, $d_{t,(N,A)}$). We drop the subscript
$(N,A)$ most of the time and simply write $\mathcal{N}_t \equiv
\mathcal{N}_{t,(N,A)}$, $d_t \equiv d_{t,(N,A)}$, but may sometimes keep
it when needed.

Recall the notation $(\bar{\mathcal{N}}_t, \bar{d}_t, \bar{\nu}_t)$ for
the rescaled mmm-space, namely
\[
    \bar{\mathcal{N}}_t = \mathcal{N}_{Nt}, \qquad
    \bar{d}_t = \frac{d_{tN}}{N}, \qquad
    \bar{\nu}_t = \frac{\nu_{tN}}{N^\gamma}.
\]
We introduce another empirical measure by giving mass $h(x)$ to a
particle at $x$, namely
\[
    \hat{\nu}_{t,(N,A)} \equiv \hat{\nu}_t
    = \frac{1}{N^\gamma} \sum_{v \in \bar{\mathcal{N}}_{t}}
    h(X_v(tN)) \delta_{X_{v}(tN)}.
\]
We will refer to this measure as the biased sampling measure. We will
also use the notation
\[
    \bar{W}_t = \angle{\hat{\nu}_t, 1}.
\]

The proof of Theorem~\ref{thm:main-cv} is divided in three main steps.
\begin{enumerate}
    \item[(A)] For $A$ fixed, we let $N \to \infty$ and apply a method of
          moments to deduce the convergence of the mmm-spaces with biased
          sampling measure $(\bar{\mathcal{N}}_{t}, \bar{d}_{t},
              \hat{\nu}_{t})$.
    \item[(B)] Using the estimate on the mean number of particles, we
          remove the bias to obtain convergence of the mmm-spaces
          $(\bar{\mathcal{N}}_{t}, \bar{d}_{t}, \bar{\nu}_{t})$
          with natural sampling measures, as $N \to \infty$.
    \item[(C)] We use a perturbation result to compare the genealogy of the
          original branching diffusion to the limit as $N \to \infty$ then
          $A \to \infty$ of the process with a cutoff $(\bar{\mathcal{N}}_{t},
              \bar{d}_{t}, \bar{\nu}_{t})$. We identify the latter limit to an
          $\alpha$-stable genealogy using the fact that its moments grow
          polynomially with $A$.
\end{enumerate}
Steps (A) and (B) only involve rather standard but cumbersome limit
theorem arguments. Conversely, step (C) requires some new ideas to
identify the limiting $\alpha$-stable branching process from moment
computations.

\subsection{Step (A)}

In this section we prove the following proposition.

\begin{proposition}[Convergence of the biased mmm-space] \label{lem:convBiasedMMM}
    Fix $A \ge 1$ and $x \in \Omega_{\infty}$. Suppose that the
    assumptions of Theorem~\ref{thm:main-cv} hold and let $\psi_A$ be as
    in \eqref{eq:psiA} with $\bar{u}_{t,A}$ as in
    \eqref{eq:laplaceExponent}. Then
    \[
        \lim_{N \to \infty}  \frac{N^\gamma}{h(x)}
        \mathscr{L}_{x}(
        \bar{\mathcal{N}}_{t}, \bar{d}_{t}, \hat{\nu}_{t}
        )
        = \bar{u}_{t,A} \mathscr{L}(\mathcal{U}_{t,A}, d_{t,A},
        \vartheta_{t,A} \otimes \Pi_\infty),
    \]
    vaguely in the Gromov-weak topology, with $(\mathcal{U}_{t,A}, d_{t,A},
        \vartheta_{t,A})$ the $\psi_A$-mm space.
\end{proposition}

This result will follow from the method of moments. It first requires
a technical lemma to compare factorial biased moments with non-factorial
ones.

\begin{lemma} \label{lem:factorial2moment}
    Fix $A > 1$ and $x \in \Omega_\infty$. Let $k \ge 2$ and $c = (c_1,
    \dots, c_p)$ be a composition of $k$ such that $c \ne (1, \dots, 1)$.
    Under the assumptions of Theorem~\ref{thm:main-cv},
    \[
        \frac{1}{h(x)} \E_{x}\Bigg[
            \sum_{\substack{v_1, \dots, v_p \in \bar{\mathcal{N}}_{t}\\ \text{$(v_i)$ distinct}}}
            \prod_{i=1}^p h(X_{v_i}(tN))^{c_i}
            \Bigg]
        = o_N\big(N^{\gamma(k-1)}\big).
    \]
\end{lemma}

\begin{proof}
    By ordering the terms in the sum in increasing order and using the
    definition of the measure $\hat{\M}^{p,t}_{x}$ in
    \eqref{eq:scaled_moments},
    \begin{align}
        \frac{1}{h(x)} \E_{x}\Bigg[
           \sum_{\substack{v_1, \dots, v_p \in \bar{\mathcal{N}}_{t}    \\
                    \text{$(v_i)$ distinct}}}
            \prod_{i=1}^p h(X_{v_i}(tN))^{c_i}
        \Bigg] 
         & = k!\, N^{\gamma(p-1)} \hat{\M}^{p,t}_{x}\Big[ \prod_{i=1}^p
        h(X_i)^{c_i-1} \Big] \nonumber                                  \\
         & = k!\, N^{\gamma(k-1)} \hat{\M}^{p,t}_{x}\Big[ \prod_{i=1}^p
            \Big(\frac{h(X_i)}{N^\gamma}\Big)^{c_i-1} \Big].
        \label{eq:factorial2moment1}
    \end{align}
    The result follows by noting that:
    \begin{itemize}
        \item since $h \to h_\infty$ uniformly on compacts as $N
                  \to \infty$ by the second point of Theorem~\ref{thm:main-cv},
              $h(x) / N^\gamma \to 0$ uniformly on compacts as $N \to \infty$;
        \item since $c \ne (1,\dots,1)$, $c_i > 1$ for some $i \le p$
              and so the integrand in \eqref{eq:factorial2moment1}
              vanishes uniformly on compacts;
        \item this integrand is bounded by the second point of
              Theorem~\ref{thm:main-cv}: for all $y \in
                  \Omega_{AN^\gamma}$, $h(y) \le C \sup_{\Omega_{AN^\gamma}}
                  h_{\infty} \le CAN^\gamma$;
        \item the sequence of measures $(\hat{\M}^{p,t}_{x})_{N \ge 1}$ is tight
              since it  converges weakly as $N \to \infty$ by
              \eqref{eq:convergenceMoment}.
              \qedhere
    \end{itemize}
\end{proof}

\begin{proof}[Proof of Proposition~\ref{lem:convBiasedMMM}]
    Let $F \colon \mathbb{U}^*_k \to \R$ be continuous bounded and let
    $\Phi$ be the corresponding polynomial as in \eqref{eq:defPolynomial}.
    Then, re-ordering the individuals as
    in~\eqref{eq:biasedMomentBBM}-\eqref{eq:scaled_moments}
    \begin{align*}
        N^\gamma
        \E_{x}\big[
            \Phi\big(
            \bar{\mathcal{N}}_{t}, \bar{d}_{t}, \hat{\nu}_{t}
            \big) \big]
         & = \frac{1}{N^{\gamma(k-1)}}
        \E_{x}\Big[
            \sum_{v_1, \dots, v_k \in \bar{\mathcal{N}}_{t}}
            F\big(\tfrac{d_{tN}(\mathbf{v})}{N}, X_{\mathbf{v}}(tN) \big)
            \prod_{i=1}^k h(X_{v_i}(tN))
        \Big]                                                       \\
         & = \frac{1}{N^{\gamma(k-1)}}
        \E_{x}\Big[
        \sum_{\substack{v_1, \dots, v_k \in \bar{\mathcal{N}}_{t}   \\
                    \text{$(v_i)$ distinct}}}
            F\big(\tfrac{d_{tN}(\mathbf{v})}{N}, X_{\mathbf{v}}(tN) \big)
            \prod_{i=1}^k h(X_{v_i}(tN))
        \Big] + R                                                   \\
         & = \sum_{P} h(x) \hat{\mathbf{M}}^{k,t}_x[F \circ P] + R,
    \end{align*}
    where the remainder $R$ is obtained by removing the terms $(v_1,
        \dots, v_k)$ in the sum such that $v_i = v_j$ for some $i \ne j$.
    Letting $i \sim j$ iff $v_i = v_j$ defines a partition of
    $\{1, \dots, k\}$. Summing over all $k$-tuples of elements of
    $\bar{\mathcal{N}}_t$ that are not all distinct amounts to summing
    over all possible such partitions $\pi = (\pi_i)_i$ with $\abs{\pi}
        < k$ blocks, and then over all ways to assign distinct elements of
    $\bar{\mathcal{N}}_t$ to the blocks. This leads to the bound
    \[
        \abs{R} \le \sum_{\pi \ne \{ \{1\}, \dots, \{k\} \}}
        \frac{\kappa_\pi}{N^{\gamma(k-1)}} \norm{F}_\infty
        \E_{x}\Big[
        \sum_{\substack{v_1, \dots, v_{\abs{\pi}} \in \bar{\mathcal{N}}_{t}\\
                \text{$(v_i)$ distinct}}}
        \prod_{i=1}^{\abs{\pi}} h(X_{v_i}(tN))^{\abs{\pi_i}}
        \Big],
    \]
    where the sum is taken over all partitions of $\{1,\dots,k\}$ and
    $\kappa_\pi$ is a combinatorial factor. By Lemma~\ref{lem:factorial2moment},
    $\lim_{N \to \infty} R = 0$ and hence by \eqref{eq:convergenceMoment}
    \begin{equation} \label{eq:convBiasedMMM1}
        \lim_{N \to \infty} \frac{N^\gamma}{h(x)}
        \E_{x}\big[ \Phi\big(
            \bar{\mathcal{N}}_{t}, \bar{d}_{t}, \hat{\nu}_{t}
            \big) \big]
        = \sum_{P} \big(\hat{\mathcal{M}}^{k,t}_{A} \otimes
        (h_\infty\tilde{h}_\infty)^{\otimes k}\big)\big[F \circ P\big].
    \end{equation}
    By Proposition~\ref{prop:momentMmSpace}, the right-hand side is the
    $k$-th moment of the $\psi_A$ mm-space, multiplied by
    $\bar{u}_{t,A}$. Moreover, the measure $\Lambda_A$ in
    \eqref{eq:rescaled_lambda} has a finite exponential moment moment in
    the sense of \eqref{eq:finiteExpMoment}, since $\Lambda$ does so. It
    is not hard to see that this entails that the corresponding
    $\psi_A$-CSBP (and thus the $\psi_A$-mm space) also has a finite
    exponential moment. Therefore, a version of the method of moments for
    vague convergence in the Gromov-weak topology, \cite[Theorem~3.1]{foutel24vague}, 
    allows us to deduce that
    \[
        \lim_{N \to \infty} \frac{N^\gamma}{h(x)}
        \mathscr{L}_{x}\big( \bar{\mathcal{N}}_{t}, \bar{d}_{t}, \hat{\nu}_{t} \big)
        =
        \bar{u}_{t,A} \mathscr{L}\big( \mathcal{U}_{t,A}, d_{t,A},
        \vartheta_{t,A} \otimes \Pi_\infty \big)
    \]
    vaguely in the Gromov-weak topology.
\end{proof}

\subsection{Step (B)}

We now remove the bias and prove convergence of the branching diffusion
endowed with its natural empirical measure. For a measure $\nu$ on $X
\times \Omega_\infty$ and a map $f \colon \Omega_\infty \to \R_+$, we
will use the notation $f\cdot\nu$ for the measure whose Radon--Nikodym
derivative w.r.t.\ $\nu$ is $(x,y) \in X \times \Omega_\infty \mapsto
f(y) \in \R_+$.

\begin{proposition}[Convergence of the unbiased mmm-space] \label{lem:removeBias}
    Fix $A \ge 1$ and $x \in \Omega_{\infty}$. Under the
    assumptions of Theorem~\ref{thm:main-cv},
    \[
        \lim_{N \to \infty} \frac{N^\gamma}{h(x)}
        \mathscr{L}_{x}(
        \bar{\mathcal{N}}_{t}, \bar{d}_{t}, \bar{\nu}_{t}
        )
        = \bar{u}_{t,A} \mathscr{L}(\mathcal{U}_{t,A}, d_{t,A},
        \vartheta_{t,A} \otimes \tilde{h}_\infty),
    \]
    vaguely for the Gromov-weak topology, where $(\mathcal{U}_{t,A}, d_{t,A},
        \vartheta_{t,A})$ is the $\psi_A$-mm space corresponding to the
    same branching mechanism as in Propositin~\ref{lem:convBiasedMMM}. Moreover,
    if $\bar{\Xi}_{t,A} \sim \vartheta_{t,A}(\mathcal{U}_{t,A})$
    \begin{equation} \label{eq:limitSmallMass}
        \forall c > 0,\quad
        \lim_{N \to \infty} \frac{N^\gamma}{h(x)}
        \E_{x}[ \bar{Z}_{t} \indic_{\{ \bar{Z}_{t} \le c \}} ]
        = \E[ \bar{\Xi}_{t,A} \indic_{\{ \bar{\Xi}_{t,A} \le c \}} ].
    \end{equation}
\end{proposition}

\begin{proof}
    The two sampling measures are related through
    \[
        \bar{\nu}_{t} = \frac{1}{h} \cdot \hat{\nu}_{t}.
    \]
    Since $1/h$ is not bounded, we cannot directly deduce the convergence
    of $\bar{\nu}_{t}$ from that of $\hat{\nu}_{t}$. We
    obtain our result following the same strategy as in
    \cite[Section~7.2]{schertzer2023spectral}: we first remove particles
    at locations $y$ such that $h_\infty(y) < \delta$ for some small
    $\delta$ and then let $\delta \to 0$.

    \medskip
    \noindent
    \emph{Step 1. Convergence for $\delta$ fixed.}
    Fix $\delta > 0$ and let $H$ and $H_\infty$ be the maps
    \[
        H \colon (X,d,\nu) \mapsto \big(
        X, d, \tfrac{\indic(h(y) > \delta)}{h(y)} \cdot \nu
        \big),
        \qquad
        H_\infty \colon (X,d,\nu) \mapsto \big(
        X, d, \tfrac{\indic(h_\infty(y) > \delta)}{h_\infty(y)} \cdot \nu
        \big).
    \]
    We want to apply a version of the continuous mapping theorem for
    vague convergence. It is not hard to verify that, if $(X_N, d_N,
        \nu_N)_{N \ge 1}$ are mmm-spaces converging to a limit $(X, d, \nu)$,
    then, since $h \to h_\infty$ uniformly on compacts,
    \[
        H(X_N, d_N, \nu_N) \to H_\infty(X, d, \nu)
    \]
    (provided that $\nu(\{ y \in \Omega_\infty : h_\infty(y) = \delta \})
        = 0$). This follows directly from the characterisation of the
    Gromov-weak convergence in terms of isometric embeddings in a common metric
    space (for instance \cite[Lemma~5.8]{Greven2009}).
    Therefore, \cite[Theorem~3.3]{foutel24vague} allows us to deduce
    from Proposition~\ref{lem:convBiasedMMM} that
    \begin{equation} \label{eq:removeBias-1}
        \lim_{N \to \infty} \frac{N^\gamma}{h(x)}
        \mathscr{L}_{x} \big(\bar{\mathcal{N}}_{t}, \bar{d}_{t},
        \indic_{\{ h(y) > \delta\}} \cdot \bar{\nu}_{t} \big)
        = \bar{u}_{t,A} \mathscr{L}(\mathcal{U}_{t,A}, d_{t,A},
        \vartheta_{t,A} \otimes
        \indic_{\{ h_\infty(y) > \delta\}} \tilde{h}_\infty).
    \end{equation}

    \medskip
    \noindent
    \emph{Step 2. Letting $\delta \to 0$.} We will use the perturbation
    result for metric measure spaces in \cite[Theorem~5.1]{foutel24vague}.
    Denote the rescaled number of particles of the branching diffusion in
    $\Omega_{AN^\gamma}$ whose mark belongs to $\{ y \in \Omega_\infty :
        h(y) \le \delta\}$ as
    \[
        \bar{Z}^{\delta}_{t,(N,A)} \equiv
        \bar{Z}^{\delta}_{t} = \frac{1}{N^\gamma}
        \sum_{v \in \bar{\mathcal{N}}_{t}}
        \indic_{\{ h(X_v(tN)) \le \delta \}}.
    \]
    In view of \cite[Theorem~5.1]{foutel24vague}, we need to prove that
    \begin{equation} \label{eq:uniformSmallDelta}
        \forall \epsilon > 0,\quad
        \adjustlimits\lim_{\delta \to 0}\limsup_{N \to \infty}
        \frac{N^\gamma}{h(x)}
        \P_{x}\big(
        \bar{Z}^\delta_{t} \ge \epsilon
        \big) = 0.
    \end{equation}
    By \eqref{eq:removeBias-1} and a vague convergence version of
    the continuous mapping theorem, \cite[Corollary~3.4]{foutel24vague},
    \[
        \lim_{N \to \infty} \frac{N^\gamma}{h(x)}
        \mathscr{L}_{x}( \bar{Z}_{t} - \bar{Z}^\delta_{t} )
        = \bar{u}_{t,A} \mathscr{L}\Big( \bar{\Xi}_{t,A}
        \int_{\Omega_\infty} \indic_{\{h_{\infty}(y) > \delta\}} \tilde{h}_\infty(y) \diff y
        \Big),
    \]
    vaguely as measures on $(0, \infty]$, where $\bar{\Xi}_{t,A}$ is
    distributed as $\vartheta_{t,A}(\mathcal{U}_{t,A})$. As a
    consequence, an adaptation of Fatou's lemma for vague convergence
    (Lemma~\ref{lem:fatouVague}) and assumption \eqref{eq:mainCV2}
    show that
    \begin{align*}
        \limsup_{N \to \infty}
        \frac{N^\gamma}{h(x)} \E_{x}[ \bar{Z}^\delta_{t} ]
         & = \lim_{N \to \infty} \frac{N^\gamma}{h(x)}
        \E_{x}[ \bar{Z}_{t}]
        - \liminf_{N \to\infty} \frac{N^\gamma}{h(x)}
        \E_{x}[ \bar{Z}_{t}-\bar{Z}^\delta_{t} ]       \\
         & \le \bar{u}_{t,A} \E[\bar{\Xi}_{t,A}]
        - \bar{u}_{t,A} \E[\bar{\Xi}_{t,A}]
        \int_{\Omega_\infty} \indic_{\{h_{\infty}(y) > \delta\}}
        \tilde{h}_\infty(y) \diff y.
    \end{align*}
    Markov's inequality proves \eqref{eq:uniformSmallDelta}, which
    shows that point~(ii) of \cite[Theorem~5.1]{foutel24vague} is
    satisfied. Therefore,
    \[
        \lim_{N \to \infty} \frac{N^\gamma}{h(x)}
        \mathscr{L}_x\big( \bar{\mathcal{N}}_t, \bar{d}_t, \bar{\nu}_t \big)
        = \lim_{\delta \to 0}
        \bar{u}_{t,A} \mathscr{L}\big(
        \mathcal{U}_{t,A}, d_{t,A}, \vartheta_{t,A} \otimes
        \indic_{\{ h_\infty(y) > \delta\}} \tilde{h}_\infty
        \big)
        = \bar{u}_{t,A} \mathscr{L}\big(
        \mathcal{U}_{t,A}, d_{t,A}, \vartheta_{t,A} \otimes \tilde{h}_\infty
        \big)
    \]
    vaguely in the Gromov-weak topology.

    It remains to prove the second part of the statement. Fix $c > 0$.
    By the first part of the proof $N^\gamma \mathscr{L}_x(\bar{Z}_t
        \indic_{\{ \bar{Z}_t \le c \}}) / h(x) \to \bar{u}_{t,A}
        \mathscr{L}(\bar{\Xi}_{t,A} \indic_{\{\bar{\Xi}_{t,A} \ge c \}})$
    vaguely as $N \to \infty$. Fatou's lemma for vague
    convergence (Lemma~\ref{lem:fatouVague}) shows that
    \begin{gather}
        \begin{split} \label{eq:removeBias1}
            \liminf_{N \to \infty}
            \frac{N^\gamma}{h(x)}
            \E_{x}\big[\indic_{\{ \bar{Z}_{t} > c\}} \bar{Z}_{t}\big]
            \ge \bar{u}_{t,A} \E\big[ \indic_{\{ \bar{\Xi}_{t,A} > c\}}
            \bar{\Xi}_{t,A} \big], \\
            \liminf_{N \to \infty}
            \frac{N^\gamma}{h(x)}
            \E_{x}\big[\indic_{\{ \bar{Z}_{t} \le c\}} \bar{Z}_{t}\big]
            \ge \bar{u}_{t,A} \E\big[ \indic_{\{ \bar{\Xi}_{t,A} \le c\}}
                \bar{\Xi}_{t,A} \big].
        \end{split}
    \end{gather}
    By our assumption \eqref{eq:mainCV2} we also know that
    \[
        \lim_{N \to \infty} \frac{N^\gamma}{h(x)}
        \E_{x}[\bar{Z}_{t}]
        = \bar{u}_{t,A} \E[ \bar{\Xi}_{t,A} ].
    \]
    The only way that this limit occurs together with \eqref{eq:removeBias1}
    is that both limits inferior in \eqref{eq:removeBias1} are limits, and both
    inequalities are equalities.
\end{proof}

\subsection{Step (C)}

Completing the last step requires us to show two claims:
\begin{enumerate}
    \item The limit as $A \to \infty$ of the $\psi_A$-mm space is the
          genealogy of the $\alpha$-stable branching process.
    \item The limits as $N \to \infty$ and as $A \to \infty$ of the
          branching diffusion on $\Omega_{AN^\gamma}$ can be interchanged.
\end{enumerate}
For the first claim, we will rely on the following result.

\begin{proposition} \label{lem:invarianceProperty}
    Let $(\psi_{A})_{A \ge 1}$ be a collection of branching mechanisms
    converging pointwise to a limit $\psi$ as $A \to \infty$. Suppose
    that $\psi_A$ is as in \eqref{eq:psiA}, with $a_A, b_A \to 0$
    as $A \to \infty$ and $\Lambda_A$ defined in \eqref{eq:rescaled_lambda}
    from some measure $\Lambda$.
    Then $\psi(\theta) = \psi(1) \theta^\alpha$, $\theta \ge 0$.
\end{proposition}

\begin{proof}
    By definition of $\Lambda_A$,
    \begin{align*}
        \forall \theta \ge 0,\quad \psi_A(\theta)
         & = a_A \theta + \tfrac{b_A}{2} \theta^2 +
        \int_0^\infty \big( e^{-\theta y} - 1 + \theta y \big) \Lambda_A(\diff y) \\
         & = a_A \theta + \tfrac{b_A}{2} \theta^2
        + A^{-\alpha} \int_0^\infty \big( e^{-A \theta y} - 1 + A \theta y \big)
        \Lambda(\diff y)                                                          \\
         & = a_A \theta + \tfrac{b_A}{2} \theta^2
        + \theta^{\alpha} \int_0^\infty \big( e^{-y} - 1 + y \big)
        \Lambda_{A\theta}(\diff y).
    \end{align*}
    This shows that
    \[
        \psi(\theta) = \lim_{A \to \infty} \psi_A(\theta) = \lim_{A \to \infty}
        \theta^\alpha \psi_{A\theta}(1) = \theta^\alpha \psi(1).
        \qedhere
    \]
\end{proof}

Interchanging the two limits will require the following uniform estimate
on the number of particles being killed by adding the spatial cutoff.

\begin{lemma}\label{lem:particlesAbsorbed}
    Let $R_t \equiv R_{t,(N,A)}$ be the number of particles of the
    branching diffusion on $\Omega_\infty$ that exit the subdomain
    $\Omega_{AN^\gamma}$ before time $t$.
    Then for every $x \in \Omega_{AN^\gamma}$
    \[
        \adjustlimits\lim_{A\to\infty}\limsup_{N\to\infty}
        \frac{N^\gamma}{h(x)} \P_{x}( R_{tN} > 0 )
        = 0.
    \]
\end{lemma}

\begin{proof}
    Let us consider the process in $\Omega_\infty$. Since $h_\infty$ is
    harmonic, the process
    \[
        W_{t,\infty} \coloneqq \sum_{u\in{\cal N}_{t,\infty}} h_\infty(X_{v}(t))
    \]
    is an additive martingale.
    Recall from \eqref{eq:truncatedDomain} that
    \[
        \Omega_{AN^\gamma} = \{ y \in \Omega_\infty : h_\infty(y) < A N^\gamma \}.
    \]
    Thus,
    \begin{align*}
        \P_x( R_{tN} > 0 )
         & = \P_x\big(
        \exists s \in [0,tN], v \in \bar{\mathcal{N}}_{t} :
        X_{v}(s) \notin \Omega_{AN^\gamma}
        \big)              \\
         & \le \P_{x}\big(
        \sup_{s \in [0,tN]} W_{s,\infty} \ge A N^\gamma
        \big)
    \end{align*}
    By Doob's maximal inequality, it follows that
    \begin{equation*}
        \P_{x}( R_{tN} > 0 )
        \le \frac{h_{\infty}(x)}{A N^\gamma},
    \end{equation*}
    hence the result recalling that $h(x) \to h_\infty(x)$ as $N \to
        \infty$ by assumption.
\end{proof}

We are now in position to prove Theorem~\ref{thm:main-cv}. Recall that
$\bar{Z}_{t, \infty} = Z_{tN, \infty} / N^\gamma$ is the rescaled number
of particles at time $tN$ in the process on the full domain $\Omega_\infty$.

\begin{proof}[Proof of Theorem~\ref{thm:main-cv}]
    The proof is divided into two steps:
    \begin{enumerate}
      \item We use the fact that the mmm-space with cutoff $(\bar{\mathcal{N}}_{t},
        \bar{d}_{t}, \bar{\nu}_{t})$ converges for every $A\geq 1$ (see Step (B))
        together with a  perturbation theorem from~\cite{foutel24vague} to prove 
        that the mmm-space without cutoff $(\bar{\mathcal{N}}_{t,\infty}, 
        \bar{d}_{t,\infty}, \bar{\nu}_{t,\infty})$ also converges to a limit.
      \item It then remains to prove that this limiting mmm-space is a
      marked $\psi$-mm space, with $\psi(\theta) = C \theta^\alpha$ for some $C > 0$,
      that is to prove that the limits $N\to\infty$ and $A\to\infty$ commute. 
      To do so,
      \begin{enumerate}
        \item[(i)] we first prove that the total mass of the limiting mmm-space 
        is distributed as a $\psi$-CSBP;
        \item[(ii)] we then lift this result to metric spaces using 
        Proposition~\ref{prop:equivalenceBranching}.
      \end{enumerate}
    \end{enumerate}
    In this proof, will use repeatedly that, $h(x) / h_\infty(x) \to 1$ as $N \to
        \infty$ for any $A \ge 1$ without further mention of this fact. 
        
        \textit{Step 1.} On
    the event $\{ R_{tN} = 0 \}$ that no particle exit
    $\Omega_{AN^\gamma}$ before time $tN$, the branching diffusion on
    $\Omega_\infty$ and that on $\Omega_{AN^\gamma}$ are perfectly
    coupled until time $tN$. Therefore by Lemma~\ref{lem:particlesAbsorbed}
    \[
        \adjustlimits\lim_{A \to \infty}\limsup_{N \to \infty}
        \frac{N^\gamma}{h_{\infty}(x)}
        \P_{x}\big( \bar{Z}_{t,\infty} - \bar{Z}_{t} \ge \epsilon \big)
        \le
        \adjustlimits\lim_{A \to \infty} \limsup_{N \to \infty}
        \frac{N^\gamma}{h_{\infty}(x)}
        \P_{x}\big( R_{tN} > 0 \big) = 0.
    \]
    Combining this uniform bound with step (B) (Proposition~\ref{lem:removeBias})
    a perturbation theorem for vague convergence \cite[Theorem~5.1]{foutel24vague}
    shows that there exists a (possibly infinite) measure $M_t$ on
    the space of mmm-spaces such that
    \begin{equation} \label{eq:inversionAN}
        \lim_{N \to \infty} \frac{N^\gamma}{h_{\infty}(x)}
        \mathscr{L}_{x}(\bar{\mathcal{N}}_{t,\infty},
        \bar{d}_{t,\infty}, \bar{\nu}_{t,\infty})
        = \lim_{A \to \infty} \bar{u}_{t,A}
        \mathscr{L}(
        \mathcal{U}_{t,A}, d_{t,A}, \vartheta_{t,A} \otimes \tilde{h}_\infty
        )
        = M_t,
    \end{equation}
    vaguely in the marked Gromov-weak topology. 
    
    \textit{Step 2.(i).} All what remains to be
    shown is that $M_t$ is a multiple of the law of a $\psi$-mm space, with
    $\psi(\theta) = C \theta^\alpha$ for some $C > 0$. In this step, we 
    show that $M_t$ is finite and that the total mass of the mmm-space
    under $M_t$ is distributed as an $\alpha$-stable CSBP.
    We first prove that
    \begin{equation} \label{eq:L1convFinal}
        \forall c > 0,\quad
        \lim_{N \to \infty} \frac{N^\gamma}{h_{\infty}(x)} \E_{x}\big[
            \bar{Z}_{t,\infty} \indic_{\{ \bar{Z}_{t,\infty} \le c \}} \big]
        = \int_{\{\abs{X} \le c\}} \abs{X} \,\diff M_t(X,d,\nu),
    \end{equation}
    where we used the notation $\abs{X} = \nu(X \times E)$. Note that
    (reintroducing the subscript $(N,A)$ for clarity), for $A' > A$ and
    $c > 0$:
    \begin{equation} \label{eq:L1coupling}
        \frac{N^\gamma}{h_{(N,A)}(x)}
        \E_{x}\big[ \bar{Z}_{t,(N,A')}\indic_{\{ \bar{Z}_{t,(N,A')} \le c\}}
            - \bar{Z}_{t,(N,A)}\indic_{\{ \bar{Z}_{t,(N,A)} \le c\}}
            \big]
        \le \frac{cN^\gamma}{h_{(N,A)}(x)} \P_{x}( R_{tN, (N,A)} > 0).
    \end{equation}
    By Lemma \ref{lem:particlesAbsorbed}, the right-hand side of the
    above inequality vanishes uniformly in $N$ as $A \to \infty$.
    Therefore, using a diagonal argument and \eqref{eq:limitSmallMass},
    proving \eqref{eq:L1convFinal} boils down to showing that
    \begin{equation} \label{eq:proofMainCV-1}
        \lim_{A \to \infty} \bar{u}_{t,A}
        \E\big[ \indic_{\{\bar{\Xi}_{t,A} \le c\}}
            \bar{\Xi}_{t,A} \big]
        = \int_{\{\abs{X} \le c\}} \abs{X} \,\diff M_t(X,d,\nu),
    \end{equation}
    where $\bar{\Xi}_{t,A}$ is distributed as $\vartheta_{t,A}(\mathcal{U}_{t,A})$.
    For any $\delta >0 $, letting $N \to \infty$ in \eqref{eq:L1coupling},
    \[
        \limsup_{A' \to \infty} \bar{u}_{t,A'}
        \E\big[ \indic_{\{\bar{\Xi}_{t,A'} \le \delta\}}
            \bar{\Xi}_{t,A'} \big]
        \le
        \bar{u}_{t,A} \E\big[ \indic_{\{\bar{\Xi}_{t,A} \le \delta\}}
            \bar{\Xi}_{t,A} \big]
        + \limsup_{N \to \infty}
        \frac{cN^\gamma}{h_{(N,A)}(x)} \P_{x}( R_{tN, (N,A)} > 0).
    \]
    Letting $\delta \to 0$ then $A \to \infty$ with
    Lemma~\ref{lem:particlesAbsorbed} in this inequality,
    \[
        \adjustlimits\lim_{\delta \to 0}\limsup_{A' \to \infty}
        \bar{u}_{t,A'} \E\big[ \indic_{\{\bar{\Xi}_{t,A'} \le \delta\}}
            \bar{\Xi}_{t,A'} \big]
        = 0.
    \]
    Moreover, the vague convergence \eqref{eq:inversionAN} entails that
    \[
        \forall \delta > 0,\quad \lim_{A \to \infty} \bar{u}_{t,A}
        \E\big[ \indic_{\{\delta < \bar{\Xi}_{t,A} \le c\}}
            \bar{\Xi}_{t,A} \big]
        = \int_{\{\delta < \abs{X} \le c\}} \abs{X} \,\diff M_t(X,d,\nu).
    \]
    These two limits prove that \eqref{eq:proofMainCV-1} (and therefore
    \eqref{eq:L1convFinal}) holds.

    Now, we prove that $(\psi_A)_A$ converges as $A \to \infty$ to an
    $\alpha$-stable branching mechanism. Let $(\Xi_{t,A}(z))_{z \ge 0}$
    be a compound Poisson process with jump rate $\bar{u}_{t,A}$ and
    jumps distributed as $\bar{\Xi}_{t,A}$. (Note that here $z$, not
    $t$, is used as the time variable for the Poisson process.) It is
    well-known that $\Xi_{t,A}(z)$ is distributed as a $\psi_A$-CSBP at
    time $t$ started from $z$ \cite{bertoin2000Bolthausen}. By standard
    convergence results for subordinators
    \cite[Theorem~15.14]{kallenberg2002foundations},
    \eqref{eq:inversionAN} and \eqref{eq:L1convFinal} show that
    $(\Xi_{t,A}(z))_{z \ge 0}$ has a limit $(\Xi_t(z))_{z \ge 0}$ in law
    as $A \to \infty$. Moreover, $(\Xi(z)_t)_{z \ge 0}$ is a subordinator
    whose jump measure is the pushforward of $M_t$ by the map $(X,d,\nu)
        \mapsto \abs{X}$. Proposition~\ref{prop:equivalenceBranching} shows
    that there exists a branching mechanism $\psi$ such that
    \[
        \forall \theta > 0,\quad \lim_{A \to \infty} \psi_A(\theta) =
        \psi(\theta),
    \]
    and by Proposition~\ref{lem:invarianceProperty}, $\psi(\theta) =
        \psi(1) \theta^\alpha$. Since $\alpha \in (1,2)$, the $\alpha$-stable
    CSBP satisfies Grey's condition \eqref{eq:greyCondition} and has a
    first finite, and we can define the corresponding limit Laplace
    exponent $\bar{u}_t$ and $\psi$-mm space $(\mathcal{U}_t, d_t,
        \vartheta_t)$. In particular, this shows that $M_t$ is a finite
    measure with mass $\bar{u}_t$, and that the population size $\abs{X}$
    under $M_t / \bar{u}_t$ is distributed as $\bar{\Xi}_t$, the entrance
    law of the $\psi$-CSBP.

    \textit{Step 2.(ii).} We now apply the second part of
    Proposition~\ref{prop:equivalenceBranching}.     Let us start by showing
    that $\lim_{A \to \infty} \bar{u}_{t,A} = \bar{u}_t$. By killing the
    particles hitting the boundary $\Omega_{AN^\gamma}$, there is a
    natural coupling between the branching process on the full domain
    $\Omega_\infty$ and that on $\Omega_{AN^\gamma}$ such that
    \[
        \forall t \ge 0,\quad  \bar{Z}_t \le \bar{Z}_{t,\infty}.
    \]
    Therefore, by the vague convergence as $N \to \infty$ obtained in the
    first part of the proof,
    \begin{align*}
        \forall z > 0,\quad
        \bar{u}_{t,A} \P( \bar{\Xi}_{t,A} \ge z )
         & = \lim_{N \to \infty} \frac{N^\gamma}{h_\infty(x)}
        \P_x( \bar{Z}_{t} \ge z )                               \\
         & \le \lim_{N \to \infty} \frac{N^\gamma}{h_\infty(x)}
        \P_{x}( \bar{Z}_{t,\infty} \ge z )                      \\
         & = \bar{u}_t \P( \bar{\Xi}_t \ge z ) .
    \end{align*}
    Letting $z \to 0$, this shows that $\bar{u}_{t,A} \le \bar{u}_t$.
    On the other hand, using the vague convergence as $A \to \infty$,
    \[
        \forall z > 0,\quad \bar{u}_t \P(\bar{\Xi}_t \ge z)
        = \lim_{A \to \infty} \bar{u}_{t,A} \P(\bar{\Xi}_{t,A} \ge z)
        \le \liminf_{A \to \infty} \bar{u}_{t,A}.
    \]
    Letting $z \to 0$, we deduce that $\lim_{A \to \infty} \bar{u}_{t,A}
        = \bar{u}_t$.

    To end the proof, note that the Poisson structure of CSBPs entails
    that $\bar{u}_{t,A} \E[\bar{\Xi}_{t,A}] = \E[\Xi_{t,A}(1)] = e^{-a_A
                t}$. Thus, since $a_A \to 0$ as $A \to \infty$ by assumption,
    $\E[\bar{\Xi}_{t,A}] \to 1 = \E[\bar{\Xi}_t]$ as $A \to \infty$.
    By Proposition~\ref{prop:equivalenceBranching}, we can conclude that
    \[
        \lim_{A \to \infty} \bar{u}_{t,A}
        \mathscr{L}(\mathcal{U}_{t,A}, d_{t,A}, \vartheta_{t,A})
        =
        \bar{u}_t \mathscr{L}(\mathcal{U}_t, d_t, \vartheta_t)
    \]
    weakly as measures on mm-spaces. In particular, the latter convergence
    also holds vaguely, and the limit $M_t$ obtained in
    \eqref{eq:inversionAN} is
    \[
        M_t = \bar{u}_t \mathscr{L}(\mathcal{U}_t, d_t, \vartheta_t \otimes \tilde{h}_\infty),
    \]
    which ends the proof.
\end{proof}

\section{Branching Brownian motion in an interval}

In this section, we study the BBM with a cutoff introduced in
Section~\ref{sec:heuristics_BBM}. Recall that the cutoff $L$ is chosen as
in~\eqref{def:LA0}, namely
\[
    L\equiv L_{(N,A)} \coloneqq \frac{1}{2\beta}\log N+\frac{1}{\mu-\beta}\log A.
\]
\subsection{Spectral theory}
\label{sec:spectral_theory}

We first recall known facts about the spectral decomposition of the
Sturm--Liouville problem \eqref{SLP} below. For a detailed exposition of
the subject, we refer the reader to \cite[Section 4.1]{schertzer2023spectral}.
A solution of the Sturm--Liouville problem
\begin{equation}\tag{SLP}
    \frac{1}{2} u''(x)+W(x)u(x)=\rho u(x), \qquad  x\in(0,L), \quad
\end{equation}
with boundary conditions
\begin{equation}
    \tag{BC}
    u(0)=u(L)=0,
\end{equation}
is a function $u:[0,L]\to\mathbb{R}$ such that $u$ and $u'$ are
absolutely continuous on $[0,L]$ and satisfy \eqref{SLP} a.e. on $(0,L)$.
Since $W$ is continuous on $[0,L]$, the solutions are twice
differentiable on $[0,L]$ and \eqref{SLP} holds for all $x\in(0,L)$. A
complex number $\rho$ is an eigenvalue of \eqref{SLP} if \eqref{SLP}
has a solution $v$ that is not identically zero on $[0,L]$. The set of
eigenvalues is referred to as the spectrum. The spectrum of \eqref{SLP}
is infinite, countable and has no finite accumulation point. It is upper
bounded
and the maximal eigenvalue $\lambda_{(N,A)}\equiv\lambda$ is referred to as the principal
eigenvalue. Under \eqref{Hwp} and for $L$ large enough, this eigenvalue is
positive. We denote by $v_{(N,A)} \equiv v$ the associated eigenvector,
renormalised to satisfy $v(1) = 1$. It is known that the function $v$
does not change sign in $(0,L)$. Hence, we have $v>0$ on $(0,L)$ and we
note that
\begin{equation}\label{formula:v1}
     v(x)= \frac{\sinh\left(\sqrt{2\lambda}(L-x)\right)}%
    {\sinh\left(\sqrt{2\lambda}(L-1)\right)}, \quad x\in[1,L].
\end{equation}
As explained in the heuristics (see \eqref{eq:def_h}), the function
\[
    h(x)=c_{h} e^{\mu x} v(x),
\]
is the unique (up to a
multiplicative constant) harmonic function associated with the BBM with
drift $-\mu$, killed upon exiting $(0,L)$. The corresponding spine
process, defined as the Doob $h$-transform of the mean semigroup of the
BBM, solves the SDE
\[
    \diff \zeta_t=\frac{v'(\zeta_t)}{v(\zeta_t)}\,\diff t + \diff B_t,
\]
where $(B_t)_{t \ge 0}$ is a standard Brownian motion. Let $\sigma =
\inf\{t \geq 0 : \zeta_t \in[0,1]\}$. Using~\eqref{formula:v1}, we have
\begin{equation} \label{eq:SDE_drift}
    \diff \zeta_{t\wedge \sigma}
    = F\bigl(\sqrt{2\lambda},\, L-\zeta_{t\wedge \sigma}\bigr)\,\diff t
    + dB_{t\wedge \sigma}, \quad t>0,
    \qquad
    F(b,z)=-\frac{b}{\tanh(bz)}, \quad b,z>0.
\end{equation}
One can easily check that,
\begin{equation*}
    -1 < \partial_b F(b,z) < 0,
    \quad \text{and} \quad
    0 < \partial_z F(b,z) \leq \frac{1}{z^2},
\end{equation*}
so that, for all $b_1, b_2, z_1, z_2 > 0$,
\begin{equation}
    \label{eq:dev_F}
    \abs{F(b_1, z_1) - F(b_2, z_2)}
    \leq \abs{b_1 - b_2} + \abs*{\frac{1}{z_1}-\frac{1}{z_2}}.
\end{equation}

Recall that $q_t(x,y)$ refers to the transition density of the spine
process. It is well known that the spine process admits a stationary
distribution given by 
\begin{equation} \label{eq:defPi}
    \Pi(x)=c_hc_{\tilde h}v(x)^2=\tilde{h}(x)h(x),
\end{equation}
where $\tilde{h}$ is as in \eqref{eq:def_h}. Note that with this
renormalisation, $\Pi$ is not a probability distribution.

The generalized Green's function (or resolvent) of the spine process,
defined as
\[
    G_{\xi}(x,y)
    \coloneqq \int_0^\infty e^{-\xi t} q_t(x,y)\,\diff t, \qquad \xi>0,
\]
can be expressed in terms of the fundamental solutions of the ODE~\eqref{SLP}.
Let $g_\rho$  (resp.\ $d_\rho$)
denote a solution of~\eqref{SLP} satisfying $g_\rho(0) = 0$
(resp.\ $d_\rho(L)=0$). Define the Wronskian
\[
    \omega_\rho\equiv \omega_{\rho, (N,A)}
    \coloneqq d_\rho(1)g_\rho'(1)-d_\rho'(1)g_\rho(1).
\]
The functions \(g_\rho\) and \(d_\rho\) are unique up to multiplicative
constants. Without loss of generality, we fix the normalisation
\begin{equation} \label{eq:norm}
     d_\rho(x)=\frac{\sinh\bigl(\sqrt{2\rho}(L-x)\bigr)}{\sinh\bigl(\sqrt{2\rho}(L-1)\bigr)}, \; \quad x\in [1,L], \quad 
    \text{and}
    \qquad
    g_\rho'(0) = v'(0).
\end{equation}
It is a standard result (see e.g.~\cite[Proposition 4.11]{schertzer2023spectral})
that the Green's function admits the representation
\begin{equation} \label{def:green}
    G_\xi(x,y)=
    \begin{cases}
        (\omega_{\lambda_\infty+\xi})^{-1}\frac{v(y)}{v(x)}
        \,d_{\lambda_\infty+\xi}(x)\,g_{\lambda_\infty+\xi}(y),
         & 0\le y\le x\le L, \\[1ex]
        (\omega_{\lambda_\infty+\xi})^{-1}\frac{v(y)}{v(x)}
        \,g_{\lambda_\infty+\xi}(x)\,d_{\lambda_\infty+\xi}(y).
         & 0\le x\le y\le L.
    \end{cases}
\end{equation}
We begin this section by recalling results
from~\cite{schertzer2023spectral} for the harmonic function $h$, the
corresponding spine process, and its Green's function.

\begin{lemma}[Convergence of the principal 
    eigenelements~\cite{schertzer2023spectral}, Lemma~5] \label{lem:spectral}
    Under assumption~\eqref{Hwp}, the following statements hold.
    \begin{enumerate}[(i)]
        \item For every $A\geq 1$, the principal eigenvalue $\lambda$
            converges to a positive limit $\lambda_\infty$ as
            $N\to\infty$.
        \item There exists a constant $c_{w}>0$ such that, for every $A\geq 1$,
              \begin{equation} \label{eq:defW}
                w  = \lambda_\infty-\lambda
                  \sim \frac{c_{w}}{A^{\alpha-1}N},
                  \qquad \text{as } N\to\infty.
              \end{equation}
   
        \item The eigenfunction \(v\) converges pointwise, uniformly on
              compact sets and in $L^2(\R^+)$ to a limiting function
              \(v_{\infty}\) as \(N\to\infty\). Furthermore, there exist
              two constants $C,C'>1$ such that, for \(N\) large enough,
              for all $A\geq 1$, 
              \[
                  C'\,(1\wedge x\wedge (L-x))\,\mathrm{e}^{-\beta x}
                  \le v(x)
                  \le C\,(1\wedge x\wedge (L-x))\,\mathrm{e}^{-\beta x}.
              \]
    \end{enumerate}
\end{lemma}

\begin{proof}
    The result is a reformulation of \cite{schertzer2023spectral},
    Lemma~5 using that $e^{2\beta L} = A^{\frac{2\beta}{\mu-\beta}} N =
    A^{\alpha-1}N$.
\end{proof}

\begin{rem}
The third point of the lemma implies that the mass of $\Pi$ converges to
$1$ as $N\to\infty$.
\end{rem}

\begin{cor}
    \label{lem:eigenBound}
      For all $A\geq 1$, there exists a constant $C>0$ such that for
      all $N$ large enough,
      \begin{equation} \label{eq:boundIntegralPi}
          \E_{\Pi}[ h(\zeta_s) ] =\int_0^L \Pi(x) h(x) \diff x \le C  N^{\gamma-1}
          = C  N^{\gamma(2-\alpha)}.
      \end{equation}
      and
        \begin{equation}
          \label{eq:easy_bound1}
          \sup_{x \in (0,L)} h(x) \le C N^{\gamma},
          \quad
          \sup_{x \in (0,L)} h(x)\Pi(x)
          \le C N^{\gamma-1} = C N^{\gamma(2-\alpha)}.
      \end{equation}
\end{cor}

\begin{proof}
    The result follows by recalling that $\Pi(x)h(x) = C e^{\mu x}
    v^3(x) \le C' e^{(\mu-3\beta)x}$ by Lemma~\ref{lem:spectral} and
    noting that $L^{\mu - 3\beta} = (A^{\alpha-1} N)^{\frac{\mu-3\beta}{2\beta}} 
    = (A^{\alpha-1} N)^{\gamma - 1}$.
\end{proof}

\begin{proposition}[Heat kernel estimate \cite{schertzer2023spectral}, 
    Proposition~12] \label{prop:hk}
    There exists a constant $c > 0$ such that
    for all $A\geq 1$, all $t > c L$,
    \begin{equation*}
        \abs{q_t(x,y) - \Pi(y)} =o(1) \Pi(y),
    \end{equation*}
    where $o(1)$ denotes a quantity that goes to $0$ as $N\to\infty$, 
    that may depend on $A$ but not on $x,y\in(0,L)$.
\end{proposition}

\begin{lemma}[Upper bound on the Green's function
    \cite{schertzer2023spectral}, Lemma 9 and Lemma 10]
    \label{lem:boundGreen}
    Let $T>0$ and $A \geq 1$. Let $\xi\equiv \xi_{(N,A)}$ be such that 
    $\tfrac{1}{N}\leq \xi \ll \frac{1}{L}$.
    We have
    \begin{align}
        d_{\lambda_\infty+\xi}(x)
         & = O(1)\,(1\wedge (L-x))\,\mathrm{e}^{-\beta x},
        \label{est:phi2}
        \\
        g_{\lambda_\infty+\xi}(x)
         & = (1\wedge x)\Bigl(O(1)\,\mathrm{e}^{-\beta x}
        + O(\xi)\,\mathrm{e}^{\beta x}\Bigr),
        \label{est:psi2}
    \end{align}
    and the Wronskian is bounded below by
    \[
        \omega_{\lambda_\infty+\xi}
        \ge \xi\left(\int_0^L v(y)^2\,\diff y + O(\xi L)\right),
    \]
    where $O(\cdot)$ denotes a quantity bounded in absolute value by a constant 
    times the quantity inside the parentheses. This constant may depend on $T$ and $A$ 
    but not on $N$.
\end{lemma}

\subsection{Relaxation of the spine process}

This section studies the relaxation of the spine process $\zeta$ over a
time interval of length $aL$, $a > 0$. We prove two results: (i) regardless of
the starting point of the spine process, the locations of particles
produced after the first branch point are far from the right boundary
after a time $aL$ (Lemma~\ref{lem:boundSpineProcess2}); (ii) the
spine process started close to the right boundary $L$ does not reach
$x=1$ by time $aL$ with high probability if $a < 1/\beta$
(Lemma~\ref{lem:boundSpineProcess}).

\begin{lemma} \label{lem:boundSpineProcess}
    Fix $a < 1/\beta$.
    There exists $\Delta > 0$ such that, for all $A\geq 1$,
    \[
        \forall R > 0,\qquad
        \adjustlimits \limsup_{N \to \infty} \sup_{z \in (0, R)}
        \P_{L-z}\big( \inf_{s \in (0, aL)} \zeta_s \le 1 \big)
        \le N^{-\Delta}.
    \]
\end{lemma}

\begin{proof}
    Fix $R > 0$. Since the spine process has continuous paths, for all $z \in (0, R)$,
    \begin{equation}
        \label{eq:continuity}
        \P_{L-z}\big( \inf_{s \in (0,aL)} \zeta_s \le 1 \big)
        \le
        \P_{L-R}\big( \inf_{s \in (0,aL)} \zeta_s \le 1 \big).
    \end{equation}
    We want to compare $(\zeta_t)_{t \ge 0}$ to a Brownian
    motion $(X_t)_{t \ge 0}$ with an appropriate drift $c^*$. Fix some
    $0<z^*<R$ and let $c^* = F(\sqrt{2\lambda},z^*) < 0$. Then
    \begin{equation} \label{eq:comparisonSDE}
        \zeta_t - X_t
        = \int_0^t
        \Big( F(\sqrt{2\lambda}, L-\zeta_s) - c^* \Big) \diff s,
        \qquad
        X_t \coloneqq \zeta_0 + B_t + c^* t.
    \end{equation}
    Define the hitting time of $L-z^*$ by $X_t$ as
    \[
        \tau = \inf \{ t > 0 : X_t > L - z^*\}.
    \]
    We claim that $X_t \le \zeta_t$ for $t < \tau$. Indeed, it is clear
    from~\eqref{eq:comparisonSDE} (and the fact that $F$ increasing
    in its second variable) that $\zeta_t-X_t$ has positive derivative on
    the set of times such that $\zeta_t < L-z^*$ and hence it can only
    hit $0$ at a time $t$ such that $X_t = \zeta_t > L-z^*$. Therefore
    \begin{align*}
        \P_{L-R}\big( \inf_{s \in (0,aL)} \zeta_s \le 1 \big)
         & \le
        \P_{L-R}( \tau \le aL )
        + \P_{L-R}\big( \inf_{s \in (0,aL)} X_s \le 1 \big) \\
         & \le
        \P_{L-R}( \tau < \infty )
        + \P\big(
        \inf_{s \in (0,aL)} B_s \le 1 - aL c^* - L + R
        \big).
    \end{align*}
    The result follows by choosing $R$ and $z^*$ appropriately.
    First, note that $c^* = F(\sqrt{2\lambda}, z^*) \to
        -\sqrt{2\lambda}$ as $z^* \to \infty$. Further using that
    $\sqrt{2\lambda} \to \beta$ as $N \to \infty$ by
    Lemma~\ref{lem:spectral}~(i), we can choose $z^*$ large enough so
    that $\limsup_{N \to \infty} - c^* a < 1$. By the reflection
    principle and a standard Gaussian bound,
    \begin{align*}
        \P\big(
        \inf_{s \in (0, aL)} B_s \le 1 -aL c^* - L + R
        \big)
         & =
        \P\big(
        \abs{B_{aL}} \ge L + aL c^* - R - 1
        \big)                                                         \\
         & \le \exp\Big( - \frac{((1 + c^* a)L - R - 1)^2}{aL} \Big).
    \end{align*}
    Second, recall that the maximum of a Brownian motion with negative
    drift is exponentially distributed, and thus
    \[
        \P_{L-R}( \tau < \infty ) = e^{- 2 c^* (R-z^*)}.
    \]
    We obtain the desired bound by choosing $R_N = b \log N$ with $b <
        1-ac^*$ and using~\eqref{eq:continuity}.
\end{proof}

\begin{lemma} \label{lem:boundSpineProcess2}
    For any $0 < a < 1/\beta$, there exist $\mathbf{c} \in (0,1)$ and
    $\Delta > 0$ such that, for all $A\geq 1$,
    \[
        \limsup_{N \to \infty}
        \sup_{x \in (0, L)} \P_x( \zeta_{aL} \ge \mathbf{c} L)
        \le N^{-\Delta}.
    \]
\end{lemma}

\begin{proof}
    Fix some $\mathbf{c} \in (0,1)$.
    Since $F(\sqrt{2\lambda}, z) \le -\sqrt{2\lambda}$,
    Lemma~\ref{lem:boundSpineProcess} and \eqref{eq:SDE_drift} imply that
    \begin{align*}
        \sup_{x \in (0,L)} \P_x( \zeta_{aL} \ge \mathbf{c} L)
         & =
        \sup_{x \in (L-1,L)} \P_x( \zeta_{aL} \ge \mathbf{c} L) \\
         & \le
        \sup_{x \in (L-1,L)}
        \Big(
        \P_x\big( \inf_{s \in (0,aL)} \zeta_s \le 1 \big)
        +
        \P_x\big( L + B_{aL} - \sqrt{2\lambda} aL \ge
            \mathbf{c} L \big)
        \Big).
    \end{align*}
    Moreover
    \begin{equation*}
        \P( L + B_{aL} - \sqrt{2\lambda} aL \ge \mathbf{c} L)
        =
        \P\big(
        B_{aL}
        \ge
        L \big( \mathbf{c} - 1 + \sqrt{2\lambda} a \big)
        \big).
    \end{equation*}
    Since $\sqrt{2\lambda} \to \beta$ as $N \to \infty$ by
    Lemma~\ref{lem:spectral}~(i), $\mathbf{c}$ can be chosen large enough
    so that $\liminf_{N \to \infty} \sqrt{2\lambda}a > 1 - \mathbf{c}$.
    In this case, there are some constants $C > 0$ and $\Delta > 0$ such
    that
    \[
        \P( L + B_{aL} - \sqrt{2\lambda} aL \ge \mathbf{c} L)
        \le e^{- CL} \le N^{-\Delta}. \qedhere
    \]
\end{proof}

In the next lemma we show that, when the spine process starts close to the right
boundary, it is rapidly driven back to the left and spends only a time
of order $1$ in a neighbourhood of $L$. In the sequel, we write
\begin{equation} \label{def:spine_fromL}
    \tilde{\zeta}_t = L-\zeta_t,
\end{equation}
for the spine process seen from $L$.

\begin{lemma} \label{lem:integrabilityReversedSpine}
    Fix some $a < 1/\beta$. For all $A\geq 1$,
    \[
        \forall b > 0,\quad
        \forall R > 0,\quad
        \adjustlimits \sup_{N \ge 1} \sup_{z \in (0, R)}
        \int_0^{aL} \E_{L-z}[ e^{- b \tilde{\zeta}_s} ] \diff s < \infty.
    \]
\end{lemma}
\begin{proof}
    First, note that
    $-F(\sqrt{2\lambda}, z) \ge \sqrt{2 \lambda}$. Thus, on the event
    that $\tilde{\zeta}_s$ does not reach $L-1$,
    \[
        d \tilde{\zeta}_s \ge \sqrt{2\lambda} \diff s + dB_s.
    \]
    Therefore,
    \begin{align*}
        \int_0^{aL} \E_{L-z}\big[ e^{-b \tilde{\zeta}_s} \big]\diff s
         & \le
        aL \P_{L-z}\big(
        \sup_{s \in (0,aL)} \tilde{\zeta}_s \ge L-1
        \big)
        + \int_0^\infty \E\big[ e^{-b(\sqrt{2\lambda}s + B_s)} \big]
        \diff s \\
         & =
        aL \P_{L-z}\big(
        \sup_{s \in (0,aL)} \tilde{\zeta}_s \ge L-1
        \big)
        + \int_0^\infty e^{-b \sqrt{2\lambda}s + b^2 s / 2}\diff s.
    \end{align*}
    The integral is finite if $b$ is small enough, and the first term
    vanishes uniformly on compacts as $N \to \infty$ by
    Lemma~\ref{lem:boundSpineProcess}. Since $\tilde{\zeta} \geq 0$, the
    result holds for every $b > 0$.
\end{proof}
\subsection{Integrable moment bounds}

In this last section, we derive rough bounds  that will serve as
integrable bounds when taking the large $N$ limit in the recursion
formula~\eqref{eq:branching:rn0-1}. These bounds are obtained inductively
using Proposition~\ref{def:planar-moments} and the next preparatory
lemma.

\begin{lemma} \label{lem:m2}
    Let $T > 0$ and $\mathbf{c} \in (0,1]$. Then, for any $A \geq 1$, there
    exists a constant $C > 0$ such that
    \begin{equation*}
        \forall t\in[{L^2},TN], \quad
        \sup_{x \in (0, {\bf c} L)} \int_0^{t}\E_{x}(h(\zeta_{u}))\diff u
        \leq C N^{\gamma-1} \big(t \vee N^{\bf c}\big).
    \end{equation*}
    In particular, for $\mathbf{c}=1$ and $t=TN$,
    \begin{equation*}
        \sup_{x \in (0, L)}
        \int_0^{TN} \E_{x}(h(\zeta_{u})) \diff u \leq C N^\gamma.
    \end{equation*}
\end{lemma}

\begin{proof}
    Noting that for every $t>0$, $\mathbf{1}_{s\in[0,t]}\leq e^{\frac{t-s}{t}}$ we see that
    \begin{equation*}
        \int_0^t q_s(x, y)\diff s \leqslant e G_{\frac{1}{t}}(x, y).
    \end{equation*}
    It then follows from Fubini's theorem that
    \begin{equation*}
        \int_0^{t}\E_{x}(h(\zeta_{s})) \diff s = \int_0^{t}\int_0^{L}
        q_s(x,y)h(y)\diff y \ \diff s \leq
        e \int_0^{L} h(y) G_\xi (x,y)\diff y, \quad \text{with} \quad \xi=\frac{1}{t}.
    \end{equation*}
    By the expression for the Green's function in \eqref{def:green},
    \begin{multline*}
        v(x)\int_0^{L} h(y) G_\xi (x,y)\diff y
        \\= (\omega_{\lambda_\infty+\xi})^{-1} d_{\lambda_\infty+\xi}(x)
        \int_0^x h(y) v(y) \vl(y) \diff y
        + (\omega_{\lambda_\infty+\xi})^{-1} \vl(x) \int_x^L h(y) v(y)
        d_{\lambda_\infty+\xi}(y) \diff y.
    \end{multline*}
    We know from Lemma~\ref{lem:spectral} and
    Lemma~\ref{lem:boundGreen} that
    \begin{align*}
        J_1 & \coloneqq (\omega_{\lambda_\infty+\xi})^{-1} d_{\lambda_\infty+\xi}(x)
        \int_0^x h(y) v(y) \vl(y) \diff y                                                                         \\
            & =(1\wedge  x\wedge(L-x))e^{-\beta x}\left(O(t)e^{(\mu-3\beta ) x}+O(1)e^{(\mu-\beta) x}\right) \\
            & =v(x)\left(O(t)e^{(\mu-3\beta ) x}+O(1)e^{(\mu-\beta) x}\right).
    \end{align*}
    Similarly, we get that
    \begin{align*}
        J_2 & = (\omega_{\lambda_\infty+\xi})^{-1}\vl( x)\int_ x^{L}
        h(y)v(y)d_{\lambda_\infty+\xi}(y)\diff y =e^{(\mu-3\beta)L}v(x)(O(t)+O(1)e^{2\beta x}).
    \end{align*}
    Combining this with \eqref{def:LA0}, we get that, for all
   $x \in (0, \mathbf{c} L)$,
    \begin{equation*}
        J_1 \leq C v(x) (tN^{\gamma-1}+N^{\bf{c}\gamma})
        = C v(x) N^{\gamma-1} (t + N^{1-\gamma+\bf{c}\gamma}), \quad
        J_2 \leq Cv(x)N^{\gamma-1}(t+N^{\bf{c}}),
    \end{equation*}
    where the constants ${C}$ may depend on $A$ and $T$ but not on $N$.
    Then, note that
    \[
        1 - \gamma + \mathbf{c}\gamma = \mathbf{c} +
        (1-\gamma)(1-\mathbf{c}) < \mathbf{c}
    \]
    since $\gamma > 1$ and $\mathbf{c} < 1$. Thus
    \begin{equation*}
        J_1+J_2\leq \tilde{C}v(x)N^{\gamma-1}(t\vee N^\c).
    \end{equation*}
\end{proof}

\begin{lemma}\label{lem:momentBound}
    For any $T > 0$, $A \ge 1$, and $k \ge 1$ there exists a constant $C >
        0$ such that, for $N$ sufficiently large,
    \[
        \sup_{x \in (0, L)} \sup_{t \in [0, TN]} \mathbf{M}^{k,t}_x[1]
        \le C N^{\gamma (k-1)}
        \quad\text{or equivalently}\quad
        \sup_{x \in (0, L)} \sup_{t \in [0, T]} \hat{\mathbf{M}}^{k,t}_x[1]
        < \infty.
    \]
\end{lemma}

\begin{proof}
    The result is derived by induction. For $k=1$, we use
    that $\M^{1,t}_x[1] = e^{-wt}$ by Proposition~\ref{def:planar-moments}~(i), 
    and that $wN \to c_{w}/A^{\alpha-1}$ as $N \to \infty$ by
    Lemma~\ref{lem:spectral}~(ii). Let us now take $k\geq 2$.
    Suppose that the result holds for $i < k$, and let $(C_i(A,T))$ be
    the corresponding constants. Proposition~\ref{def:planar-moments}~(ii) 
    implies that
    \begin{equation}
        \label{eq:bound_ind}
        \M^{k,t}_x[1]=
        \int_0^t e^{-w (t-s)}
        \E_x\left[ r(\zeta_{t-s}) h(\zeta_{t-s})
            \sum_{n=1}^{k-1}\M_{\zeta_{t-s}}^{n,s}[1]
            \M_{\zeta_{t-s}}^{k-n,s}[1]
            \right]
       \diff s,
    \end{equation}
    where the right-hand side is obtained by summing over all the
    $c = (n, k-n)$-tree decompositions.
    This yields
    \begin{align*}
        \mathbf{M}_x^{k,t}[1] & \leq ||r||_\infty \sum_{n=1}^{k-1} C_k(A,T) C_{n-k}(A,T)
        N^{\gamma(k-2)} \int_0^{NT}\E_{x}(h(\zeta_u)) \diff u
    \end{align*}
    and the result follows from Lemma~\ref{lem:m2} with ${\bf c}=1$.
\end{proof}

We will need extensively the following immediate consequence of the
moment bounds above, which we state as a separate corollary.

\begin{cor} \label{cor:usefulBound}
    For any $A \ge 1$, any $k > n \ge 1$, there exists $C > 0$ such that,
    for $N$ large enough,
    \[
        \E_\Pi\big[
        r(\bar{\zeta}_s) h(\bar{\zeta}_s)
        \hat{\M}^{n,t-s}_{\bar{\zeta}_s}[1]
        \hat{\M}^{k-n,t-s}_{\bar{\zeta}_s}[1]
        \big]
        \le
        C N^{\gamma-1},
    \]
    or, equivalently,
    \[
        \E_\Pi\big[
        r(\bar{\zeta}_s) h(\bar{\zeta}_s)
        \M^{n,t-s}_{\bar{\zeta}_s}[1]
        \M^{k-n,t-s}_{\bar{\zeta}_s}[1]
        \big]
        \le
        C N^{\gamma(k-\alpha)}.
    \]
\end{cor}

\begin{proof}
    This is a straightforward combination of Lemma~\ref{lem:momentBound}
    and of the bound \eqref{eq:boundIntegralPi} in Corollary~\ref{lem:eigenBound}.
\end{proof}

Let us end this section with a technical result that we will need in the
proof of Theorem~\ref{th:cv_mmm}. Recall that $X_i \colon \mathbb{U}^*_k
    \to \R_+$ denotes the mark (position) of the $i$-th leaf of a marked
ultrametric matrix.

\begin{lemma} \label{lem:57}
    Let $k\in\mathbb{N}$ and $a < 1/\beta$. There exists
    $\mathbf{c} \in (0, 1)$ and $\Delta > 0$ such that, for any $A \geq
        1$, there is a constant $C > 0$ s.t.
    \begin{equation*}
        N^{-\gamma(k-1)}
        \sup_{x \in (0,L)} \mathbf{M}_x^{k,aL}\left[
        \mathbf{1}_{\exists i\in[k]: X_i\geq \mathbf{c}L}
        \right]
        \le CN^{-\Delta}.
    \end{equation*}
\end{lemma}

\begin{proof}
    We know from Proposition~\ref{def:planar-moments}~(i) and
    Lemma~\ref{lem:boundSpineProcess2} that, for all sufficiently large
    $N$, all $A \geq 1$ and all $y \in (0,L)$,
    \begin{align*}
        \mathbf{M}_y^{1,aL}\left[
        \mathbf{1}_{{X_1} \geq \mathbf{c}L}
        \right]
         & = e^{-waL}\mathbb{P}_y\left(\zeta_{aL} \geq
        \mathbf{c}L \right)
        \leq N^{-\Delta}.
    \end{align*}
    A union bound then yields that
    \begin{equation*}
        \mathbf{M}_x^{k,aL}\left[
        \mathbf{1}_{\exists i\in[K]: {X_i}\geq \mathbf{c}L}
        \right]
        \leq \sum_{i=1}^{k} \mathbf{M}_x^{k,aL}\left[
        \mathbf{1}_{ {X_i}\geq \mathbf{c}L}
        \right]
        = k\mathbf{M}_x^{k,aL}\left[
        \mathbf{1}_{ {X_1}\geq \mathbf{c}L}
        \right].
    \end{equation*}
    Hence, it is sufficient to prove that for all  $k\in \N$ and all $x\in(0,L)$,
    \begin{equation}\label{eq:ubound}
        \mathbf{M}_x^{k,aL}\left[
        \mathbf{1}_{ {X_1}\geq \mathbf{c} L}
        \right]
        \leq N^{\gamma(k-1)} (aL)^{k-1}
        \mathbf{M}_x^{1,aL}\left[
        \mathbf{1}_{{X_1}\geq \mathbf{c}L }
        \right].
    \end{equation}
    We prove this bound by induction. For $k=1$, \eqref{eq:ubound} is an
    equality. Let $k\geq 2$. Assume that \eqref{eq:ubound} holds up until
    rank $k$. It then follows from Proposition~\ref{def:planar-moments}
    and our moment bound from Lemma~\ref{lem:momentBound} that
    \begin{align*}
        \mathbf{M}_x^{k,aL}\left[\mathbf{1}_{ {X_1}\geq \mathbf{c}L}\right]
         & =\int_{0}^{aL} e^{-w(aL-u)}
        \E_{x}\left[ r(\zeta_{aL-u}) h(\zeta_{aL-u})
        \sum_{n=1}^{k-1} \M_{\zeta_{aL-u}}^{n,u}\left[
        \mathbf{1}_{ {X_1}\geq \mathbf{c}L}
        \right]
        \M^{k-n,u}_{\zeta_{aL-u}}[1]
        \right]\diff u,                         \\
         & \leq N^{\gamma (k-2)} (aL)^{k-2}
        \mathbf{M}_x^{1,aL}\left[
        \mathbf{1}_{{X_1}\geq \mathbf{c}L}
        \right]
        \int_0^{aL} e^{-w(aL-u)}
        \E_{x}[r(\zeta_{aL-u}) h(\zeta_{aL-u})] \diff u.
    \end{align*}
    The claim~\eqref{eq:ubound} then follows from~\eqref{eq:easy_bound1}.
    We conclude the proof by choosing $\Delta'$ such that $(aL)^{k-1}
        N^{-\Delta}\leq N^{-\Delta'}$.
\end{proof}

\section{The BBM seen from the tip}
\label{sect:BB
    M-from-tip}

We now consider the reversed process introduced in
Section~\ref{sec:step2} (Step 2), that is, a branching Brownian motion
with branching rate $\frac{1}{2}$, positive drift $+\mu$ and killed at
$0$. As outlined in Step~2 and Step~3 of the heuristics, this reversed
process provides a good approximation of the BBM on $(0,L)$, when started
close to the cutoff $L$ and over a short time interval ($\vep\propto L$).
This approximation, in turn, allows us to characterise the jump measure
of the limiting $\alpha$-stable genealogy.

Let us recall some notation that we will need, which were mostly introduced in
Section~\ref{sec:step2} (Step~2 of the heuristics). By analogy with the
BBM, we write $\la{\mathcal{N}}_t$ for the set of particles at time $t$
in the reversed BBM and $\la{X_v}(t)$ for the location at that time of $v
    \in \la{\mathcal{N}}_t$. We write $\la{\P}_z$ for a probability measure
under which the process starts from a single particle at $z > 0$, and let
$\la{\E}_z$ be the corresponding expectation. In
\eqref{eq:notation_reversed}, we introduced the following functions
\begin{equation*}
    \la{v}(z) = 2 e^{\beta} \sinh(\beta z),\quad
    \la{h}(z) = c_h \la{v}(z) e^{-\mu z},\quad
    \la{\widetilde{h}}(z) = c_{\tilde{h}} \la{v}(z) e^{\mu z},\quad
    \la{\Pi}(z) = \la{\widetilde{h}}(z) \la{h}(z).
\end{equation*}
It is straightforward to check that $\la{h}$ is harmonic for the
reversed BBM and we introduced the corresponding additive martingale in
\eqref{eq:reversedMartingale} as
\begin{equation*}
    \forall t \ge 0,\quad \la{W_t} =
    \sum_{v \in \la{\mathcal{N}}_t} \la{h}\left(\la{X_v}(t)\right).
\end{equation*}
This martingale is positive and thus converges a.s.\ to a limit
$\la{W}_\infty$. Mimicking the notation for the moments of the BBM from
\eqref{eq:biasedMomentBBM}, we defined the $k$-th moment of the reversed
BBM in \eqref{eq:biasedReversedMoment} as
\[
    \la{\M}^{k,t}_z[1] \coloneqq
    \frac{1}{\la{h}(z)}
    \la{\E}_z\Bigg[
        \sum_{v_1 < \dots < v_k \in \la{\mathcal{N}}_t}
        \prod_{i=1}^k \la{h}(\la{X_v}(t))
        \Bigg].
\]
The aim of this section is to prove the following result.

\begin{proposition} \label{prop:limitReversed}
    For any $z > 0$ and $k \ge 1$,
    \[
        \lim_{t \to \infty} \la{\M}^{k,t}_z[1]
        \eqqcolon \la{\M}^{k,\infty}_z[1]
        =
        \frac{\la{\E}_z[\la{W}^k_\infty]}{k!\la{h}(z)} < \infty.
    \]
    Moreover, there exists a measure $\Lambda$ on $(0,\infty)$ with
    a finite exponential moment in the sense of
    \eqref{eq:finiteExpMoment} and such that
    \begin{equation*}
        \forall k \ge 2,\quad
        \int_0^\infty z^k \Lambda(\diff z)
        =
        \frac{k!}{2} \int_0^\infty \la{\Pi}(z) \la{h}(z) \sum_{p=1}^{k-1}
        \la{\M}^{p,\infty}_z[1] \la{\M}^{k-p,\infty}_z[1] \diff z.
    \end{equation*}
\end{proposition}

\subsection{Integrable moment bounds}

As for the BBM on $(0,L)$, our proof relies on a recursive formula for
the moments of the reversed BBM. To state this formula, let us introduce
the reversed spine process $(\la{\zeta_t})_{t \ge 0}$ as the solution to
the SDE
\begin{equation} \label{eq:reversed_SDE}
    \diff \la{\zeta_t} = -F(\beta, \la{\zeta_t})\diff t + \diff B_t, \quad t > 0,
\end{equation}
where $F$ is defined in \eqref{eq:SDE_drift}. Recall that $\la{h}$ is
harmonic for the reversed BBM. It is straightforward to check that
the reversed spine process corresponds to Doob's $\la{h}$-transform of
the reversed BBM. The moments of the reversed BBM can then be computed as
follows using Proposition~\ref{def:planar-moments}.

\begin{proposition}[Recursive equations for the moments of the reversed process]
    For any $z > 0$,
    \begin{equation} \label{eq:manyToFewReversed}
        \la{\M}^{1,t}_z[1] = 1;\quad
        \forall k\geq 2, \quad \la{\M}^{k,t}_z[1]
        = \frac{1}{2} \int_0^t \la{\E}_z\left[
        \la{h}(\la{\zeta_s})
        \sum_{p=1}^{k-1} \la{\M}^{p,t-s}_{\la{\zeta_s}}[1]
        \la{\M}^{k-p,t-s}_{\la{\zeta_s}}[1]\right]\diff s.
    \end{equation}
\end{proposition}

The reversed process turns out to be much easier to study than the BBM on
$(0,L)$. In particular, the reversed spine process $(\la{\zeta_t})_{t \ge
            0}$ is transient, and we do not need to consider an exponential tilting
of its density to define its Green's function. The next result
follows from the explicit expression for the Green's function of a
one-dimensional diffusion, see for instance
\cite[Theorem~3.19]{etheridge2011some}, and from a simple computation.

\begin{lemma}[Reversed Green's function] \label{lem:green:reversed}
    For any $z, y > 0$, define the Green's function of $(\la{\zeta_t})_{t \ge 0}$ as
    \[
        \la{G}(z,y) \coloneqq
        2 \la{v}(y)^2 \int_{z \vee y}^\infty \la{v}(z')^{-2} \diff z'.
    \]
    Then for any measurable $f \colon [0,\infty) \to [0,\infty)$
    and any $z > 0$,
    \begin{equation} \label{eq:occupationMeasure}
        \la{\mathbb{E}}_z\left[
            \int_0^\infty f(\la{\zeta_t}) \diff t
            \right]
        = \int_0^\infty \la{G}(z,y) f(y) \diff y.
    \end{equation}
\end{lemma}

Those two results readily lead to an integrable bound on the moments of
the reversed BBM.

\begin{lemma} \label{lem:rough:b:reversed}
    For any $k \in \mathbb{N}$
    \begin{equation*}
        \sup_{z > 0} \la{h}(z) < \infty,
        \quad \sup_{z,y > 0} \la{G}(z,y) < \infty,
        \quad \sup_{z > 0} \int_0^\infty
        \la{\E}_z[\la{h}(\la{\zeta_s})] \diff s < \infty,
        \quad \text{and} \quad
        \sup_{t,z > 0} \la{\M}^{k,t}_z[1] < \infty,
    \end{equation*}
\end{lemma}

\begin{proof}
    The bound on $\la{h}$ and $\la{G}$ can be directly checked. For the
    third bound, by Lemma~\ref{lem:green:reversed},
    \begin{align*}
        \int_0^\infty \la{\E}_z[ \la{h}(\la{\zeta_s}) ]\diff s
         & = \int_0^\infty \la{h}(y) \la{G}(z,y) \diff y                \\
         & \le \sup_{z,y'} \la{G}(z,y') \int_0^\infty \la{h}(y) \diff y \\
         & < \infty.
    \end{align*}
    The last point follows by an induction. For $k=1$ there is nothing to
    prove. For any $k \ge 2$, we see from
    \eqref{eq:manyToFewReversed} that
    \begin{align*}
        \la{\M}^{k,t}_z[1]
         & = {\frac{1}{2}}  \int_0^t \la{\E}_z\Big[
        \la{h}(\la{\zeta_s}) \sum_{p=1}^{k-1}\la{\M}^{p,t-s}_{\la{\zeta_s}}[1]
        \la{\M}^{k-p,t-s}_{\la{\zeta_s}}[1]\Big] \diff s                                        \\
         & \le{\frac{1}{2}} \Bigg( \sum_{p=1}^{k-1} \Big(\sup_{s,y > 0} \la{\M}^{p,s}_y(1)\Big)
        \Big(\sup_{s,y > 0} \la{\M}^{k-p,s}_y[1]\Big) \Bigg)
        \int_0^t \E_z[ \la{h}(\la{\zeta_s}) ] \diff s,
    \end{align*}
    and the right-hand side is uniformly bounded by the second
    point.
\end{proof}

\subsection{Proof of Proposition~\ref{prop:limitReversed}}

We start with the first part of Proposition~\ref{prop:limitReversed} as a
separate lemma.

\begin{lemma} \label{lem:limitMomentReversed}
    For any $z > 0$ and $k \ge 1$,
    \[
        \lim_{t \to \infty} \la{\M}^{k,t}_z[1]
        = \frac{\la{\E}_z[\la{W}^k_\infty]}{k!\la{h}(z)} < \infty.
    \]
\end{lemma}

\begin{proof}
    We first prove that the limit of $\la{\M}^{k,t}_z[1]$ exists when
    $t\to\infty$. By the recursive formula for moments
    \eqref{eq:manyToFewReversed},
    \[
        \la{\M}^{k,t}_z[1]
        =
        \frac{1}{2} \int_0^t \la{\E}_z\left[ \la{h}(\la{\zeta_s})
        \sum_{p=1}^{k-1} \la{\M}^{p,t-s}_{\la{\zeta_s}}[1]
        \la{\M}^{k-p,t-s}_{\la{\zeta_s}}[1]\right]\diff s.
    \]
    A simple induction combined with the bounded convergence theorem
    (using the bound in Lemma~\ref{lem:rough:b:reversed}) shows that
    \begin{equation} \label{eq:recursionReversed}
        \forall z > 0, \quad
        \la{\M}^{k,\infty}_z[1] \coloneqq
        \lim_{t \to \infty} \la{\M}^{k,t}_z[1]
        =
        \frac{1}{2} \int_0^t \la{\E}_z\left[ \la{h}(\la{\zeta_s})
        \sum_{p=1}^{k-1} \la{\M}^{p,\infty}_{\la{\zeta_s}}[1]
        \la{\M}^{k-p,\infty}_{\la{\zeta_s}}[1]\right]\diff s.
    \end{equation}

    It now remains to show that $\la{\M}^{k,\infty}_z[1] = \la{\E}_z[\la{W}^k_\infty]
        / (k!\,\la{h}(z))$.
    Recall the definitions of $\la{W}_t$ and $\la{\M}^{k,t}_z[1]$
    in \eqref{eq:reversedMartingale} and \eqref{eq:biasedReversedMoment}.
    Expanding the product, $\la{W}_t^k$ can be expressed as sum over all
    $k$-tuples of individuals at time $t$, whereas $\la{\M}^{k,t}_z[1]$
    only involves a sum over ordered $k$-tuples of distinct individuals.
    By removing the ``diagonal terms'' from the sum and ordering the
    $k$-tuples, we obtain that
    \begin{equation} \label{decomp:W}
        \la{\E}_z[\la{W}_t^k] = k!\,\la{h}(z) \la{\M}_z^{k,t}[1]
        + \la{\E}_z\left[
            \sum_{v_1,\dots,v_k \in \la{\mathcal{N}_t}}
            \indic_{\{ \exists i \ne j : v_i = v_j\}}
            \prod_{p=1}^k \la{h}(\la{X}_{v_p}(t))
            \right].
    \end{equation}
    By a union bound,
    \begin{align}
        \la{\E}_z
        \left[
            \sum_{v_1,\dots,v_k \in \la{\mathcal{N}_t}}
            \indic_{\{ \exists i \ne j : v_i = v_j\}}
            \prod_{p=1}^k \la{h}(\la{X}_{v_p}(t))
            \right]
         & \le \binom{k}{2}
        \la{\E}_z\left[
            \sum_{v_1,\dots,v_k \in \la{\mathcal{N}_t}}
            \indic_{\{ v_{k-1} = v_k\}}
            \prod_{p=1}^k \la{h}(\la{X}_{v_p}(t))
        \right] \nonumber   \\
         & \le \binom{k}{2}
        \la{\E}_z\left[
        \left( \sup_{v\in\la{\mathcal{N}_t}}\la{h}(\la{X}_v(t)) \right)
        \cdot \la{W}_t^{k-1}
        \right]. \label{eq:boundWMoment}
    \end{align}
    Since $\la{h}(z)$ and $\la{\M}^{k,t}_z[1]$ are uniformly bounded by
    Lemma~\ref{lem:rough:b:reversed}, a simple induction shows that
    \[
        \forall k \ge 0,\quad
        \sup_{t, z > 0} \la{\E}_z\big[ \la{W}_t^k \big] < \infty.
    \]
    Therefore, $(\la{W}_t^k)_{t \ge 0}$ is uniformly integrable for all
    $k \ge 1$ and
    \[
        \forall z > 0,\quad
        \lim_{t \to \infty} \la{\E}_z[\la{W}_t^k] =
        \la{\E}_z[\la{W}^k_\infty \big].
    \]
    All what remains to show is that the second term in the right-hand
    side of \eqref{decomp:W} vanishes, which we prove by applying the
    bounded convergence theorem to \eqref{eq:boundWMoment}. First, note
    that by removing the killing boundary, the reversed BBM can be
    coupled to a standard BBM with drift $+\mu$ and branching rate
    $\frac{1}{2}$ in such a way that the latter has more particles than
    the former. Since it is well known that the left-most particle of a
    BBM with \emph{no} drift and branching rate $\frac{1}{2}$ moves at
    speed $-1$, adding a drift $\mu > 1$ ensures that this left-most
    particle goes to infinity. Therefore,
    \[
        \adjustlimits \lim_{t \to \infty} \inf_{v \in \la{\mathcal{N}}_t}
        \la{X}_v(t) = \infty,
        \qquad \text{a.s.}
    \]
    Recalling that $\la{h}(z) \to 0$ as $z \to \infty$, this shows that the
    integrand in \eqref{eq:boundWMoment} vanishes a.s.\ as $t \to
        \infty$, and since $(\la{W}^k_t)_{t \ge 0}$ is uniformly integrable
    the result follows.
\end{proof}

\begin{proof}[Proof of Proposition~\ref{prop:limitReversed}]
    The first part of the result has been proved in
    Lemma~\ref{lem:limitMomentReversed} and it only remains to show the
    second point. Let $k \ge 2$. By \eqref{eq:recursionReversed} and the
    definition of the Green's function in Lemma~\ref{lem:green:reversed},
    \begin{align} \label{eq:limitMomentReverse}
        \la{\M}^{k,\infty}_z[1]
         & = \frac{1}{2}\E_z\left[\int_0^\infty \la{h}(\la{\zeta_s})
            \sum_{p=1}^{k-1}\left( \la{\M}^{p,\infty}_{\la{\zeta_s}}[1]
            \la{\M}^{k-p,\infty}_{\la{\zeta_s}}[1] \right) \diff s
        \right] \nonumber                                            \\
         & = \frac{1}{2} \int_0^\infty \la{G}(z, y) \la{h}(y)
        \sum_{p=1}^{k-1} \left(
        \la{\M}^{p,\infty}_y[1]
        \la{\M}^{k-p,\infty}_y[1]
        \right) \diff y.
    \end{align}
    It will be convenient to note that
    \begin{equation} \label{eq:otherExpressionMeasure}
        \la{\Pi}(z) \la{G}(z, y) = 2 c_{h} c_{\tilde h} \la{v}(z)^2
        \la{v}(y)^2 \int_{y \vee z}^{\infty} \la{v}(z')^{-2} \diff z'
        = \la{\Pi}(z \wedge y) \la{G}(y \vee z, y\vee z).
    \end{equation}
    It can also directly be checked from the explicit expression of
    $\la{v}$ that
    \[
        \lim_{z \to \infty}
        \la{G}(z, z)
        = \lim_{z \to \infty}
        2 \int_z^\infty \Big(\frac{\la{v}(z)}{\la{v}(y)}\Big)^2 \diff y
        = \frac{1}{\beta}.
    \]
    Using that $\la{\Pi}$ is increasing, we deduce from these two
    expressions that
    \[
        \forall y > 0,
        \qquad
        \lim_{z \to \infty} \la{\Pi}(z) \la{G}(z,y) = \frac{\la{\Pi}(y)}{\beta},
        \qquad
        \la{h}(y) \la{\Pi}(z) \la{G}(z,y)
        \le \la{h}(y) \la{\Pi}(y) \sup_{z' > 0} \la{G}(z', z').
    \]
    Noting that $y \mapsto \la{h}(y) \la{\Pi}(y)$ is integrable since
    $\mu > 3\beta$ under \eqref{Hwp}, the dominated convergence theorem
    applied to \eqref{eq:limitMomentReverse} shows that
    \begin{equation} \label{eq:momentReversed1}
        \lim_{z \to \infty} k!\, \la{\Pi}(z) \la{\M}^{k,\infty}_z[1]
        = \frac{k!}{\beta} \int_0^\infty \la{\Pi}(y) \la{h}(y)
        \sum_{p=1}^{k-1} \la{\M}^{p,\infty}_y[1]
        \la{\M}^{k-p,\infty}_y[1] \diff y.
    \end{equation}
    Let $\mathscr{L}_z(\la{W}_\infty)$ denote the law of $\la{W}_\infty$
    under $\la{\P}_z$ and let $\Lambda_z = \beta/2 \la{\widetilde{h}}(z)
    \mathscr{L}_z(\la{W}_\infty)$. Using Lemma~\ref{lem:limitMomentReversed},
    equation~\eqref{eq:momentReversed1} above can be reformulated as
    \begin{multline*}
        \forall k \ge 2,\quad
        \int_{\R_+} y^k \Lambda_z(\diff y)
        = \frac{\beta}{2} \la{\widetilde{h}}(z) \la{\E}_z\big[ \la{W}_\infty^k \big]
        = \frac{\beta}{2}k!\, \la{\Pi}(z) \la{\M}^{k,\infty}_z[1] \\
        \longrightarrow
        \frac{k!}{2} \int_0^\infty \la{\Pi}(y) \la{h}(y)
        \sum_{p=1}^{k-1} \la{\M}^{p,\infty}_y[1]
        \la{\M}^{k-p,\infty}_y[1] \diff y,
        \qquad \text{as $z \to \infty$}.
    \end{multline*}

    We will end the proof by using an adaptation of the method of moments to
    vague convergence, see Lemma~\ref{lem:vagueMethodMoments}. This
    requires us to check a growth condition on the limiting moments which
    ensures that the limit is characterised by its moments. Define a
    sequence $(a_k)_{k \ge 1}$ as
    \[
        a_1 = \la{\M}^{1,\infty}_z[1] = 1,
        \qquad
        \forall k \ge 1,\quad a_k = \sup_{z > 0}
        \big( (1 \vee \la{\Pi}(z)) \la{\M}_z^{k,\infty}[1] \big).
    \]
    Clearly, $\la{\M}^{k,\infty}_z[1] \le a_k$ so that
    \begin{align*}
        \forall k \ge 2, \; \forall z > 0, \quad
        \la{\Pi}(z) \la{\M}^{k,\infty}_z[1]
         & = \frac{1}{2} \int_0^\infty \la{\Pi}(z) \la{G}(z, y) \la{h}(y)
        \sum_{p=1}^{k-1} \left( \la{\M}^{p,\infty}_y[1]
        \la{\M}^{k-p,\infty}_y[1] \right) \diff y                              \\
         & \le \frac{1}{2}
        \int_0^\infty \la{\Pi}(z) \la{G}(z,y) \la{h}(y) d y
        \cdot
        \sum_{p=1}^{k-1} a_p a_{k-p}                                      \\
         & \le \frac{1}{2} \Big( \sup_{z' > 0} \la{G}(z',z') \Big) \cdot
        \int_0^\infty \la{\Pi}(y) \la{h}(y) \diff y
        \cdot
        \sum_{p=1}^{k-1} a_p a_{k-p},
    \end{align*}
    where we have used \eqref{eq:otherExpressionMeasure} in the last
    line. Since $y \mapsto \la{\Pi}(y) \la{h}(y)$ is integrable, there
    exists $C > 0$ such that $a_k \le C \sum_{p=1}^{k-1} a_p a_{k-p}$.
    Using Lemma~\ref{cor:boundShapes}, this ensures that there exists $R
        > 0$ such that $a_k \le R^k$, $k \ge 1$. Therefore, up to choosing a
    larger $R$, this leads to the bound
    \[
        \lim_{z \to \infty} \int_{\R_+} y^k \Lambda_z(\diff y) = 
        \frac{\beta}{2} k!\,
        \lim_{z \to \infty}\la{\Pi}(z) \la{\M}^{k,\infty}_z[1]
        \le \frac{\beta k!\, a_k}{2}
        \le k!\, (R')^k
    \]
    up to choosing $R'$ large enough. This checks the condition of
    Lemma~\ref{lem:vagueMethodMoments}, and we conclude that there exists
    a limit measure $\Lambda$ such that $\Lambda_z \to \Lambda$ vaguely
    as $z \to \infty$ and whose $k$-th moment is given by \eqref{eq:momentReversed1}. 
\end{proof}

\section{Proof of Theorem~\ref{th:cv_mmm}}
\label{sec:recursion}

We now prove Theorem~\ref{th:cv_mmm}. As mentioned above, the bulk of the
work consists in computing the limit of the moments of the BBM, that is,
showing that \eqref{eq:convergenceMoment2} holds, which corresponds to the
first condition of Theorem~\ref{thm:main-cv}. The proof implements the
heuristic arguments which were outlined in Section~\ref{sec:heuristics_BBM}
and follows closely the steps described in that section. Recall the
definition of $\epsilon$ in \eqref{eq:defEps} and $L$ in \eqref{def:LA0}, which
will be used frequently in this section. Moreover, set
\begin{equation}
    \label{def:renorm_times}
    \hep \coloneqq \frac{\vep}{N}, \quad \text{and} \quad
    \t \coloneqq \frac{L^2}{N}.
\end{equation}

\subsection{Step 1: fast mixing}

In this section, we carry out Step~1 of the heuristics by establishing
\eqref{eq:approxRecursion} in Section~\ref{sec:step1}. Recall the
definition of the map $S^\epsilon$ in \eqref{eq:mergedMap}.

\begin{proposition} \label{lem:approx:1}
    Let $k \geq 2$ and  let $c$ be a composition of $k$ with $\abs{c} = p
        < k$. Let $G$ be a functional of the form \eqref{eq:productFunctional},
    that is
    \begin{equation} \label{eq:Feps1}
        \forall U^* \in \mathbb{U}^*_k,\quad
        G(U^*) = \indic_{\{ c(U) = c\}} f(\tau(U)) 
        \prod_{i=1}^{\abs{c}} F_i(U^*_i),
    \end{equation}
    Let $t>0$, $A\geq 1$ and  $\c\in(0,1)$. Uniformly in
    $x \in [0, \c L]$,
    \begin{multline*}
        \lim_{N\to\infty}
        \hat{\mathbf{M}}_{x}^{k,t}[G \circ S^{\hat{\epsilon}}]
        -
        \frac{\hat{m}_{p}}{p!} \int_{0}^{t} f(t-u) e^{-wNu}
        \prod_{i=1}^p \hat{\mathbf{M}}_{\Pi}^{c_{i},t-u }[F_{i}] \diff u =0 \\
        \text{where } \hat{m}_{p} \equiv \hat m_{p,(N,A)} =
        \frac{p!}{N^{\gamma(p-\alpha)}}
        \int_0^L r(y) h(y)
        \sum_{i=1}^{p-1}
        \mathbf{M}_{y}^{i,\eps}[1]
        \mathbf{M}_{y}^{p-i,\eps}[1]
        \Pi(y) \diff y.
    \end{multline*}
\end{proposition}

The proof of this result relies on a mixing argument by showing that the
initial condition can be forgotten, which makes
\eqref{eq:heuristicInitialCondition} from the heuristics rigorous. We
will need the following stronger formulation of this result.

\begin{lemma}[Forgetting the initial condition] \label{prop:spineEquilibriate}
    For any $T > 0$, $A \ge 1$, $k \ge 1$, and $\mathbf{c} \in (0,1)$
    \[
        \adjustlimits \sup_{x \in [0, \mathbf{c}L]} \sup_{t, t' \in [\t, T]}
        \abs*{\hat{\M}^{k,t}_x[F] - \hat{\M}^{k,t'}_\Pi[F]}
        \le C \norm{F}_\infty (o_{N}(1) + \abs{t-t'}),
    \]
    for all continuous bounded $F \colon \mathbb{U}^*_k \to \R$.
\end{lemma}

\begin{proof}
    Let $G$ be a functional of the product form as in
    \eqref{eq:Feps1} with $c = (n, k-n)$, $n < k$. We split the moment
    depending on whether the first branch point occurs before or after
    some fixed time $s \le t$. By the recursive formula for moments of
    Proposition~\ref{def:planar-moments} (after rescaling) and the Markov
    property of the spine process,     
    \begin{align}
        \hat{\M}^{k,t}_x\big[ G(U^*) \indic_{\{ \tau(U) \le t-s \}} \big]
         & = \frac{1}{N^{\gamma-1}} \int_s^t f(t-u) e^{-wNu}
        \E_x\Big[
        r(\bar{\zeta}_u) h(\bar{\zeta}_u)
        \hat{\M}^{n, t-u}_{\bar{\zeta}_u}[F_1]
        \hat{\M}^{k-n, t-u}_{\bar{\zeta}_u}[F_2]
        \Big] \diff u \nonumber                                                               \\
         & = \begin{multlined}[t]
                 \frac{1}{N^{\gamma-1}} \int_0^{t-s} f(t-s-v) e^{-wN(v+s)}
                 \E_x\Big[
                 r(\bar{\zeta}_{v+s}) h(\bar{\zeta}_{v+s}) \\
                 \times
                 \hat{\M}^{n, t-s-v}_{\bar{\zeta}_{v+s}}[F_1]
                 \hat{\M}^{k-n, t-s-u}_{\bar{\zeta}_{v+s}}[F_2]
                 \Big] \diff v
             \end{multlined} \nonumber \\
         & = e^{-wNs} \E_x\big[ \hat{\M}^{k,t-s}_{\bar{\zeta}_s}[G] \big].
        \label{eq:spineEquilibriate1}
    \end{align}
    By a monotone class argument, \eqref{eq:spineEquilibriate1}
    still holds when $G$ is replaced by any continuous bounded $F \colon
        \mathbb{U}^*_k \to \R$. For the rest of the proof, we fix such a
    functional $F$.

    Now, recalling that $\hat{\M}^{k,t-s}_y[1]$ is uniformly bounded by
    Lemma~\ref{lem:momentBound}, there exists $C > 0$ such that    \begin{align}
        \abs*{\hat{\M}^{k,t}_x\big[ F(U^*) \indic_{\{ \tau(U) > t-s \}} \big]}
         & \leq \norm{F}_\infty
        \hat{\M}^{k,t}_x\big[  \indic_{\{ \tau(U) > t-s \}} \big]
        \nonumber                                                   \\
         & \le \sum_{n=1}^{k-1}\frac{\norm{F}_\infty}{N^{\gamma-1}}
        \int_0^s e^{- wNu}
        \E_x\Big[
        r(\bar{\zeta}_u) h(\bar{\zeta}_u)
        \hat{\M}^{n, t-u}_{\bar{\zeta}_u}[1]
        \hat{\M}^{k-n, t-u}_{\bar{\zeta}_u}[1]
        \Big] \diff u \nonumber                                          \\
         & \le \frac{C\norm{F}_\infty}{N^{\gamma-1}}
        \int_0^s \E_x\big[ r(\bar{\zeta}_u) h(\bar{\zeta}_u) \big] \diff u,
        \label{eq:spineEquilibriate2}
    \end{align}
    and where we used Proposition~\ref{def:planar-moments} for the
    second inequality. Integrating the two previous inequalities
    with respect to the stationary measure of the spine process $\Pi$, we
    first obtain that
    \begin{equation*}
        \abs*{\hat{\M}^{k,t}_\Pi[F] - \hat{\M}^{k,t-s}_\Pi[F]}
        \le
        \norm{F}_{\infty} \hat{\M}^{k,t-s}_\Pi[1] \big( 1 - e^{-wNs} \big)
        + \frac{C \norm{F}_\infty}{N^{\gamma-1}}
        \cdot \int_0^s \E_{\Pi}[r(\bar{\zeta_u})h(\bar{\zeta}_u)]\diff s.
    \end{equation*}
    We now bound each term: by Corollary~\ref{lem:eigenBound},
    $\E_\Pi[h(\bar{\zeta}_u)] \le C' N^{\gamma-1}$ for some $C' > 0$;
    by Lemma~\ref{lem:spectral}~(ii), $wN \to c_{w} /A^{1-\alpha}$ as
    $N \to \infty$; and by Lemma~\ref{lem:momentBound},
    $\hat{\M}_\Pi^{k,t-s}[1]$ is uniformly bounded. Altogether, we can
    find $C > 0$ such that, for $N$ large enough, 
    \begin{equation} \label{eq:spineEquilibriate3}
        \sup_{0 \le t,t' \le T}
        \abs*{\hat{\M}^{k,t}_\Pi[F] - \hat{\M}^{k,t'}_\Pi[F]}
        \le
        C \norm{F}_\infty \abs{t-t'}.
    \end{equation}
    Using again \eqref{eq:spineEquilibriate1} and
    \eqref{eq:spineEquilibriate2} but at time $s = \t$, we get that
    \begin{multline*}
        \abs{\hat{\M}^{k,t}_x[F] - \hat{\M}^{k,t-\t}_\Pi[F]}
        \le
        e^{-wN \t}
        \abs*{\E_x\big[ \hat{\M}^{k,t-\t}_{\bar{\zeta}_{\t}}[F] \big]
        - \hat{\M}^{k,t-\t}_\Pi[F]} \\
        + \norm{F}_\infty \big(1 - e^{-w N \t}\big) \hat{\M}^{k,t-\t}_\Pi[1]
        + \frac{C \norm{F}_\infty}{N^{\gamma-1}} \int_0^{\t}
        \E_x\big[ r(\bar{\zeta}_u) h(\bar{\zeta}_u) \big] \diff u.
    \end{multline*}
    By Lemma~\ref{lem:m2} (after time rescaling by $1/N$), for any
    $\mathbf{c} \in (0, 1)$,
    \[
        \sup_{x \in (0, \mathbf{c}L)}
        \frac{1}{N^{\gamma-1}}
        \int_0^{\t} \E_x\big[ r(\bar{\zeta}_u) h(\bar{\zeta}_u) \big] \diff u
        \le C N^{-(1-\mathbf{c})}
        \to 0,\quad \text{as $N \to \infty$.}
    \]
    Moreover,
    \[
        \abs*{\E_x\big[ \hat{\M}^{k,t-\t}_{\bar{\zeta}_{\t}}[F] \big]
        - \hat{\M}^{k,t-\t}_\Pi[F]}
        \le
        \norm{F}_\infty \int_0^L \abs*{q_{N \t}(x,y) - \Pi(y)}
        \hat{\M}^{k,t-\t}_y[1] \diff y,
    \]
    and the right-hand side vanishes uniformly in $x \in (0, L)$
    and $t \in (\t, T)$ by  Proposition~\ref{prop:hk}.
    Finally, by Lemma~\ref{lem:spectral}~(ii),
    $wN\t \to 0$ as $N \to \infty$. Those three bounds
    together show that
    \begin{equation} \label{eq:spineEquilibriate4}
        \abs{\hat{\M}^{k,t}_x[F] - \hat{\M}^{k,t-\t}_\Pi[F]}
        = \norm{F}_\infty o_N(1),
    \end{equation}
    where $o_N(1) \to 0$ uniformly in $x \in (0, \mathbf{c}L)$ and $t \ge
        \t$ as $N \to \infty$. The result
    follows by writing, for $x \in (0, \mathbf{c}L)$ and $t,t' \in (\t,
        T)$,
    \begin{align*}
        \abs{\hat{\M}^{k,t}_x[F] - \hat{\M}^{k,t'}_\Pi[F]}
         & \le
        \abs{\hat{\M}^{k,t}_x[F] - \hat{\M}^{k,t-\t}_\Pi[F]}
        +
        \abs{\hat{\M}^{k,t-\t}_\Pi[F] - \hat{\M}^{k,t'}_\Pi[F]} \\
         & \le
        C \norm{F}_\infty o_N(1) + C \norm{F}_\infty \abs{t-t'}
    \end{align*}
    where we used \eqref{eq:spineEquilibriate3} and
    \eqref{eq:spineEquilibriate4}.
\end{proof}

\begin{proof}[Proof of Proposition \ref{lem:approx:1}]
    By Lemma~\ref{prop:spineEquilibriate},
    $\abs{\hat{\M}^{k,t}_x[G \circ S^{\hat{\epsilon}}]
        - \hat{\M}^{k,t}_{\Pi}[G \circ S^{\hat{\epsilon}}]} \to 0$ as $N \to
        \infty$ uniformly in $x \in [0, \mathbf{c}L]$. It is therefore
    sufficient to prove the result starting from the equilibrium
    distribution $\Pi$ rather than from a fixed starting point $x \in [0,
            \c L]$.

    We need to treat the case that $\tau(U) < \hat{\epsilon}$ separately.
    Summing over all possible compositions in \eqref{eq:recursion-f},
    \begin{align*}
        \abs{\hat{\M}^{k, t}_\Pi[\indic_{\{\tau(U) \le \hat{\epsilon}\}}
         & G \circ S^{\hat{\epsilon}}(U^*)]}
        \le \norm{G}_\infty \hat{\M}_\Pi^{k,t}[\indic_{\{\tau(U) \le
                \hat{\epsilon}\}}]
        = N^{-\gamma(k-1)} \norm{G}_\infty \M_\Pi^{k,t}[\indic_{\{\tau(U) \le
        \epsilon \}}]                         \\
         & = N^{-\gamma(k-1)} \norm{G}_\infty
        \sum_{n=1}^k \int_0^{\epsilon} e^{-w(t-u)}
        \int_0^L r(y) h(y) \M^{n,u}_y[1] \M^{k-n,u}_y[1]
        \Pi(y) \diff y \,\diff u.
    \end{align*}
    Using Lemma~\ref{lem:momentBound} to bound the moments and
    \eqref{eq:boundIntegralPi} for the remaining term,
    \[
        \abs{\hat{\M}^{k, t}_\Pi[\indic_{\{\tau(U) \le \hat{\epsilon}\}}
                G \circ S^{\hat{\epsilon}}(U^*)]}
        \le
        C N^{-\gamma} \epsilon N^{\gamma-1}
        = C \frac{\epsilon}{N} = o_N(1).
    \]

    When $\tau(U) > \hat{\epsilon}$, the recursive formula
    \eqref{eq:branching:rn0-1} reads
    \begin{multline*}
        \hat{\M}^{k, t}_\Pi[\indic_{\{\tau(U) > \hat{\epsilon}\}}
            G \circ S^{\hat{\epsilon}}(U^*)]
        = \frac{1}{N^{\gamma(p-\alpha)}} \int_{0}^{t-\hep} f(t-s)  e^{-wNs}
        \E_\Pi\bigg[ r(\bar{\zeta}_{s}) h(\bar{\zeta}_{s})
        \\ \times
        \sum_{k=1}^{p-1} \mathbf{M}_{\bar{\zeta}_{s}}^{k,\eps}\left[
            \prod_{i=1}^{k} \hat{\mathbf{M}}^{c_i,t-s-\hep}_{X_i}[F_i] \right]
        \mathbf{M}_{\bar{\zeta}_{s}}^{p-k,\eps}\left[
        \prod_{i=1}^{p-k} \hat{\mathbf{M}}^{c_{i+k},t-s-\hep}_{X_i}[F_{i+k}] \right]
        \bigg]
       \diff s.
    \end{multline*}
    We divide the integration interval into three subintervals,
    $\{0<s\leq \t\}$, $\{ \t < s < t - \t \}$ and $\{ t - \t \leq s
        < t-\hep\}$ and refer to $\tilde{I}_1$, $\tilde{I}_2$ and
    $\tilde{I}_3$ for the corresponding integrals.

    We first control $\tilde{I}_1$. The bound in
    Corollary~\ref{cor:usefulBound} shows that
    \begin{align*}
        \tilde{I}_1 \leqslant CN^{-\gamma(p-\alpha)}
        \int_{0}^{\t}\E_\Pi\left[ r(\bar{\zeta}_{u}) h(\bar{\zeta}_{u})
        \sum_{k=1}^{p-1}  \mathbf{M}_{\bar{\zeta}_{u}}^{k,\vep}\left[ 1
    \right]  \mathbf{M}_{\bar{\zeta}_{u}}^{p-k,\vep}\left[ 1
    \right]\right] \diff u \leqslant C \t.
    \end{align*}
    A similar bound on $\tilde{I}_3$ follows from the same arguments.

    We now estimate $\tilde{I}_2$. Let $\mathbf{c}'$ be the constant
    in Lemma~\ref{lem:57} corresponding $a=1/(2\beta)$. Since
    $\epsilon = L / 2\beta$, the moment bound in
    Lemma~\ref{lem:momentBound} yields that
    \[
        \frac{1}{N^{\gamma(k-1)}}
        \mathbf{M}^{k,\epsilon}_y\left[
            \indic_{\{ \exists i \in [k] : X_i \ge \mathbf{c}' L \}}
            \prod_{i=1}^k \hat{\M}^{c_i,s}_{X_i}[F_i]
            \right]
        \le
        \frac{C'}{N^{\gamma(k-1)}}
        \mathbf{M}^{k,\epsilon}_y\left[
            \indic_{\{ \exists i \in [k] : X_i \ge \mathbf{c}' L \}}
            \right]
        \xrightarrow[N \to \infty]{}  0,
    \]
    uniformly in $y\in(0,L)$ and $s \in [\t, t]$. Moreover, by
    Lemma~\ref{prop:spineEquilibriate} $\hat{\mathbf{M}}^{c_i,s}_z[F_i]
        \to \hat{\mathbf{M}}^{c_i,s}_\Pi[F_i]$ as $N \to \infty$ uniformly
    for $z \in [0, \mathbf{c}'L]$ and $s \in [\t, t]$. Combining these two
    estimates together
    \begin{equation} \label{eq:lemApprox1}
        \frac{1}{N^{\gamma(k-1)}}
        \abs*{
        \mathbf{M}^{k, \epsilon}_y \left[
            \prod_{i=1}^k \hat{\M}^{c_i,s}_{X_i}[F_i]
            \right]
        -
        \prod_{i=1}^k \hat{\M}^{c_i,s}_\Pi[F_i]
        \mathbf{M}^{k, \epsilon}_y[1]
        }
        \xrightarrow[N \to \infty]{}  0,
    \end{equation}
    uniformly in $y \in (0, L)$ and $s \in [\t, t]$. Furthermore,
    Lemma~\ref{prop:spineEquilibriate} allows us to replace each
    $\hat{\M}^{c_i,s}_{\Pi}[F_i]$ by $\hat{\M}^{c_i, s-\hat{\epsilon}}_{\Pi}[F_i]$
    in the limit above. Another application of Corollary~\ref{cor:usefulBound}
    and \eqref{eq:lemApprox1} above yields that
    \begin{equation*}
        \lim_{N\to\infty} \tilde{I}_2 -
        \int_{\t}^{t-\t} \bigg( f(t-s) e^{-wNs}
        \prod_{i=1}^p \hat{\mathbf{M}}^{c_{i},t-s}_{\Pi}[F_{i}]
        \E_\Pi\left[
        \frac{r(\bar{\zeta}_{s}) h(\bar{\zeta}_{s})}{N^{\gamma(2-\alpha)}}
        \sum_{k=1}^{p-1}
        \frac{\mathbf{M}_{\bar{\zeta}_{s}}^{k,\eps}[1]}{N^{\gamma(k-1)}}
        \frac{\mathbf{M}_{\bar{\zeta}_{s}}^{p-k,\eps}[1]}{N^{\gamma(p-k-1)}}
        \right] \bigg)\diff s
        = 0,
    \end{equation*}
    uniformly in $x \in [0, \c L]$.

    Finally, as for $\tilde{I}_1$ and $\tilde{I}_3$, one can show that
    \begin{equation*}
        \frac{1}{N^{\gamma(p-\alpha)}}
        \int_{s \in [0,\t] \cup [t-\t, t-\t + \hat{\epsilon}]}
        f(t-s) e^{-wNs} \prod_{i=1}^p \hat{\mathbf{M}}^{c_{i},t-s}_{\Pi}[F_{i}]
        \E_\Pi\left[
        r(\bar{\zeta}_{s}) h(\bar{\zeta}_{s})
        \sum_{k=1}^{p-1}
        \mathbf{M}_{\bar{\zeta}_{s}}^{k,\eps}[1]
        \mathbf{M}_{\bar{\zeta}_{s}}^{p-k,\eps}[1]
        \right]\diff s
    \end{equation*}
    tends to $0$ as $N \to \infty$. This concludes the proof of the lemma.
\end{proof}

\subsection{Step 2: reversing the process}

Next, we justify the approximation \eqref{eq:heuristic-1} from
Step~2 of the heuristics (Section~\ref{sec:step2}). Recall the notation
$(\la{\zeta}_t)_{t \ge 0}$ for the reversed spine process introduced in
\eqref{eq:reversed_SDE} and $\la{\M}^{k,t}_z[1]$ for the moments of the
reversed BBM in \eqref{eq:biasedReversedMoment}. Replacing the moments of
the original BBM by those of the reversed one requires that we compare
the eigenelements of the two processes.
The main result of this section is motivated by Proposition \ref{lem:approx:1}.

\begin{proposition} \label{lem:estimateBranchingMoment}
    For all $A\geq 1$ and $p \ge 2$
    \begin{equation*}
        \hat{m}_{p,(N,A)} = \frac{p!}{N^{(p-\alpha)\gamma}}
        \int_0^L r(y) h(y)
        \sum_{i=1}^{p-1} \mathbf{M}_{y}^{i,\eps}\left[ 1 \right]
        \mathbf{M}_{y}^{p-i,\eps}\left[ 1 \right] \Pi(y) \diff y
        \xrightarrow[N \to \infty]{}
        \int_{(0,\infty)} z^p \Lambda_{A}(\diff z)
    \end{equation*}
    where $\hat{m}_{p}$ is defined as in Proposition~\ref{lem:approx:1}
    (or \eqref{eq:m_c} in the heuristics), $\Lambda$ is the measure
    introduced in Proposition~\ref{prop:limitReversed} and $\Lambda_{A}$
    the dilated measure defined in \eqref{eq:rescaled_lambda}.
\end{proposition}

In order to establish the result, we start with some technical lemmas.

\begin{lemma} \label{lem:relationReverseEigenelements}
    For any $A \ge 1$, uniformly for $z$ in compact sets
    \[
        \lim_{N \to \infty} \frac{h(L-z)}{AN^\gamma} = \la{h}(z),\qquad
        \lim_{N \to \infty} \frac{\Pi(L-z)}{(AN^\gamma)^{1-\alpha}} = \la{\Pi}(z).
    \]
\end{lemma}

\begin{proof}
    By Lemma~\ref{lem:spectral}~(ii) and by definition of
    $\beta$ in \eqref{eq:decay_v}, $\sqrt{2 \lambda} = \beta + O(1/N)$
    and so $L \sqrt{2 \lambda} = L \beta + o(1)$, as $N \to \infty$. By
    the explicit expression in \eqref{formula:v1},
    \[
        v(L-z) = \frac{\sinh(\sqrt{2\lambda} z )}{\sinh(\sqrt{2\lambda}(L-1))}
        \sim 2e^\beta \sinh(\beta z ) \exp( - \beta L ),
        \qquad \text{as $N \to \infty$,}
    \]
    uniformly for $z$ in compacts. Recall that, by \eqref{def:LA0},
    $AN^\gamma = e^{(\mu-\beta) L}$ and $(AN^\gamma)^{\alpha} = e^{(\mu+\beta)L}$.
    The result follows by replacing $h$ and $\Pi$ by their expressions
    \[
        \frac{h(L-z)}{AN^\gamma}
        = \frac{c_h v(L-z)e^{\mu(L-z)}}{e^{(\mu-\beta)L}}
        \to
        2 c_h e^{\beta} \sinh(\beta z) e^{-\mu z}
        = \la{h}(z), \qquad \text{as $N \to \infty$}
    \]
    and
    \begin{align*}
        \frac{\Pi(L-z)}{(AN^\gamma)^{1-\alpha}}
         & = \frac{h(L-z)}{AN^\gamma}
        \frac{c_{\tilde{h}} v(L-z)e^{-\mu(L-z)}}{(AN^\gamma)^{-\alpha}} \\
         & \to
        \la{h}(z) \cdot 2 c_{\tilde{h}} e^\beta \sinh(\beta z) e^{\mu z}
        \qquad \text{as $N \to \infty$}                                 \\
         & = \la{h}(z) \la{\widetilde{h}}(z) = \la{\Pi}(z).
        \qedhere
    \end{align*}
\end{proof}

The proof of Step~2 will also rely on approximating of the spine process
starting close to $L$ by the reversed spine process. This is the object
of the following lemma.

\begin{lemma} \label{lem:couplingSpines}
    Let $(\zeta_s)_{s \ge 0}$ be the spine process and define
    $\tilde{\zeta}_s = L - \zeta_s$ as in~\eqref{def:spine_fromL}.
    The reversed process $(\la{\zeta_s})_{s \ge 0}$ can be coupled to
    $(\tilde \zeta_s)_{s \ge 0}$ such that $\la{\zeta_0} = \tilde \zeta_0$ and
    \[
        \forall s \ge 0,\quad \forall \eta>0,\qquad
        \lim_{N \to \infty} \P_{L-z}( \abs{\tilde{\zeta}_s - \la{\zeta_s}} \ge \eta )
        = 0,
    \]
    uniformly for $z$ in compacts sets. 
\end{lemma}

\begin{proof}
    By~\eqref{eq:SDE_drift}, $\tilde{\zeta}$ is the solution to
    \[
        \diff \tilde{\zeta}_t = -F\big( \sqrt{2\lambda}, \tilde{\zeta}_t
        \big) \diff t +  \diff  B_t,
    \]
    as long as $\tilde{\zeta}_t \le L - 1$, for some Brownian motion
    $(B_t)_{t \ge 0}$. We couple $\la{\zeta}$ to $\tilde{\zeta}$ by
    letting $\la{\zeta}$ be the (unique strong) solution to
    \eqref{eq:reversed_SDE} using the same Brownian path $(B_t)_{t \ge 0}$
    as for $\tilde{\zeta}$ and starting from $\la{\zeta}_0 =
        \tilde{\zeta}_0$. By~\eqref{eq:dev_F}, on the event that
    $\tilde{\zeta}_s \le L-1$ for all $s \in (0,t)$, we have
    \[
        \forall 0\leq s \leq t,\qquad
        \abs{\la{\zeta_s} - \tilde{\zeta}_s}
        \leq
        \int_0^s \abs[\big]{F\big(\beta, \la{\zeta_u} \big)
            - F\big( \sqrt{2\lambda}, \tilde{\zeta}_u \big)}
        \diff u
        \le
        \abs{\beta - \sqrt{2\lambda}} s +
        \int_0^s \frac{\left| \tilde{\zeta}_u - \la{\zeta_u}\right|}{\la{\zeta_u}
            \cdot \tilde{\zeta}_u} \diff u.
    \]
    Gronwall's inequality then yields that, again on the event that
    $(\tilde{\zeta}_s)_s$ does not hit $L-1$ before time $t$,
    \begin{equation} \label{eq:couplingSpines-1}
        \abs{\la{\zeta_t} - \tilde{\zeta}_t}
        \le
        \abs{\beta - \sqrt{2\lambda}} t
        \exp\left(\int_0^t \frac{\diff s}{\la{\zeta_s}\cdot
            \tilde{\zeta}_s}\right).
    \end{equation}
    Recall that
    \[
        \lim_{N \to \infty}
        \P_{L-z}\big( \sup_{s \in (0, \vep)} \tilde{\zeta}_s \ge L-1 \big) = 0
    \]
    uniformly for $z$ in compact sets by Lemma~\ref{lem:boundSpineProcess}.
    The result now follows from a simple manipulation: take any sequence
    $z_N \to z$. Given any subsequence, one can extract a further
    subsequence such that almost surely $\sup_{(0, \vep)} \tilde{\zeta}_s <
        L-1$ for $N$ large enough. Then, because $\sqrt{2\lambda} \to
        \beta$ as $N \to \infty$ by Lemma~\ref{lem:spectral}~(i),
    $\tilde{\zeta_s} \to \la{\zeta_s}$ almost surely along that
    subsequence by \eqref{eq:couplingSpines-1}. This shows that
    \[
        \lim_{N \to \infty} \P_{L-z_N}\big(
        \abs{\tilde{\zeta}_s - \la{\zeta_s}} \ge \eta
        \big) = 0,\qquad \eta > 0,
    \]
    and the claim follows since $(z_N)_N$ is an arbitrary converging
    subsequence.
\end{proof}

\begin{lemma} \label{lem:comparisonReversed}
    For any $k \ge 1$ and $A \ge 1$,
    \begin{equation} \label{eq:convMomentRelax}
        \lim_{N \to \infty}
        \frac{\M^{k,\epsilon}_{L-z}[1]}{(AN^{\gamma})^{(k-1)}}
        = \la{\M}_z^{k,\infty}[1],
    \end{equation}
    uniformly for $z$ in compact sets.
\end{lemma}

\begin{proof}
    We prove by induction the slightly stronger claim that, for any fixed
    $t \ge 0$, \eqref{eq:convMomentRelax} holds when $\epsilon$ is replaced by
    $\epsilon - t$ (for $N$ large enough). For $k=1$, by definition of
    $\la{\M}^{1,s}_z[1]$ and by the first point of
    Proposition~\ref{def:planar-moments},
    \[
        \la{\M}^{1,\epsilon-t}_{z}[1]= 1
        \quad \text{and}\quad
        \M^{1,\epsilon-t}_{L-z}[1] = e^{-w(\epsilon-t)}= 1-o(1),
    \]
    where we used that $w\vep=o(1)$, see Lemma~\ref{lem:spectral}~(ii).

    Let us now assume that the claim holds for all $i \le k$ for some $k
        \ge 1$. We know from~\eqref{eq:bound_ind} that
    \begin{equation} \label{eq:rec:relax}
        \frac{1}{(AN^\gamma)^k} \M_{L-z}^{k+1,\vep-t}[1]
        = \int_0^{\vep-t} \E_{L-z}\bigg[
        e^{-ws} r(L-\tilde{\zeta}_s) \frac{h(L-\tilde{\zeta}_s)}{AN^{\gamma}}
        \sum_{n=1}^{k}
        \frac{\M^{k+1-n,\vep-s}_{L-\tilde{\zeta}_s}[1]}{(AN^\gamma)^{k-n}} \,
        \frac{\M^{n,\vep-s}_{L-\tilde{\zeta}_s}[1]}{(AN^{\gamma})^{n-1}}
        \bigg]\diff s,
    \end{equation}
    with $\tilde{\zeta}_s = L - \zeta_s$. We will let $N \to \infty$
    using the dominated convergence theorem.

    Consider a sequence $(z_N)_N$ converging to $z$, and let
    $\tilde{\zeta}$ and $\la{\zeta}$ be coupled as in
    Lemma~\ref{lem:couplingSpines} under $\P_{L-z_N}$. Under this coupling,
    by induction and by Lemma~\ref{lem:relationReverseEigenelements}, for
    any $s \ge 0$
    \begin{equation*}
        \lim_{N \to \infty}
        e^{-ws} r(L-\tilde{\zeta}_s) \frac{h(L-\tilde{\zeta}_s)}{AN^{\gamma}}
        \sum_{n=1}^{k}
        \frac{\M^{k+1-n,\epsilon-s}_{L-\tilde{\zeta}_s}[1]}{(AN^\gamma)^{k-n}} \,
        \frac{\M^{n,\epsilon-s}_{L-\tilde{\zeta}_s}[1]}{(AN^\gamma)^{n-1}}
        =
        \frac{1}{2}
        \la{h}(\la{\zeta_s}) \sum_{p=1}^{k}
        \la{\M}^{p,\infty}_{\la{\zeta_s}}[1] \,
        \la{\M}^{k+1-p,\infty}_{\la{\zeta_s}}[1]
    \end{equation*}
    in probability under $\P_{L-z_N}$. We now prove that this integrand is
    uniformly integrable (seen as a function of $s$ and of the randomness)
    by showing that it has uniformly bounded second moment. By the
    moment bound in Lemma~\ref{lem:momentBound} and the bound for $h$ in
    Corollary~\ref{lem:eigenBound}, \eqref{eq:easy_bound1}, we can find $C$ such that
    \[
        \forall z \in [0, L],\quad
        e^{-ws} r(L-z) \frac{h(L-z)}{AN^{\gamma}}
        \sum_{n=1}^{k}
        \frac{\M^{k+1-n,\epsilon-s}_{L-z}[1]}{(AN^\gamma)^{k-n}} \,
        \frac{\M^{n,\epsilon-s}_{L-z}[1]}{(AN^\gamma)^{n-1}}
        \le \frac{C e^{(\mu-\beta)(L-z)}}{A N^{\gamma}}
        = C e^{-(\mu-\beta) z}.
    \]
    Since $\mu > \beta$, Lemma~\ref{lem:integrabilityReversedSpine}
    entails that
    \[
        \limsup_{N \to \infty}
        \int_0^{\epsilon} \E_{L-z_N}[ (e^{-(\mu-\beta)
        \tilde{\zeta}_s})^2 ]\diff s
        < \infty.
    \]
    By the dominated convergence theorem (applied to the double integral
    under $\diff s \otimes \P_{L-z_N}$), for any $t \in[0,\vep]$,
    \[
        \lim_{N \to \infty}
        \frac{\M_{L-z_N}^{k+1,\epsilon-t}[1]}{(AN^\gamma)^k}
        =
        \frac{1}{2} \int_0^\infty
        \la{\E}_z\Big[
        \la{h}(\la{\zeta_s}) \sum_{p=1}^{k}
        \la{\M}^{p,\infty}_{\la{\zeta_s}}[1] \,
        \la{\M}^{k+1-p,\infty}_{\la{\zeta_s}}[1]
        \Big]\diff s.
    \]
    By \eqref{eq:recursionReversed}, the right-hand side is
    $\la{\M}^{k+1,\infty}_z[1]$. Since $(z_N)_N$ is an arbitrary
    converging sequence, the convergence also holds uniformly on
    compacts.
\end{proof}

\begin{proof}[Proof of Proposition \ref{lem:estimateBranchingMoment}]
    Recall that for every $p\geq 2$,
    \[
        \int_{(0,\infty)} z^p \Lambda_A(\diff z) = A^{p-\alpha}
        \int_{(0,\infty)} z^p \Lambda(\diff z).
    \]
    Thus, by Proposition~\ref{prop:limitReversed} we only need to prove that, for all $A\geq 1$,
    \begin{multline*}
        \frac{1}{N^{\gamma(p-\alpha)}}
        \int_0^L r(y) h(y)
        \sum_{i=1}^{p-1}   \mathbf{M}_{y}^{i,\eps}\left[ 1 \right]
        \mathbf{M}_{y}^{p-i,\eps}\left[ 1 \right] \Pi(y) \diff y \\
        \xrightarrow[N \to \infty]{}
        \frac{A^{p-\alpha}}{2}
        \int_0^\infty \la{h}(z) \la{\Pi}(z) \sum_{i=1}^{p-1}
        \la{\mathbf{M}}_{z}^{i,\infty}\left[ 1 \right]
        \la{\mathbf{M}}_{z}^{p-i,\infty}\left[ 1 \right] \diff z.
    \end{multline*}
    First,
    \begin{multline}
        \label{eq:estBranchingMoment1}
        \frac{1}{(AN^\gamma)^{p-\alpha}}
        \int_0^L r(y) h(y)
        \sum_{k=1}^{p-1}
        \mathbf{M}_{y}^{k,\eps}\left[ 1 \right]
        \mathbf{M}_{y}^{p-k,\eps}\left[ 1 \right] \Pi(y) \diff y
        \\
        =
        \int_0^{L} \frac{ r(L-z)h(L-z)\Pi(L-z)}{(AN^\gamma)^{2-\alpha}}
        \sum_{k=1}^{p-1}
        \frac{\mathbf{M}_{L-z}^{k,\eps}\left[ 1 \right]}{(AN^\gamma)^{k-1}}
        \frac{\mathbf{M}_{L-z}^{p-k,\eps}\left[ 1 \right]}{(AN^\gamma)^{p-k-1}}
        \diff z.
    \end{multline}
    By Lemma~\ref{lem:comparisonReversed},
    \[
        \lim_{N \to \infty}
        \sum_{k=1}^{p-1}
        \frac{\mathbf{M}_{L-z}^{k,\eps}\left[ 1 \right]}{(AN^\gamma)^{k-1}}
        \frac{\mathbf{M}_{L-z}^{p-k,\eps}\left[ 1 \right]}{(AN^\gamma)^{p-k-1}}
        =
        \sum_{k=1}^{p-1}
        \la{\mathbf{M}}_{z}^{k,\infty}\left[ 1 \right]
        \la{\mathbf{M}}_{z}^{p-k,\infty}\left[ 1 \right]
    \]
    uniformly for $z$ in compact sets. Moreover, by
    Lemma~\ref{lem:relationReverseEigenelements}
    \[
        \lim_{N \to \infty} \frac{r(L-z) h(L-z) \Pi(L-z)}{(AN^\gamma)^{2-\alpha}}
        = \frac{1}{2} \la{h}(z) \la{\Pi}(z)
    \]
    uniformly on compacts. Finally, the moment bound in
    Lemma~\ref{lem:momentBound} and the bound for $h \Pi$ in
    Corollary~\ref{lem:eigenBound} show that
    \[
        \forall z \in [0,L],\quad
        \frac{\M^{k,\epsilon}_{L-z}[1]\M^{p-k,\epsilon}_{L-z}[1]}{(AN^\gamma)^{p-2}}
        \frac{r(L-z) h(L-z) \Pi(L-z)}{(AN^\gamma)^{2-\alpha}}
        \le C \frac{e^{(3\beta-\mu)(L-z)}}{(AN^{\gamma})^{(2-\alpha)}}
        = C e^{-(\mu-3\beta)z}
    \]
    for some $C > 0$. Since $\mu > 3\beta$ under \eqref{Hwp}, the
    result follows by letting $N \to \infty$ in \eqref{eq:estBranchingMoment1}
    using the dominated convergence theorem.
\end{proof}

\subsection{Completing the proof}

We are now ready to prove Theorem~\ref{th:cv_mmm}.

\begin{proof}[Proof of Theorem~\ref{th:cv_mmm}]
    We start by checking the first condition of Theorem~\ref{thm:main-cv}
    by induction on $k$. Namely, that for every  a continuous bounded $F
        \colon \mathbb{U}^*_k \to \R$, we have
    \begin{equation} \label{s}
        \forall k \geq 1, t \ge 0, \quad
        \lim_{N \to \infty} \hat{\mathbf{M}}^{k,t}_{x}[F] =
        \hat{\mathcal{M}}^{k,t}_{A} \otimes
        \Pi_\infty^{\otimes k}[F],
    \end{equation}
    where $\hat{\mathcal{M}}^{k,t}_{A}$ is the $k$-th moment of the
    $\psi_{A}$-mm space with
    \[
        \psi_{A}(\theta) =
        \frac{c_{w}}{A^{\alpha-1}}\theta +\int_{(0,\infty)}
        \big(e^{\theta z}-1-\theta z\big) \Lambda_{A}(\diff z).
    \]
    See Definition~\ref{def:planarMoments} for the recursive definition of
    $\hat{\mathcal{M}}^{k,t}_{A}$. In order for the recursion to work, we
    will actually need to work more and show that the convergence is
    uniform for $x$ on the interval $[0,{\bf c}L]$, for any $\mathbf{c} \in
        (0,1)$.

    \bigskip

    For $k = 1$, $\mathbb{U}_1$ is made of a single element (the null
    distance on a singleton), so that we can identify $\mathbb{U}^*_1$ with
    the set of marks, $\mathbb{U}^*_1 \cong \R_+$. If $F \colon \R_+ \to \R$
    if continuous bounded, Proposition~\ref{def:planar-moments} shows that
    \[
        \hat{\M}^{1,t}_x[F] = e^{-wtN} \E_x\big[ F(\zeta_{Nt}) \big].
    \]
    According to the definition of the moments $\hat{\mathcal{M}}^{k,t}_{A}$
    given in Definition~\ref{def:planarMoments}, the convergence of the moments
    $k=1$ boils down to proving that
    \[
        \lim_{N\to\infty} \ \hat{\M}^{1,t}_x[F] \ = \
        e^{-c_w t / A^{\alpha-1}} \int_0^\infty F(y) \Pi_\infty(y) \diff y.
    \]
    The result is now a straightforward consequence of the convergence of the
    spine process derived in Proposition~\ref{prop:hk} and of
    Lemma~\ref{lem:spectral}~(ii), which shows that $wN \to c_{w} /
    A^{\alpha-1}$ as $N \to \infty$.

    Let us now assume that the result holds up until rank $k-1$ for some
    $k \ge 2$. We first argue that the limit also holds if the process starts
    from the stationary measure $\Pi$ instead of a fixed $x$. We know
    from Lemma~\ref{lem:spectral}~(iii) that $\Pi(y) \to \Pi_\infty(y)$
    pointwise as $N \to \infty$. This convergence also holds in $L^1(\R_+)$
    by Scheffé's lemma. Here we used that, although $\Pi$ is not a
    probability measure under our renormalisation, its total mass
    converges to $1$ by Lemma~\ref{lem:spectral}~(iii).

    Fix $n < k$ and a continuous bounded $\bar{F} \colon \mathbb{U}^*_n
    \to \R$. By our induction assumption, $\hat{\M}^{n,s}_y[\bar{F}]$
    converges uniformly for $y \in [0, \mathbf{c} L]$. Further, since the
    integrand in the following equation is uniformly bounded (by 
    Lemma~\ref{lem:momentBound}),
    \begin{align}
        \hat{\M}^{n, s}_\Pi[\bar F]
         & = \int_0^{\mathbf{c}L} \Pi(y) \hat{\M}^{n,s}_y[\bar F] \diff y
        + \int_{\mathbf{c}L}^L \Pi(y) \hat{\M}^{n,s}_y[\bar F] \diff y \nonumber \\
         & \xrightarrow[N \to \infty]{}
        \int_0^\infty \Pi_\infty(y)
        \left(
        \hat{\mathcal{M}}_{A}^{n,s}\otimes (\Pi_\infty)^{\otimes n}
        \right)[\bar F] \diff y
        = \left(
        \hat{\mathcal{M}}_{A}^{n,s}\otimes (\Pi_\infty)^{\otimes n}
        \right)[\bar F].
        \label{eq:statInitCond}
    \end{align}
    We now consider a functional $G \colon \mathbb{U}^*_k \to \R$ of the form
    \[
        \forall U^* \in \mathbb{U}^*_k,\quad
        G(U^*) = \indic_{\{ c(U) = c\}}
        f(\tau(U)) \prod_{i=1}^{\abs{c}} F_i(U^*_i).
    \]
    Combining Step~1 (Proposition~\ref{lem:approx:1}) with Step~2
    (Proposition~\ref{lem:estimateBranchingMoment}) of the heuristics
    and recalling that $wN \to c_{w} / A^{\alpha-1}$ as above, we deduce that
    \begin{equation} \label{eq:limMoment1}
        \lim_{N \to \infty}
        \hat{\M}^{k,t}_x[G \circ S^{\hat{\epsilon}}]
        - \frac{1}{\abs{c}!} \int_{(0,\infty)} z^{\abs{c}} \Lambda_{A}(\diff z)
        \times
        \int_0^t f(t-s) e^{-c_{w} s / A^{\alpha-1}}
        \prod_{i=1}^{\abs{c}} \hat{\M}_\Pi^{c_i, t-s}[F_i] \diff s \ = \ 0,
    \end{equation}
    uniformly in $x \in [0, \mathbf{c} L]$. By the dominated convergence
    theorem applied to integral against time using \eqref{eq:statInitCond}
    (and the bound on the integrand provided by Lemma~\ref{lem:momentBound}),
    this finally yields that
    \begin{equation} \label{eq:limMoment2}
        \lim_{N \to \infty} \hat{\M}^{k,t}_x[G^{\hat{\epsilon}}]
        = \frac{1}{\abs{c}!} \int_{(0,\infty)} z^{\abs{c}}
        \Lambda_{A}(\diff z)
        \int_0^t f(t-s) e^{-c_{w} s/A^{\alpha-1}}
        \prod_{i=1}^{\abs{c}}
        \big(
        \hat{\mathcal{M}}^{c_i,t-s}_{A} \otimes \Pi_\infty^{\otimes c_i}
        \big)[F_i]\diff s,
    \end{equation}
    uniformly in $x \in [0, \mathbf{c}L]$.
    Finally, the recursive relation on moments of the $\psi_{A}$-mm space in
    Definition~\ref{def:planarMoments} yields
    \begin{equation}
        \label{eq:limMoment3}
        \lim_{N \to \infty} \hat{\M}^{k,t}_x[G\circ S^{\hat{\epsilon}}] \ = \
        \hat{\mathcal{M}}^{k,t}_{A} \otimes \Pi_\infty^{\otimes k}[G]
    \end{equation}
    uniformly for $x \in [0, \mathbf{c}L]$.

    We will complete the proof of our claim by using a straightforward
    adaptation of Lemma~\ref{lem:product2Tree} to marks. Provided that we can
    check point~(ii) of that lemma, \eqref{eq:limMoment3} entails that the
    pushforward of $\hat{\M}^{k,t}_x$ by $S^{\hat{\epsilon}}$ converges to
    $\hat{\mathcal{M}}^{k,t}_A \otimes \Pi^{\otimes k}_\infty$. To see why
    point~(ii) holds, note that by Portmanteau's theorem
    \begin{equation} \label{eq:limMoment4}
        \limsup_{N \to \infty}
        \hat{\M}^{k,t}_x\big( \big\{
        \tau(C^{\hat{\epsilon}}(U)) -
        \max_i \tau\big( U_i\big( C^{\hat{\epsilon}}(U) \big) \big)
        < \eta
        \big\} \big)
        \le
        \hat{\mathcal{M}}^{k,t}_A\big( \big\{
        \tau(U) - \max_i \tau(U_i(U))
        < \eta
        \big\} \big).
    \end{equation}
    By Definition~\ref{def:planarMoments}, the pushforward of
    $\hat{\mathcal{M}}^{k,t}_A$ by the map $\tau$ has a density with respect
    to the Lebesgue measure, from which it is not hard to deduce that
    $\hat{\mathcal{M}}^{k,t}_A(\{\tau(U) = \max_i \tau(U_i(U)) \}) = 0$. This
    entails that the right-hand side of \eqref{eq:limMoment4} vanishes as
    $\eta \to 0$. Finally, note that by definition of $S^{\hat{\epsilon}}$ in
    \eqref{eq:mergedMap},
    \[
        \forall U \in \mathbb{U}_k,\quad
        \max_{i,j \le k} \abs{U_{ij} - S_{ij}^{\hat{\epsilon}}(U)}
        \le \hat{\epsilon}.
    \]
    This readily entails that, for instance, the bounded Lipschitz distance
    between $\hat{\M}^{k,t}_x$ and its pushforward by $S^{\hat{\epsilon}}$
    vanishes since $\hat{\epsilon} \to 0$ as $N \to \infty$. Therefore
    $\hat{\M}^{k,t}_x$ also converges to $\hat{\mathcal{M}}_A^{k,t}
        \otimes \Pi_\infty^{\otimes k}$. This completes the induction and proves
    that the first condition of Theorem~\ref{thm:main-cv} holds for the BBM.
    By Lemma~\ref{prop:spineEquilibriate}, this convergence actually holds 
    uniformly for $x \in [0, \mathbf{c} L]$ as required for the induction
    to work.

    It remains to check the second condition of Theorem~\ref{thm:main-cv}.
    That $h$ converges to $h_\infty$ uniformly on compacts follows by
    Lemma~\ref{lem:spectral}~(iii) and the definition of $h$ in terms of $v$.
    The bound on $h$ was established in \eqref{eq:easy_bound1}. Finally, by
    Proposition~\ref{def:planar-moments}~(i)
    \[
        \hat{\mathbf{M}}^{1,t}_x\Big[ \frac{1}{h} \Big]
        =
        e^{wtN}
        \int_0^L \frac{q_{tN}(x,y)}{h(y)} \diff y
        \sim e^{-c_w t / A^{\alpha-1}} \int_0^L \frac{\Pi(y)}{h(y)} \diff y
        \sim e^{-c_w t / A^{\alpha-1}} \int_0^L \tilde{h}(y) \diff y
    \]
    as $N \to \infty$ by Lemma~\ref{lem:spectral}~(ii) and
    Proposition~\ref{prop:hk} since
    $tN > cL$ for $N$ large enough. This shows that $\hat{\M}^{1,t}[1/h]$
    converges to $e^{-c_w t / A^{\alpha-1}}$ as $N \to \infty$, which is
    equal to $\hat{\mathcal{M}}^{1,t}_A[1]$ by
    Definition~\ref{def:planar-moments}.
\end{proof}

\section{Appendix}

This section collects various results of technical nature that are needed
in the proofs.

\begin{lemma} \label{lem:boundLipschitz}
    Let $(X,d,\mu)$ and $(X',d',\mu')$ be two mm-spaces, and let $\phi
        \colon \R_+^{k\times k} \to \R_+$ be Lipschitz in the sense that
    \[
        \forall (d_{ij}, d'_{ij})_{i,j \le k},\quad
        \abs{\phi\big( (d_{ij})_{i,j} \big) - \phi\big( (d'_{ij})_{i,j} \big)}
        \le L \sum_{i,j=1}^k \abs{d_{ij} - d'_{ij}}.
    \]
    Then, if $\Phi$ is the polynomial associated to $\phi$ we have
    \[
        \abs{\Phi(X,d,\mu) - \Phi(X',d',\mu')} \le (Lk^2+k) \underline{\Box}_1(X,X').
    \]
\end{lemma}

\begin{proof}
    Let $\epsilon > \underline{\Box}_1(X,X')$. By [Janson,
            Definition~3.1] we can find two measure preserving maps $\psi \colon
        [0,1] \to X$ and $\psi' \colon [0,1] \to X'$ and a set $W_\epsilon$
    such that $\Leb(W_\epsilon) \le \epsilon$ and
    \begin{equation*}
        \forall x,y \in [0, 1] \setminus W_\epsilon,\quad
        \abs{d(\psi(x,y)) - d'(\psi'(x,y))} \le \epsilon.
    \end{equation*}
    We now have that
    \begin{align*}
        | \Phi( & X,d,\mu) - \Phi(X',d',\mu') |                                        \\
                & = \abs{ \int_{[0,1]^k} \phi\big( (d(\psi(x_i,x_j)))_{i,j\le k} \big)
            - \phi\big( (d'(\psi'(x_i,x_j)))_{i,j\le k} \big) \diff x_1
        \dots \diff x_k }                                                              \\
                & \le 1-\Leb(\bar{W}_\epsilon^k) + L \int_{\bar{W}_\epsilon^k}
        \sum_{i,j=1}^k \abs{d(\psi(x_i,x_j)) - d'(\psi'(x_i,x_j))} \big)
        \diff x_1 \dots \diff x_k                                                      \\
                & \le 1-(1-\epsilon)^k + Lk^2 \epsilon^k \le (Lk^2+k) \epsilon.
    \end{align*}
    and the result follows by letting $\epsilon \to
        \underline{\Box}_1(X,X')$.
\end{proof}

\begin{lemma} \label{cor:boundShapes}
    Let $(a_k)_{k \ge 1}$ be a non-negative
    sequence and $C \ge 1$ be such that
    \[
        \forall k \ge 2,\quad a_k \le C \sum_{i=1}^{k-1} a_i a_{k-i}.
    \]
    Then, $a_k \le (a_1)^k C^{k-1} 4^k$, $k \ge 1$.
\end{lemma}

\begin{proof}
    The result follows by defining a well-chosen auxiliary sequence,
    namely $w_k = a_k / (C^{k-1} a_1^k)$. We have $w_1 = 1$ and
    \begin{align*}
        w_k
         & = a_k C^{-k+1} a_1^{-k}
        \le C^{-k+1} a_1^{-k} \sum_{i=1}^{k-1} C a_i a_{k-i}       \\
         & = C^{-k+2} \sum_{i=1}^{k-1} C^{i-1}w_i C^{k-i-1}w_{k-i}
        = \sum_{i=1}^{k-1} w_i w_{k-i}.
    \end{align*}
    A simple induction shows that $w_k \le B_k$ where $(B_k)_{k \ge 1}$
    is defined as
    \[
        B_1=1,\quad \forall k \ge 1\quad B_k = \sum_{i=1}^{k-1} B_i B_{k-i}.
    \]
    By noting that $B_k$ follows the same recursion as the number of
    rooted ordered binary trees with $k$ unlabeled leaves, we get the bound
    \[
        \forall k \ge 1,\quad a_k \le a_1^k C^{k-1} \frac{2^k (2k-3)!!}{k!}
    \]
    from which the result follows.
\end{proof}

Finally, let us show this version of Fatou's lemma for vague convergence.

\begin{lemma}[Vague Fatou's lemma] \label{lem:fatouVague}
    Let $(\nu_N)_{N \ge 1}$ be locally finite measures on $(0, \infty)$,
    and suppose that $\nu_N \to \nu$ vaguely on $(0, \infty]$. Then,
    \[
        \angle{\nu, x} \le \liminf_{N \to \infty} \angle{\nu_N, x}.
    \]
\end{lemma}

\begin{proof}
    Fix some $\delta > 0$ such that $\nu(\{ x : x=\delta \}) = 0$. By
    Fatou's lemma (see the version for weak convergence in
    \cite[Lemma~4.11]{kallenberg2002foundations}) we have that
    \[
        \angle{\nu, \indic_{\{ x \ge \delta \}} x}
        \le
        \liminf_{N \to \infty}
        \angle{\nu_N, \indic_{\{ x \ge \delta \}} x}
        \le
        \liminf_{N \to \infty}
        \angle{\nu_N, x}.
    \]
    The result is obtained by letting $\delta \to 0$.
\end{proof}

\begin{lemma}[Vague method of moments] \label{lem:vagueMethodMoments}
    Suppose that $(\nu_n)_{n \ge 0}$ is a sequence of measures on $\R_+$
    such that
    \[
        \forall k \ge 1,\quad \lim_{n \to \infty} \int_{\R_+} x^k
        \nu_n(\diff x) =
        m_k,
        \qquad \limsup_{k \to \infty} \frac{m_k^{1/k}}{k} < \infty.
    \]
    Then $(\nu_n)_{n \ge 1}$ converges vaguely as $n \to \infty$ to the
    unique measure $\nu$ such that $\int x^k \nu(\diff x) = m_k$.
\end{lemma}

\begin{proof}
    Consider the measure $\widetilde{\nu}_n(\diff x) = x \nu_n(\diff x)$.
    Clearly,
    \[
        \forall k \ge 0,\quad \lim_{n \to \infty}
        \int_{\R_+} x^k \widetilde{\nu}_n(\diff x) = m_{k+1},
        \qquad
        \limsup_{k \to \infty} \frac{m_{k+1}^{\frac{1}{k}}}{k}
        =
        \limsup_{k \to \infty} \big( 1 + \frac{1}{k} \big)
        m_{k+1}^{\frac{1}{k(k+1)}} \frac{m_{k+1}^{\frac{1}{k+1}}}{k+1}
        < \infty.
    \]
    Therefore, by the method of moments, there exists a measure
    $\widetilde{\nu}$ with moments $(m_{k+1})_{k \ge 0}$ and such that
    $(\widetilde{\nu}_n)_{n \ge 1}$ converges to $\widetilde{\nu}$ weakly.
    Define $\nu(\diff x) = (1/x) \widetilde{\nu}(\diff x)$. If $\phi \colon \R_+
    \to \R$ is continuous bounded, and $0$ is not in the support of $\phi$, then 
    $x \mapsto \phi(x)/x$ is continuous bounded and
    \[
        \int_{\R_+} \phi(x) \nu_n(\diff x)
        = \int_{\R_+} \frac{\phi(x)}{x} \widetilde{\nu}_n(\diff x)
        \longrightarrow
        \int_{\R_+} \frac{\phi(x)}{x} \widetilde{\nu}(\diff x)
        = \int_{\R_+} \phi(x) \nu(\diff x),
        \qquad \text{as $n \to \infty$.}
    \]
    In other words, $(\nu_n)_{n \ge 1}$ converges vaguely to $\nu$.
\end{proof}

\bibliographystyle{plain}
\bibliography{biblio_semipushed2.bib}

\end{document}